\documentclass[11pt,reqno]{amsart}
\usepackage[utf8]{inputenc}  
\usepackage[T1]{fontenc}
\usepackage{graphicx}
\usepackage{cite}
\usepackage{fixltx2e}
\usepackage{newunicodechar}

\usepackage{color}
\definecolor{MyLinkColor}{rgb}{0,0,0.4}
\numberwithin{equation}{section}

\newcommand{\im}{\mathop{\rm Im}\nolimits}

\newcommand{\re}{\mathop{\rm Re}\nolimits}

\newcommand{\sign}{\mathop{\rm sign}\nolimits}
\newcommand{\PV}{\mathop{\rm PV}\nolimits}

\newcommand{\0}{\Omega}
\newcommand{\e}{\varepsilon}

\newcommand{\p}{\partial}

\newcommand{\ov}{\overline}
\newcommand{\wh}{\widehat}
\newcommand{\G}{\Gamma}

\newcommand{\bA}{\mathbb{A}}

\newcommand{\bB}{\mathbb{B}}

\newcommand{\cO}{\mathcal{O}}

\newcommand{\kH}{\mathcal{H}}

\newcommand{\kL}{\mathcal{L}}

\newcommand{\C}{\mathbb{C}}

\newcommand{\E}{\mathbb{E}}
\newcommand{\R}{\mathbb{R}}
\newcommand{\s}{\mathbb S}
\newcommand{\N}{\mathbb{N}}
\newcommand{\Z}{\mathbb{Z}}

\DeclareMathOperator{\supp}{supp}
 
\newtheorem{thm}{Theorem}[section]
\newtheorem{prop}[thm]{Proposition}
\newtheorem{lemma}[thm]{Lemma}

\newtheorem{rem}[thm]{Remark}

\numberwithin{equation}{section}

\title[On some periodic Muskat problems]{Well-posedness and stability results for some  periodic Muskat problems}

\author[Bogdan--Vasile Matioc]{Bogdan--Vasile Matioc}
\address{Fakult\"at f\"ur Mathematik, Universit\"at Regensburg,   93053 Regensburg, Deutschland.}
\email{bogdan.matioc@ur.de}

\subjclass[2010]{35B35; 35B65; 35K55;  35Q35; 42B20}
\keywords{Muskat problem; Singular integral; Well-posedness; Parabolic smoothing; Stability}

\usepackage[colorlinks=true,linkcolor=MyLinkColor,citecolor=MyLinkColor]{hyperref} %dvips

\begin{document}

\begin{abstract}
We study the two-dimensional Muskat problem  in a horizontally periodic setting and for fluids with arbitrary densities and viscosities.
We show that in the presence of surface tension effects the Muskat problem is a quasilinear parabolic problem which is well-posed 
in the Sobolev space $H^r(\mathbb{S})$ for each $r\in(2,3)$.
When neglecting surface tension effects, the Muskat problem is a fully nonlinear evolution equation and of parabolic type in the regime where the Rayleigh-Taylor condition is satisfied.
We then establish the well-posedness of the Muskat problem   in the open subset of $H^2(\mathbb{S})$ defined by the  Rayleigh-Taylor condition. 
Besides, we identify  all equilibrium solutions and study the stability properties of trivial and of small finger-shaped equilibria.
Also other qualitative properties of solutions such as parabolic smoothing, blow-up behavior, and criteria for global existence are outlined. 
\end{abstract}

\maketitle
\tableofcontents

 %%%%%%%%%%%%%%%%%%%%%%%%%%%%%%%%%%%%%%%%%%%%%%%%%%%%%%%%%%%%%%%%%%%
%%%%%%%%%%%%%%%%%%%%%%%%%%%%%%%%%%%%%%%%%%%%%%%%%%%%%%%%%%%%%%%%%%%%
%%%%%%%%%%%%%%%%%%%%%%%%%%%%%%%%%%%%%%%%%%%%%%%%%%%%%%%%%%%%%%%%%%%%
%%%%%%%%%%%%%%%%%%%%%%%%%%%%%%%%%%%%%%%%%%%%%%%%%%%%%%%%%%%%%%%%%%%
%%%%%%%%%%%%%%%%%%%%%%%%%%%%%%%%%%%%%%%%%%%%%%%%%%%%%%%%%%%%%%%%%%%%
%%%%%%%%%%%%%%%%%%%%%%%%%%%%%%%%%%%%%%%%%%%%%%%%%%%%%%%%%%%%%%%%%%%%

 %%%%%%%%%%%%%%%%%%%%%%%%%%%%%%%%%%%%%%%%%%%%%%%%%%%%%%%%%%%%%%%%%%%
%%%%%%%%%%%%%%%%%%%%%%%%%%%%%%%%%%%%%%%%%%%%%%%%%%%%%%%%%%%%%%%%%%%%
%%%%%%%%%%%%%%%%%%%%%%%%%%%%%%%%%%%%%%%%%%%%%%%%%%%%%%%%%%%%%%%%%%%%
%%%%%%%%%%%%%%%%%%%%%%%%%%%%%%%%%%%%%%%%%%%%%%%%%%%%%%%%%%%%%%%%%%%
%%%%%%%%%%%%%%%%%%%%%%%%%%%%%%%%%%%%%%%%%%%%%%%%%%%%%%%%%%%%%%%%%%%%
%%%%%%%%%%%%%%%%%%%%%%%%%%%%%%%%%%%%%%%%%%%%%%%%%%%%%%%%%%%%%%%%%%%%
\section{Introduction and the main results}\label{Sec:1}
 %%%%%%%%%%%%%%%%%%%%%%%%%%%%%%%%%%%%%%%%%%%%%%%%%%%%%%%%%%%%%%%%%%%
%%%%%%%%%%%%%%%%%%%%%%%%%%%%%%%%%%%%%%%%%%%%%%%%%%%%%%%%%%%%%%%%%%%%
%%%%%%%%%%%%%%%%%%%%%%%%%%%%%%%%%%%%%%%%%%%%%%%%%%%%%%%%%%%%%%%%%%%%
%%%%%%%%%%%%%%%%%%%%%%%%%%%%%%%%%%%%%%%%%%%%%%%%%%%%%%%%%%%%%%%%%%%
%%%%%%%%%%%%%%%%%%%%%%%%%%%%%%%%%%%%%%%%%%%%%%%%%%%%%%%%%%%%%%%%%%%%
%%%%%%%%%%%%%%%%%%%%%%%%%%%%%%%%%%%%%%%%%%%%%%%%%%%%%%%%%%%%%%%%%%%%

In this paper we study the coupled system of equations
\begin{subequations}\label{P}
\begin{equation}\label{EP}
\left\{ 
\begin{array}{rllll}
 \displaystyle\p_tf(t,x)\!\!\!\!&=&\!\!\!\!  \displaystyle\frac{1}{4\pi} \PV\int_{-\pi}^{\pi}\frac{f'(t,x) (1+t_{[s]}^2)(T_{[x,s]}f(t))+t_{[s]}[1-(T_{[x,s]}f(t))^2] }{t_{[s]}^2+(T_{[x,s]}f(t))^2 }\overline\omega(t,x-s)\, ds,\\[2ex]
\overline\omega(t,x)\!\!\!\!&=&\!\!\!\! \displaystyle\frac{2k}{\mu_-+\mu_+}(\sigma\kappa(f(t))-\Theta f(t))'(x)\\[2ex]
&&\!\!\!\!-\displaystyle \frac{a_\mu}{2\pi } \PV\int_{-\pi}^{\pi}\frac{f'(t,x)t_{[s]}[1-(T_{[x,s]}f(t))^2]- (1+t_{[s]}^2)T_{[x,s]}f(t)}{t_{[s]}^2+(T_{[x,s]}f(t))^2 }\overline\omega(t,x-s)\, ds
\end{array}\right.
\end{equation}
 for $t> 0 $\footnote{When $\sigma=0$ we require that the equations \eqref{EP} are satisfied also at $t=0$.} and $ x\in\R,$ which is supplemented by the initial condition 
 \begin{equation}\label{ICP}
 f(0)=f_0.
\end{equation}
\end{subequations}
The   evolution problem \eqref{P} describes the motion
 of the  boundary $[y=f(t,x)+tV]$ separating two  immiscible fluid layers with unbounded heights
 located   in a 
homogeneous porous medium with permeability 
$k\in(0,\infty)$ or in a vertical/horizontal Hele-Shaw cell.
It is assumed that the fluid system moves with  constant velocity $(0,V)$, $V\in\R$, that the motion is periodic with respect to the horizontal variable $x$ (with period $2\pi$), and  that the fluid  velocities are
asymptotically equal  to $(0,V)$ far away from the interface.
The unknowns of the evolution problem \eqref{P} are  the functions  $(f,\overline\omega)=(f,\overline\omega)(t,x)$.
We denote by $\s:=\R/2\pi\Z$ the unit circle, functions that depend on $x\in\s $ being $2\pi$-periodic with respect to the real variable $x$. 
To be concise, we have set
\[
\delta_{[x,s]}f:=f(x)-f(x-s),\qquad T_{[x,s]} f=\tanh\Big(\frac{\delta_{[x,s]}f}{2}\Big), \qquad t_{[s]}=\tan\Big(\frac{s}{2}\Big),
\]
and $(\,\cdot\,)'$ denotes  the spatial derivative  $\p_x.$ 
We further   denote by $g$ the Earth's gravity, $\sigma\in[0,\infty)$ is the surface tension coefficient, 
$\kappa(f(t))$ is the curvature of the free boundary $[y=f(t,x)+tV]$, while
$\mu_\pm$  and $\rho_\pm$  are  the viscosity and the density, respectively,  of the fluid $\pm$ which occupies the unbounded periodic strip
\[ \0_\pm^V(t):=\{(x,y)\in\R^2\,:\,\pm(f(t,x)+tV-y)<0\}.\]
Moreover,  the real constant  $\Theta$ and  the Atwood number $a_\mu$   that appear  in $\eqref{EP}_2$ are defined by
\begin{equation*}
\Theta:=g(\rho_--\rho_+)+\frac{\mu_--\mu_+}{k}V,\qquad a_\mu:=\frac{\mu_--\mu_+}{\mu_-+\mu_+}.
\end{equation*}
The integrals in \eqref{EP} are singular at $s=0$ and  $\PV$ denotes the Cauchy principle value.
In this paper we consider a general setting where
\[
\mu_--\mu_+,\, \rho_--\rho_+\in\R.
\]
The observation that  $|a_\mu|<1$ is crucial for our analysis.
This property enables us to  prove, for suitable  $f(t)$, that  the equation $\eqref{EP}_2$  
has a unique solution $\overline\omega(t)$ (which depends in an intricate way on $f(t)$, see Sections \ref{Sec:4} and \ref{Sec:5}).
Therefore we shall only refer to   $f$ as being the solution to \eqref{P}.   

The Muskat problem, in the classical formulation \eqref{PB*}, dates back to M. Muskat's paper \cite{Mu34} from 1934.
However, many of the mathematical studies on this topic are quite recent and they cover various physical scenarios and
mathematical aspects related to the original model proposed in \cite{Mu34},
cf. \cite{Y96, A04, BCG14, BCS16, BS16x, CGSV17, CCG11, CCG13b, CG07, CG10, CCFG13, CGO14, CGCSS14x, SCH04, BV14, 
EMM12a, M17x, M16x, EMW18, M16x, PS17, CCGS13, GB14,   CCFGL12, GG14, CGFL11, 
GS14,   PS16, PS16x, A14, FT03, HTY97, CL18x, DLL17, To17} (see also \cite{PS16xx, PSW18x} for some recent 
research on the  compressible analogue of the Muskat problem, the so-called Verigin problem).

Below we  discuss only the literature pertaining to \eqref{P}  and its nonperiodic counterpart. 
In the presence of surface tension effects, that is for $\sigma>0$,   \eqref{P}  has  been studied previously only in \cite{A14} 
where the author proved well-posedness of the problem in $H^r$ 
(with $r\geq 6$) in the more general setting of interfaces which are parameterized by curves, and the zero surface
tension limit of the problem has been also considered there.   
 The nonperiodic counterpart to \eqref{P} has been investigated in  \cite{M17x} where it was shown that the problem 
 is well-posed in $H^r(\s)$ for each $r\in(2,3)$ 
 by exploiting the fact that the problem is   quasilinear parabolic   together with the abstract 
 theory outlined in  \cite{Am93, Am95} for such problems.
 Additionally, it was shown in \cite{M17x} that the problem exhibits the effect of parabolic smoothing 
 and   criteria for global existence of solutions were found.
 We shown herein that the results in the nonperiodic framework \cite{M17x} hold also for \eqref{P}. 
 Besides, this paper  provides the full  picture of the set of equilibrium solutions to \eqref{P} -- which are 
 described by either flat of finger-shaped interfaces (similarly as in the bounded periodic case 
 \cite{EMM12a}) --  and  the stability properties of the flat equilibria and of small finger-shaped equilibria are studied in the  phase space $H^r(\s)$.
 For the latter purpose we use   a   quasilinear   principle of linearized stability
 derived recently in \cite{MW18x}.

 The first main result of this paper is the following theorem establishing the well-posedness of the Muskat problem with surface tension
  in the setting of classical solutions and for general initial data together with other qualitative properties of the solutions.
  \begin{thm} \label{MT:1}
Let  $\sigma>0$ and  $r\in(2,3) $  be given.  Then, the following hold:
 \begin{itemize}
 \item[$(i)$] {\em (Well-posedness in $H^r(\s)$)} The problem \eqref{P} possesses for each   $f_0\in H^r(\s)$ a unique maximal solution 
\[ \qquad f:=f(\cdot; f_0)\in {\rm C}([0,T_+(f_0)),H^r(\s))\cap {\rm C}((0,T_+(f_0)), H^3(\s))\cap {\rm C}^1((0,T_+(f_0)), L_2(\s)),\]
with $T_+(f_0)\in(0,\infty]$, and $[(t,f_0)\mapsto f(t;f_0)]$ defines a  semiflow on $H^r(\s)$. \\[-2ex]
  \item[$(ii)$] {\em (Global existence/blow-up criterion)} If   
\[
\sup_{[0,T_+(f_0))\cap[0,T]}\|f(t;f_0)\|_{H^r}<\infty \qquad\text{for all $T>0$,}
\]
then  $T_+(f_0)=\infty$.\\[-2ex]
  \item[$(iii)$] {\em (Parabolic smoothing)}
  The mapping  $[(t,x)\mapsto f(t,x)]:(0,T_+(f_0))\times\R\to\R$ is real-analytic. In particular, $f(t)$
  is a real-analytic function for all $t\in (0,T_+(f_0))$.
 \end{itemize}
\end{thm}\medskip

\begin{rem}\label{R:1}
\begin{itemize}
\item[$(i)$] 
Despite that  we deal with a third order problem in the setting of classical solutions,
the curvature of the initial data  in Theorem \ref{MT:1}
may be unbounded and/or discontinuous.
Moreover, it becomes instantaneously real-analytic under the flow. \\[-2ex]
\item[$(ii)$] Solutions which are not global have, in view of Theorem \ref{MT:1}, the property that
\[
\sup_{[0,T_+(f_0)) }\|f(t)\|_{H^s}=\infty \qquad\text{for each $s\in(2,3)$.}
\]
\end{itemize}
\end{rem}

Concerning the stability of equilibria, we also have to differentiate between the cases $\sigma=0$ and $\sigma>0$.
 Before doing this  we point out two features that are common for both cases. 
 Firstly, the integral mean of the solutions to \eqref{P} (found in Theorem  \ref{MT:1} or Theorem \ref{MT:2} below) is
 constant with respect to time, see Section \ref{Sec:6}.
 Secondly,   \eqref{P} has the following invariance property: If $f$ is a solution to \eqref{P}, then  the translation 
\begin{align}\label{HVT}
f_{a,c}(t,x):=f(t,x-a)+c,\qquad a,\, c\in\R, 
\end{align}
is also a solution to \eqref{P}.
For these two reasons, we shall  only address the stability issue for  equilibria to \eqref{P} which have zero integral 
mean and under perturbations with zero integral mean. 
  However, because of the invariance property \eqref{HVT}, our stability results can be transferred  also to other 
  equilibria, see Remark \ref{R:0}.
  
To set the stage, let 
 \[\wh H^r(\s):=\Big\{h\in H^r(\s)\,:\, \langle h\rangle:=\frac{1}{2\pi}\int_{-\pi}^\pi h\, dx=0\Big\}, \qquad \text{$r\geq0$}.\]
  In Theorem \ref{MT:4} below we describe the stability properties of some of the equilibria to \eqref{P} when $\sigma>0$.
   In this case the equilibrium solutions to \eqref{P} are either constant functions or  finger-shaped as in Figure \ref{Fig:1}.
 The  finger-shaped equilibria exist only in the regime where $\Theta<0$, that is  when  either the fluid located below 
 has a larger density or when the 
 less viscous fluid advances into the region occupied by the other one with sufficiently high speed $|V|$.
  Furthermore, these equilibria form  global bifurcation branches 
  (see Section \ref{Sec:6} for the complete picture of the set of equilibria).

 \begin{thm}\label{MT:4}
  Let $\sigma>0$ and $r\in(2,3)$ be given. The following hold:
  \begin{itemize}
  \item[$(i)$] If $\Theta+\sigma>0$, then   $f=0$ is exponentially stable. More precisely, 
given $$\omega\in(0,k(\sigma+\Theta)/(\mu_-+\mu_+)),$$  there exist constants $\delta>0$ and $M>0$, 
with the property that if $f_0\in \wh H^r(\s)$ satisfies 
 $\|f_0\|_{H^{r}}\leq\delta$, the solution to \eqref{P} exists globally  and
  \begin{align*}
   \|f(t;f_0)\|_{H^r } \leq Me^{-\omega t}\|f_0\|_{H^{r} }\qquad \text{for all $t\geq0.$}
  \end{align*}

    \item[$(ii)$] If $\Theta+\sigma<0$, then   $f=0$ is  unstable. More precisely, there exists   
   $R>0$ and a sequence $(f_{0,n})\subset \wh  H^{r}(\s)$ of initial data  such that: \\[-2ex]
\begin{itemize}
\item[$\bullet$] $f_{0,n}\to0 $ in $\wh  H^{r}(\s);$\\[-2ex]
\item[$\bullet$] There exists $t_n\in(0,T_+(f_{0,n}))$ with $\|f(t_n;f_{0,n})\|_{H^r}=R$.\\[-2ex]
\end{itemize}

 \item[$(iii)$] {\em (Instability of small finger shaped equilibria)}
 Given $1\leq \ell\in\N$, there exists a  real-analytic bifurcation curve 
 $(\lambda_\ell,f_\ell):(-\e_\ell,\e_\ell)\to (0,\infty)\times \wh H^3(\s)$, $\e_\ell>0$,  with
  \[
  \left\{
\begin{array}{lll}
\lambda_\ell(s)=\ell^2-\cfrac{3\ell^4}{8}s^2+O(s^4)  \quad \text{in $\R$,}\\[2ex]
f_\ell(s)= s\cos(\ell x)+O(s^2)  \quad \text{in $ \wh H^3(\s)$}
\end{array}
\right.\qquad\text{for $s\to0$,}
  \]
such that  $f_{\ell}(s)$ is an even equilibrium to \eqref{P} if $\Theta=-\sigma\lambda_\ell(s)$.
The finger-shaped  equilibrium $f_\ell(s)$, $0<|s|<\e_\ell$, is unstable if $\e_\ell$ is sufficiently small in the sense there exists   
   $R>0$ and a sequence $(f_{0,n})\subset \wh  H^{r}(\s)$   such that: \\[-2ex]
\begin{itemize}
\item[$\bullet$] $f_{0,n}\to f_\ell(s) $ in $\wh  H^{r}(\s);$\\[-2ex]
\item[$\bullet$] There exists $t_n\in(0,T_+(f_{0,n}))$ with $\|f(t_n;f_{0,n})-f_{\ell}(s)\|_{H^r}=R$.\\[-2ex]
\end{itemize} 
  \end{itemize}  
 \end{thm}

With respect to Theorem  \ref{MT:4}  we add the following remarks (Remark \ref{R:0} $(i)$  
remains valid for Theorem \ref{MT:3} below  as well).

\begin{rem}\label{R:0}
\begin{itemize}
\item[$(i)$] If $f$ is an even equilibrium to \eqref{P}, the translation $f(\cdot-a)+c$, $a,\, c\in\R $,
is also an equilibrium solution. In fact, all equilibria can be obtained in this way (see Section~\ref{Sec:6}).
The invariance property \eqref{HVT} shows that $f$ and  $f(\cdot-a)+c$ have the same stability properties.
\item[$(ii)$]  It is shown in  Theorem \ref{T:GBP} that the local curves $(\lambda_\ell,f_\ell)$ 
can be continued to global bifurcation branches consisting entirely of  
equilibrium solutions to \eqref{P}. The stability issue for the large finger-shaped equilibria remains an open problem.
\end{itemize}
\end{rem}

When  switching to the regime where $\sigma=0$, many aspects in the analysis of the Muskat problem with surface tension have to be reconsidered. 
A first major difference to the case $\sigma>0$ is due to the fact that the quasilinear character of the problem, which 
is mainly due to the curvature term, is lost (excepting for the very special case when $\mu_-=\mu_+$, cf. \cite{MM17x}), 
and the problem \eqref{P} is now fully nonlinear. The second important difference, is that  the problem is of parabolic type 
only when the Rayleigh-Taylor condition holds.
The Rayleigh-Taylor condition originates from \cite{ST58} and is expressed in terms 
of the pressures  $p_\pm$ associated of the fluid~$\pm$ as follows
\begin{align}\label{RT}
 \p_\nu p_-< \p_\nu p_+ \qquad\text{on $[y=f_0(x)]$,}
\end{align}
 with  $\nu$ denoting  the  unit normal to the curve $[y=f_0(x)]$ pointing towards   $\0_+^V(0)$ .
The first result in this setting is a local existence result in $H^k(\s)$, with $k\geq 3$,
established in \cite{CCG11} in the more general  setting of interfaces parametrized by periodic curves 
(for initial data such that the Rayleigh-Taylor conditions holds).
The particular case of fluids with equal densities has been in fact investigated  previously in \cite{SCH04} and the authors have shown
the existence of global solutions for small data.
The methods from \cite{CCG11} have been then generalized in \cite{CCG13b}  to the three-dimensional case,
the analysis emerging in a local existence result in $H^k$ with $k\geq4$.
More recently in \cite{BCS16} the authors have established   global existence and uniqueness  of solutions to \eqref{P}  for small 
  data in $H^2(\s)$ together with some exponential decay estimates in  $H^r$-norms with $r\in[0,2)$.
For the nonperiodic Muskat problem with $\sigma=0$ it is moreover shown  in \cite{BCS16} there exist unique local solutions
for initial data in $H^2(\R)$ which are small in the weaker  $H^{3/2+\e}$-norm with $\e\in(0,1)$ arbitrarily small.
The latter smallness size condition on the data was dropped out in  \cite{M17x} where it is shown that the nonperiodic Muskat  
possesses for   initial data in $H^2(\R)$ 
that satisfy the Rayleigh-Taylor condition  a unique local solution  and that the solution
depends continuously on the   data. 
Lastly, we mention the paper \cite{CCPS17} where the existence and uniqueness of a weaker notion of solutions is established for the nonperiodic Muskat problem with
initial data in   critical spaces, together with some algebraic decay of the global solutions.
In this paper we first generalize the methods from the nonperiodic setting \cite{M17x} to prove the well-posedness of  
\eqref{P} for general initial data in $H^2(\s)$ and instantaneous parabolic smoothing for solutions which satisfy an additional bound.
Before presenting our result, we point out that if  $\Theta=0$, then \eqref{P} has only constant solutions for each
$f_0\in H^r(\R)$, with $r>3/2$,   as Theorem \ref{T:I1} shows that in this case $\overline\omega=0$ 
is the only solution to $\eqref{EP}_2$  that lies in $\wh L_2(\s).$
When $\Theta \neq 0$, the situation is much more complex. Letting
 \[
 \cO:=\{f_0\in H^2(\s)\,:\, \text{$\p_\nu p_-< \p_\nu p_+$ on $[y=f_0(x)]$}\}
 \]
 denote the set of initial data in $H^2(\s)$ for which the Rayleigh-Taylor condition holds,
  it is shown in  Section \ref{Sec:5} that $\cO$ is   nonempty   precisely when $\Theta>0$.
This condition on the constants has been identified also in the nonperiodic case.    
 In fact, we prove that  if $\Theta>0$, then $\cO$ is an open subset of  $H^2(\s)$ which contains  all  constant functions.
Using the abstract fully nonlinear parabolic theory  established in \cite{DG79, L95}, we prove below that
the Muskat problem without surface tension is well-posed in the set $\cO$, cf. Theorem \ref{MT:2}.
 Physically,   in the particular situation when gravity is neglected $\Theta>0$ is equivalent to
 the fact that the more viscous fluid
  enters the region occupied by less viscous one, while in the case $V=0$ the condition
  $\Theta>0$ means that the fluid located below has a larger density.

  \begin{thm} \label{MT:2}
Let  $\sigma=0$, $\mu_-\neq\mu_+$\footnote{Theorem \ref{MT:2} is still valid if $\mu_-=\mu_+$,
however its claims can be improved, 
cf. \cite[Theorem 1.1]{MM17x}, as the problem \eqref{P} is under this restriction  of quasilinear type.}, and    assume that 
$\Theta>0$.
Given $f_0\in\cO$, the  problem \eqref{P} possesses  a   solution
\[ f\in C([0,T],\cO)\cap C^1([0,T], H^1(\s))\cap C^{\alpha}_{\alpha}((0,T], H^2(\s)) \] 
for some $T>0$ and an arbitrary  $\alpha\in(0,1)$.
Additionally, the following statements are true:
\begin{itemize}
 \item[$(i)$] $f$ is the unique solution to \eqref{P} belonging to 
 \[\bigcup_{\beta\in(0,1)} C([0,T],\cO)\cap C^1([0,T], H^1(\s))\cap C^{\beta}_{\beta}((0,T], H^2(\s)).\]
 \item[$(ii)$] $f$ may be extended to a maximally defined solution 
 $$f(\,\cdot\,; f_0)\in C([0,T_+(f_0)),\cO)\cap C^1([0,T_+(f_0)), H^1(\s))\cap \bigcap_{\beta\in(0,1)}C^{\beta}_{\beta}((0,T], H^2(\s))$$
 for all $T<T_+(f_0)$, where $T_+(f_0)\in (0,\infty].$\\[-2ex]
 \item[$(iii)$] The solution map   $[(t,f_0)\mapsto f(t;f_0)]$ defines a  semiflow 
 on $\cO$ which is real-analytic in the open set $\{(t,f_0)\,:\, f_0\in\cO,\, 0<t<T_+(f_0)\}$.\\[-2ex]
 \item[$(iv)$] If $f(\,\cdot\,; f_0):[0,T_+(f_0))\cap[0,T]\to\cO$ is uniformly continuous for all $T>0$, then either  $T_+(f_0)=\infty $, or
 \[\text{$T_+(f_0)<\infty$ \,  and \, $ {\rm dist}(f(t;f_0),\p\cO)\to 0$ for $t\to T_+(f_0).$}\]
 \item[$(v)$] If $f(\,\cdot\,; f_0)\in B((0,T), H^{2+\e}(\s))$ for some $T\in(0,T_+(f_0))$ and $\e\in(0,1)$ arbitrary small, then 
 \[
 f\in C^\omega((0,T)\times\R,\R).
 \]  
\end{itemize} 
\end{thm}\medskip

The assertions of Theorem \ref{MT:2} are weaker compared to that of Theorem \ref{MT:1}. 
For example the uniqueness claim at $(i)$ is established in the setting of strict solutions (in the sense of \cite[Chapter 8]{L95}) 
 which belong additionally to some singular H\"older space
 \[
 C^\beta_\beta((0,T], H^2(\s)):=\Big\{u\in B((0,T], H^2(\s))\,:\, \sup_{s\neq t}\frac{\|t^\beta u(t)-s^\beta u(s)\|_{H^2}}{|t-s|^\beta}<\infty\Big\}
 \]
 with $\beta\in(0,1).$
This drawback results from the fact that  in the absence of  surface tension effects we deal with a fully nonlinear (and nonlocal) problem.
We also  point out that the  parabolic smoothing property    established at $(v)$ holds only for solutions  $f(\,\cdot\,; f_0)\in B((0,T), H^{2+\e}(\s))$ for some $\e>0$.
This additional boundedness condition is needed  because the space-time translation 
\[\big[u\mapsto [(t,x)\mapsto u(at, x+bt)]\big] \]
does not define  for $a,\,b >0$ a bounded operator between these singular Hölder spaces. 
This property hiders us to use the parameter trick from the proof of Theorem \ref{MT:1} to establish parabolic smoothing for all solutions in Theorem \ref{MT:2}.
 However, the boundedness hypothesis imposed at $(v)$ is satisfied if $f_0\in\cO\cap H^3(\s)$ because the statements $(i)-(iv)$ in Theorem \ref{MT:2} remain true when replacing $H^k(\s)$
 by $H^{k+1}(\s) $ for $k\in\{1,2\}$ (possibly with a smaller maximal existence time).
 
 Finally, we point out that in the case when $\sigma=0$ the equilibrium solutions to \eqref{P} are the constant functions.
  Theorem \ref{MT:3}states that the zero solution to \eqref{P} (and therewith all other equilibria)   is exponentially stable under
 perturbations with zero integral mean.

 \begin{thm}[Exponential stability]\label{MT:3}
  Let $\sigma=0$ and  $\Theta>0$. Then, given $\omega\in(0,k\Theta/(\mu_-+\mu_+))$,  there exist constants $\delta>0$ and $M>0$,  with the property that if $f_0\in \wh H^2(\s)$ satisfies 
$\|f_0\|_{H^{2}}\leq\delta,$ 
 then $T_+(f_0)=\infty$  and\footnote{We write $\dot f$ to denote the derivative $df/dt$.}
  \begin{align*}
   \|f(t)\|_{H^2}+\|\dot f(t)\|_{H^1}\leq Me^{-\omega t}\|f_0\|_{H^{2}}\qquad \text{for all $t\geq0.$}
  \end{align*} 
 \end{thm}

 Before proceeding with our analysis we emphasize that the periodic case considered herein 
 is more involved that the ``canonical'' nonperiodic Muskat problem because abstract results from harmonic analysis,
 cf. \cite[Theorem 1]{TM86}, which
 directly apply to the nonperiodic case (in order to establish useful mapping properties and commutator estimates)
 have no correspondence in the set of periodic functions.
However,  we derive in  Appendix \ref{S:A},   by  using the results from the nonperiodic case \cite{M16x,M17x},    the 
boundedness of certain multilinear singular integral operators which can be directly applied in the proofs.
A further drawback of the equations \eqref{EP} is that some of the integral terms   are of lower order 
and   some of the arguments are therefore lengthy. 
Finally, we  point out that the stability issue remains an open question for  the nonperiodic counterpart of \eqref{P}.

%%%%%%%%%%%%%%%%%%%%%%%%%%%%%%%%%%%%%%%%%%%%%%%
%%%%%%%%%%%%%%%%%%%%%%%%%%%%%%%%%%%%%%%%%%%%%%
%%%%%%%%%%%%%%%%%%%%%%%%%%%%%%%%%%%%%%%%%%%%%%%
%%%%%%%%%%%%%%%%%%%%%%%%%%%%%%%%%%%%%%%%%%%%%%
%%%%%%%%%%%%%%%%%%%%%%%%%%%%%%%%%%%%%%%%%%%%%%%
%%%%%%%%%%%%%%%%%%%%%%%%%%%%%%%%%%%%%%%%%%%%%%
%%%%%%%%%%%%%%%%%%%%%%%%%%%%%%%%%%%%%%%%%%%%%%%
%%%%%%%%%%%%%%%%%%%%%%%%%%%%%%%%%%%%%%%%%%%%%%
\section{The equations of motion and the equivalence of the formulations}\label{Sec:2}
%%%%%%%%%%%%%%%%%%%%%%%%%%%%%%%%%%%%%%%%%%%%%%%
%%%%%%%%%%%%%%%%%%%%%%%%%%%%%%%%%%%%%%%%%%%%%%
%%%%%%%%%%%%%%%%%%%%%%%%%%%%%%%%%%%%%%%%%%%%%%%
%%%%%%%%%%%%%%%%%%%%%%%%%%%%%%%%%%%%%%%%%%%%%%
%%%%%%%%%%%%%%%%%%%%%%%%%%%%%%%%%%%%%%%%%%%%%%%
%%%%%%%%%%%%%%%%%%%%%%%%%%%%%%%%%%%%%%%%%%%%%%
%%%%%%%%%%%%%%%%%%%%%%%%%%%%%%%%%%%%%%%%%%%%%%%
%%%%%%%%%%%%%%%%%%%%%%%%%%%%%%%%%%%%%%%%%%%%%%

 In this section we  present  the classical formulation  of the Muskat problem (see \eqref{PB*} below) introduced in
 \cite{Mu34} and prove that this formulation is equivalent to    the contour integral formulation \eqref{P} in a quite general setting,
 cf. Proposition \ref{P:21}.
 
 We first introduce the equations of motion.
 In the fluid layers the dynamic is governed by the  equations 
 \begin{subequations}\label{PB*}
\begin{equation}\label{eq:S1}
\left\{\begin{array}{rllllll}
{\rm div}\,  v_\pm(t)\!\!\!\!&=&\!\!\!\!0  , \\[1ex]
v_\pm(t)\!\!\!\!&=&\!\!\!\!-\cfrac{k}{\mu_\pm}\big(\nabla p_\pm(t)+(0,\rho_\pm g)\big) 
\end{array}
\right.\qquad\text{in $ \0_\pm^V(t)$},
\end{equation} 
  where $v_\pm(t):=(v_\pm^1(t),v_\pm^2(t))$   denotes the velocity field of the fluid $\pm.$
 While $\eqref{eq:S1}_1$ is the incompressibility condition, the equation $\eqref{eq:S1}_2$ is known as Darcy's  law. 
 This linear relation is frequently used for flows which are laminar, cf. \cite{Be88}.
 These equations are supplemented by the following boundary conditions  at the free interface 
\begin{equation}\label{eq:S2}
\left\{\begin{array}{rllllll}
p_+(t)-p_-(t)\!\!\!\!&=&\!\!\!\!\sigma \kappa(f(t)), \\[1ex]
 \langle v_+(t)| \nu(t)\rangle\!\!\!\!&=&\!\!\!\!  \langle v_-(t)| \nu(t)\rangle 
\end{array}
\right.\qquad\text{on $[y=f(t,x)+tV]$},
\end{equation} 
where    $\nu(t) $ is the unit normal at $[y=f(t,x)+tV]$ pointing into $\0_+^V(t)$   
and   $\langle \, \cdot\,|\,\cdot\,\rangle$  the inner product in~$\R^2$.
Additionally, we impose the  far-field boundary  condition
 \begin{equation}\label{eq:S3}
\begin{array}{llllll}
v_\pm(t,x,y)\to (0,V) &\text{for  $|y|\to\infty$ (uniformly in $x$)}.
\end{array}
\end{equation} 
The motion of the free interface  is  described by  the kinematic boundary condition
 \begin{equation}\label{eq:S4}
 \p_tf(t)\, =\, \langle v_\pm(t)| (-f'(t),1)\rangle - V\qquad\text{on   $ [y=f(t,x)+tV],$}
\end{equation}  
and, since we consider $2\pi$-periodic flows,   $f(t),$ $v_\pm(t)$, and $p_\pm(t)$ are assumed to be $2\pi$-periodic with respect to $x$ for all $t\geq 0$. 
Finally, we supplement the system with the initial condition
\begin{equation}\label{eq:S5}
f(0)\, =\, f_0.
\end{equation} 
\end{subequations}

It is convenient to  rewrite the equations \eqref{PB*} in a reference frame that moves with the constant velocity $(0,V).$
To this end we let
\[
\0_\pm(t):= \{(x,y)\in\R^2\,:\, \pm(f(t,x)-y)<0\}=\0_\pm^V(t)-(0,tV),
\]
and 
\[\left\{\begin{array}{rllllll}
P_\pm(t,x,y)\!\!\!\!&=&\!\!\!\!p_\pm(t,x,y+tV), \\[1ex]
 V_\pm(t,x,y)\!\!\!\!&=&\!\!\!\!v_\pm(t,x,y+tV)-(0,V)
\end{array}
\right.\qquad\text{for $t\geq0$ and $(x,y)\in\0_\pm(t).$}
\]
 Direct computations show that \eqref{PB*} is equivalent to
\begin{equation}\label{PB}
\left\{
\begin{array}{rlllll}
 {\rm div}\,  V_\pm(t)\!\!\!\!&=&\!\!\!\! 0 &&\text{in $ \0_{\pm}(t),$}\\[1ex]
\mu_\pm\big(V_\pm(t)+(0,V)\big)\!\!\!\!&=&\!\!\!\!-{k}\big(\nabla P_\pm(t)+(0,\rho_\pm g)\big)&&\text{in $ \Omega_{\pm}(t),$} \\[1ex]
\langle V_+(t)| \nu(t)\rangle\!\!\!\!&=&\!\!\!\! \langle V_-(t)| \nu(t)\rangle&& \text{on $ [y=f(t,x)], $}\\[1ex]
P_+(t)-P_-(t)\!\!\!\!&=&\!\!\!\! \sigma\kappa(f(t))&&\text{on $ [y=f(t,x)], $}\\[1ex]
V_\pm(t,x,y)&\to& 0 &&\text{for  $|y|\to\infty $,}\\[1ex]
\partial_tf(t) \!\!\!\!&=&\!\!\!\!\langle V_\pm(t)| (-f'(t),1)\rangle && \text{on   $ [y=f(t,x)], $}\\[1ex]
f(0)\!\!\!\!&=&\!\!\!\!f_0. 
 \end{array}
\right.
 \end{equation}

 In Proposition \ref{P:21}  we establish the equivalence of the two formulations \eqref{P} and \eqref{PB}. 
It is important to point out that  the function $\overline\omega$  in $\eqref{EP}_1$ is uniquely identified by $f$  in the space $\wh L_2(\s).$ 
(this feature is established rigorously only later on in Theorem \ref{T:I1}). 
 This aspect is   essential  at several places in this paper, see  Proposition \ref{P:21}  and the preparatory lemma below.

\begin{lemma}\label{L:220}
 Given  $f\in H^1(\s)$ and $ \overline\omega \in \wh L_2(\s)$ let  
 \begin{equation}\label{VFD}
\begin{aligned}
 V^1(x,y)&:=-\frac{1}{4\pi} \int_{\s}\overline\omega(s)\frac{\tanh((y-f(s))/2)\big[1+\tan^2((x-s)/2)\big]}{\tan^2((x-s)/2)+\tanh^2((y-f(s))/2) }\, ds,\\[1ex]
 V^2(x,y)&:=\frac{1}{4\pi} \int_{\s}\overline\omega(s)\frac{\tan ((x-s)/2)\big[1-\tanh^2((y-f(s))/2)\big]}{\tan^2((x-s)/2)+\tanh^2((y-f(s))/2) }\, ds
\end{aligned}
\end{equation}
for $(x,y)\in  \R^2 \setminus[y=f(x)] $ and set $V:=(V^1,V^2)$ and
$V_\pm:=V|_{\0_\pm},$  where  $$\0_\pm :=\{(x,y)\in\R^2\,:\, \pm(f(x)-y)<0\}.$$ 
  Then, there exists a constant $C=C(\|f\|_\infty)>0$ such that
 \[
 |V_\pm(x,y)|\leq C\|\overline\omega\|_1e^{-|y|/2} 
 \]
for all $(x,y)\in\0_\pm $ satisfying $|y|\geq 1+2\|f\|_\infty$.
\end{lemma}
\begin{proof} Let first $f\neq0$.
 Taking advantage of 
 \[
\max\Big\{\tanh\Big(\frac{\|f\|_\infty}{2}\Big),\tanh\Big(\frac{|y|}{4}\Big)\Big\}\leq\Big|\tanh\Big(\frac{y-f(s)}{2}\Big)\Big|  \qquad \text{for $|y|\geq 2\|f\|_\infty$}, 
 \]
   for $|y|\geq2\|f\|_\infty,$ it follows that
 \begin{align*}
 |V^2_\pm(x,y)|\leq \frac{\|\overline\omega\|_1}{\tanh(\|f\|_\infty/2)} \big|1-\tanh^2(y/4)\big|\leq\frac{\|\overline\omega\|_1}{ \tanh(\|f\|_\infty/2)}e^{-|y|/2}.
 \end{align*}
 In order to estimate   $V^1_\pm$ we use the fact that $\langle\overline\omega\rangle=0$ to derive, after performing some elementary estimates, that  
 \begin{align*}
 |V^1_\pm(x,y)|&\leq  \int_{\s}|\overline\omega(s)|\Big|\frac{\tanh((y-f(s))/2)\big[1+\tan^2((x-s)/2)\big]}{\tan^2((x-s)/2)+\tanh^2((y-f(s))/2) }\mp 1\Big|\, ds \\[1ex]
 &\leq C \|\overline\omega\|_1  (1- \tanh(|y|/4)) \leq C\|\overline\omega\|_1e^{-|y|/2}
 \end{align*}
 for all $|y|\geq2\|f\|_\infty$.
 The claim for $f=0$ follows in a similar way.
\end{proof}
 
 In Proposition \ref{P:21} we show that, given a solution to \eqref{P}, the velocity field in the classical formulation \eqref{PB*} at time $t$   can be expressed in terms of 
   $f:=f(t)$ and $\overline\omega:=\overline\omega(t)$ according to Lemma \ref{L:220}, provided that $f$ and $\overline\omega$ 
   have suitable  regularity properties.
   We point out that a formal derivation of the formula \eqref{VFD}  is provided, in a more general context, in \cite[Section 2]{CCG11}.
  In Lemma~\ref{L:221} we establish further properties of the velocity field  defined in Lemma \ref{L:220}.

\begin{lemma}\label{L:221} Let  $f\in H^2(\s)$ and $ \overline\omega \in \wh H^1(\s)$.   The vector field  $V_\pm$ introduced in Lemma \ref{L:220} belongs to
${\rm C}(\overline{\Omega}_\pm)\cap {\rm C}^1({\Omega_\pm})$,  it is  divergence free and irrotational, and
\begin{equation}\label{VFB}
\begin{aligned}
V_\pm (x,f(x))&=\frac{1}{4\pi} \PV\int_{-\pi}^{\pi}\overline\omega(x-s)\frac{\big(-(T_{[x,s]}f) (1+t_{[s]}^2),t_{[s]}[1-(T_{[x,s]}f)^2]\big) }{t_{[s]}^2+(T_{[x,s]}f)^2 }\, ds\\[1ex]
 &\hspace{0.424cm}\mp\frac{1}{2}\frac{\overline\omega(x)(1,f'(x))}{1+f'^2(x)},\qquad x\in\R.
\end{aligned}
\end{equation} 
Letting further  
\begin{align}\label{Pres}
 P_\pm(x,y):=c_\pm-\frac{\mu_\pm}{k}\int_0^x V_\pm^1(s,\pm d)\, ds-\frac{\mu_\pm}{k}\int_{\pm d}^y V_\pm^2(x,s)\, ds-\Big(\rho_\pm g+\frac{\mu_\pm V}{k}\Big)y
\end{align}
for $ (x,y)\in\ov\0_\pm$,  where $c_\pm\in\R$ and $d>\|f\|_\infty$, it holds that
$ P_\pm\in{\rm C}^1(\overline{\Omega}_\pm)\cap {\rm C}^2({\Omega_\pm})$ and the relations $\eqref{PB}_1$-$\eqref{PB}_3$, $\eqref{PB}_5$ are all satisfied.
\end{lemma}
\begin{proof}
The theorem on the differentiation of parameter integrals shows that  $V_\pm $ is continuously differentiable in $\0_\pm,$   divergence free, and irrotational.
In order to show that $V_\pm\in {\rm C}(\overline{\Omega}_\pm)$ it suffices to show that the one-sided limits when approaching a point $(x_0,f(x_0))\in[y=f(x)]$ from $\0_-$ and $\0_+$, respectively,
exist.
To this end we note that the complex conjugate of $(V^1_\pm,V^2_\pm)$ satisfies
\[
\overline{(V^1_\pm,V^2_\pm)}(z)=\frac{1}{4\pi i}\int_\G \frac{g(\xi)}{\tan\big((\xi-z)/2\big)}\, d\xi\qquad \text{for $z=(x,y)\not\in[y=f(x)]$,}
\]
 with  $\G$ being  a $2\pi$-period of the  graph $[y=f(x)]$ and with $g:\G\to \C$ defined by
\[
g(\xi)=-\frac{\overline\omega(s)(1-if'(s))}{1+f'^2(s)}\qquad \text{for $\xi=(s,f(s))\in\G$}. 
\]
Given $z=(x,y)\not\in[y=f(x)]$, it is convenient to write 
\[
\overline{(V^1_\pm,V^2_\pm)}(z)=\frac{1}{4\pi i}\int_\G g(\xi)\Big[\frac{1}{\tan\big((\xi-z)/2\big)}-\frac{1}{(\xi-z)/2}\Big]\, d\xi+\frac{1}{2\pi i}\int_\G  \frac{g(\xi)}{ \xi-z }\, d\xi,
\]
because Lebesgue's theorem now shows that if $z_n$ approaches $z_0=(x_0,f(x_0))$ from $\0_+$ (or $\0_-$), then
\[
\frac{1}{4\pi i}\int_\G g(\xi)\Big[\frac{1}{\tan\big((\xi-z_n)/2\big)}-\frac{1}{(\xi-z_n)/2}\Big]\, d\xi
\underset{n\to \infty}\longrightarrow\frac{1}{4\pi i}\int_\G g(\xi)\Big[\frac{1}{\tan\big((\xi-z_0)/2\big)}-\frac{1}{(\xi-z_0)/2}\Big]\, d\xi.
\]
Moreover, using Plemelj's formula, cf. e.g.  \cite[Theorem 2.5.1]{JKL93},  we find that 
 \[
\frac{1}{2\pi i}\int_\G  \frac{g(\xi)}{ \xi-z_n }\, d\xi\underset{n\to \infty}\longrightarrow\pm\frac{g(z_0)}{2} + \frac{1}{2\pi i}\PV\int_\G  \frac{g(\xi)}{ \xi-z_0 }\, d\xi,
\]
where the $\PV$ is taken at $\xi=z_0$, and we conclude that
\[
\overline{(V^1_\pm,V^2_\pm)}(z_n)\underset{n\to \infty}\longrightarrow\pm\frac{g(z_0)}{2}
+\frac{1}{4\pi i}\PV\int_\G  \frac{g(\xi)}{\tan\big((\xi-z_0)/2\big)} \, d\xi.
\]
The formula \eqref{VFB} and the property $V_\pm\in {\rm C}(\overline{\Omega}_\pm)$  follow at once.
The remaining claims are simple consequences of  Lemma \ref{L:220} and of the already established properties.
\end{proof}

 Using Lemma \ref{L:220} and \ref{L:221},  we conclude this section with the following equivalence result.
 
\begin{prop}[Equivalence of  formulations]\label{P:21}
Let   $T\in(0,\infty] $ be given.
\begin{itemize}
 \item[$(a)$] Let $\sigma=0$. 
The following are equivalent:\\[-1.5ex]
 \begin{itemize}
 \item[$(i)$] the  problem \eqref{PB} for $ f\in   {\rm C}^1([0,T), L_2(\mathbb{S}))$ and
 \begin{align*}
  \bullet &\, \, \, \, f(t)\in  H^2(\mathbb{S}),\, \, \overline\omega(t):=\big\langle (V_-(t)-V_+(t))|_{[y=f(t,x)]} \big|(1,f'(t))\big\rangle \in  \wh H^1(\mathbb{S}),\\ 
  \bullet &\, \, \, V_\pm(t)\in {\rm C}(\overline{\Omega}_\pm(t))\cap {\rm C}^1({\Omega_\pm(t)}), \, P_\pm(t)\in {\rm C}^1(\overline{\Omega}_\pm(t))\cap {\rm C}^2({\Omega_{\pm}(t)}) 
 \end{align*}
 for all $t\in[0,T)$;\\[-2ex]
\item[$(ii)$] the evolution  problem \eqref{P} for $f\in   {\rm C}^1([0,T), L_2(\mathbb{S}))$, $f(t)\in  H^2(\mathbb{S})$, and $\overline\omega(t)\in  \wh H^1(\mathbb{S})$ for all $t\in[0,T)$.\\[-1.5ex]
 \end{itemize}
\item[$(b)$] Let $\sigma>0$. 
The following are equivalent:\\[-1.5ex]
 \begin{itemize}
 \item[$(i)$] the problem \eqref{PB} for $f\in   {\rm C}^1((0,T), L_2(\mathbb{S}))\cap{\rm C}([0,T), L_2(\mathbb{S}))$ and
 \begin{align*}
  \bullet &\, \, \,  f(t)\in  H^4(\mathbb{S}),   \,\, \overline\omega(t):=\big\langle (V_-(t)-V_+(t))|_{[y=f(t,x)]} \big|(1,f'(t))\big\rangle \in  \wh H^1(\mathbb{S}),  \\ 
  \bullet &\, \, \, V_\pm(t)\in {\rm C}(\overline{\Omega}_\pm(t))\cap {\rm C}^1({\Omega_\pm(t)}), \, P_\pm(t)\in {\rm C}^1(\overline{\Omega}_\pm(t))\cap {\rm C}^2({\Omega_{\pm}(t)}) 
 \end{align*}
 for all $t\in(0,T);$\\[-2ex]
\item[$(ii)$] the Muskat problem \eqref{P} for $f\in   {\rm C}^1((0,T), L_2(\mathbb{S}))\cap {\rm C}([0,T), L_2(\mathbb{S}))$, $ f(t)\in  H^4(\mathbb{S})$, and 
$ \overline\omega(t)\in  \wh H^1(\mathbb{S}) $ for all $t\in(0,T)$.
 \end{itemize}
 \end{itemize}
 \end{prop}
\begin{proof}    
To prove the implication $(i)\Rightarrow (ii)$ of $(a)$, let $(f,V_\pm,P_\pm)$ be a solution to \eqref{PB}  on $[0,T)$ and
 choose $t\geq0$ fixed but arbitrary (the time dependence is not written explicitly in this proof).
Letting 
\[
\omega:=\p_x V^2-\p_yV^1\in\mathcal{D}'(\R^2)
\]
denote the vorticity associated to the global velocity field
\[
(V^1,V^2):=V_-{\bf 1}_{\0_-}+V_+{\bf 1}_{\0_+},
\]
 where ${\bf 1}_{\0_\pm}$ is the characteristic function of    $\0_\pm$,
 it follows from $\eqref{PB}_3$ and Stokes' theorem that 
\[
\omega=\overline\omega\delta_{[y=f(x)]},
\]
where
\begin{align*} 
 \overline\omega:= \big\langle(V_- -V_+)|_{[y=f(x)]}\big| (1,f')\big\rangle\in\wh H^1(\s).
\end{align*}
Similarly as in the particular case $\mu_-=\mu_+$, cf. \cite[Proposition 2.2]{MM17x}, we find that the global velocity field
$(V^1,V^2)$ is given by \eqref{VFD}. Lemma \ref{L:221} now shows, together with the kinematic boundary condition,   that $f$ solves    the equation $\eqref{EP}_1.$
 Besides, differentiating the Laplace-Young equation  $\eqref{PB}_4$, the relations $\eqref{PB}_2$ and \eqref{VFB} finally  lead us to  $\eqref{EP}_2,$ and the proof of this implication is complete.
 
 For the reverse implication, we define $V_\pm$ according to \eqref{VFD}, and the pressures by \eqref{Pres}. 
For suitable  $c_\pm$, it follows from $\eqref{EP}_2$ and Lemmas \ref{L:220}-\ref{L:221} that indeed $(f,V_\pm,P_\pm)$ solves \eqref{PB}.

The equivalence stated at $(b)$ follows in a similar way.
\end{proof}

 %%%%%%%%%%%%%%%%%%%%%%%%%%%%%%%%%%%%%%%%%%%%%%%%%%%%%%%%%%%%%%%%%%%
%%%%%%%%%%%%%%%%%%%%%%%%%%%%%%%%%%%%%%%%%%%%%%%%%%%%%%%%%%%%%%%%%%%%
%%%%%%%%%%%%%%%%%%%%%%%%%%%%%%%%%%%%%%%%%%%%%%%%%%%%%%%%%%%%%%%%%%%%
%%%%%%%%%%%%%%%%%%%%%%%%%%%%%%%%%%%%%%%%%%%%%%%%%%%%%%%%%%%%%%%%%%%
%%%%%%%%%%%%%%%%%%%%%%%%%%%%%%%%%%%%%%%%%%%%%%%%%%%%%%%%%%%%%%%%%%%%
%%%%%%%%%%%%%%%%%%%%%%%%%%%%%%%%%%%%%%%%%%%%%%%%%%%%%%%%%%%%%%%%%%%%
\section{The double layer potential and its adjoint}\label{Sec:3}
 %%%%%%%%%%%%%%%%%%%%%%%%%%%%%%%%%%%%%%%%%%%%%%%%%%%%%%%%%%%%%%%%%%%
%%%%%%%%%%%%%%%%%%%%%%%%%%%%%%%%%%%%%%%%%%%%%%%%%%%%%%%%%%%%%%%%%%%%
%%%%%%%%%%%%%%%%%%%%%%%%%%%%%%%%%%%%%%%%%%%%%%%%%%%%%%%%%%%%%%%%%%%%
%%%%%%%%%%%%%%%%%%%%%%%%%%%%%%%%%%%%%%%%%%%%%%%%%%%%%%%%%%%%%%%%%%%
%%%%%%%%%%%%%%%%%%%%%%%%%%%%%%%%%%%%%%%%%%%%%%%%%%%%%%%%%%%%%%%%%%%%
%%%%%%%%%%%%%%%%%%%%%%%%%%%%%%%%%%%%%%%%%%%%%%%%%%%%%%%%%%%%%%%%%%%%
We point out that the equation $\eqref{EP}_2$ is linear with respect to   $\overline\omega(t)$. 
The   main goal of this section is to address the solvability of this equation for $\overline\omega(t)$ in suitable function spaces, cf. Theorems \ref{T:I1} and
\ref{T:I2}. To this end we first associate to  \eqref{EP} two singular operators and study  their mapping properties (see Lemmas \ref{L:32} and \ref{L:31}).
 Finally, in Theorem \ref{T:I3} and Lemma \ref{L:33} we study the properties of  the adjoints of these singular operators.

To begin, we write  $\eqref{EP}_2$  in the more compact form
\begin{align}\label{EP2}
 (1+a_\mu\bA(f))[\overline\omega]  =  \displaystyle\frac{2k}{\mu_-+\mu_+}(\sigma\kappa(f)-\Theta f)' ,
&&\!\!\!\!
\end{align}
where $\bA(f)$ is the linear operator
\begin{align}\label{ADL}
 \bA(f)[\overline\omega](x):=\frac{1}{2\pi }\PV\int_{-\pi}^{\pi}\frac{f'(x)t_{[s]}[1-(T_{[x,s]}f)^2]- (1+t_{[s]}^2)T_{[x,s]}f}{t_{[s]}^2+(T_{[x,s]}f)^2 }\overline\omega(x-s)\, ds.
\end{align}
Given $f\in H^r(\s)$ with $r>3/2$, we prove in Lemma \ref{L:31}  that  $\bA(f)\in\kL(L_2(\s)).$  
Then, it is a matter of direct computation  to verify that $\bA(f)$ is the $L_2$-adjoint of the double layer potential
\begin{align}\label{DL}
 (\bA(f))^*[\xi](x):=\frac{1}{2\pi } \PV\int_{-\pi}^{\pi}\frac{(1+t_{[s]}^2)(T_{[x,s]}f)-f'(x-s)t_{[s]}[1-(T_{[x,s]}f)^2]}{t_{[s]}^2+(T_{[x,s]}f)^2 }\xi(x-s)\, ds.
\end{align}
A main part of the subsequent analysis is devoted to the study of the invertibility of the linear operator $1+a_\mu\bA(f) $  in the algebras $\kL(\wh L_2(\s))$ and  $\kL(\wh H^1(\s))$.
These invertibility properties   enable us    to solve $\eqref{EP2}$ and to formulate \eqref{P} as an evolution equation for $f$ only, that is 
\begin{equation}\label{EVE}
\p_tf=\frac{k}{\mu_-+\mu_+}\bB(f)\big[(1+a_\mu\bA(f))^{-1}[(\sigma\kappa(f)-\Theta f)']\big],
\end{equation}
 where we have associated to $\eqref{EP}_1$ the operator  $\bB(f)$  defined by
 \begin{align}\label{BDL}
  \bB(f)[\overline\omega](x):= \frac{1}{2\pi} \PV\int_{-\pi}^{\pi}\frac{f'(x) (1+t_{[s]}^2)(T_{[x,s]}f)+t_{[s]}[1-(T_{[x,s]}f)^2] }{t_{[s]}^2+(T_{[x,s]}f)^2 }\overline\omega(x-s)\, ds.
 \end{align}
  
As a first result we establish the following mapping properties.

\begin{lemma}\label{L:32} 
 Given $r>3/2$, it holds  that 
 \begin{equation}\label{E:RB}
   \bB\in {\rm C}^\omega (H^r(\s),\kL(\wh L_2(\s)))\cap  {\rm C}^\omega (H^2(\s),\kL(\wh H^1(\s))).
 \end{equation}
\end{lemma}
\begin{proof} 
Let us first assume that 
\begin{align}\label{DAR1}
 \bB\in {\rm C}^\omega (H^r(\s),\kL(L_{2}(\s)))\cap  {\rm C}^\omega (H^2(\s),\kL(H^1(\s))).
\end{align}
Given $f,\, \overline\omega\in {\rm C}^\infty(\s)$ with $\langle \overline\omega\rangle=0$ let $V_-$ be as defined in Lemma \ref{L:220}.
Observing that 
 \[
 \bB(f)[\overline\omega]=2\langle V_-|_{[y=f(x)]}|(-f',1)\rangle\in {\rm C}(\s), 
\]
Stokes' formula together with Lemmas \ref{L:220}-\ref{L:221}   yields
\begin{align*}
 \frac{1}{2}\langle \bB(f)[\overline\omega]\rangle=\int_\G \langle V_-|\nu\rangle\, d\sigma=\int_{\0_-}{\rm div}\, V_-\, d(x,y)=0,
\end{align*}
and therefore $\bB(f)[\overline\omega]\in \wh L_2(\s)$.
This immediately implies \eqref{E:RB}.

Hence, we are left to establish \eqref{DAR1}.
To this end it is convenient to write
\begin{align*}
 \bB(f)=f'\bB_1(f)-\bB_2(f)+\bB_3(f),
\end{align*}
where
\begin{align*}
 \bB_1(f)[\overline\omega](x)&:=\frac{1}{2\pi} \int_{-\pi}^{\pi}\frac{ t_{[s]}^2  T_{[x,s]}f }{t_{[s]}^2+(T_{[x,s]}f)^2 }
 \overline\omega(x-s)\, ds\\[1ex]
 &\hspace{0.424cm}+\frac{1}{2\pi} \int_{-\pi}^{\pi} \Big[\frac{ T_{[x,s]}f }{t_{[s]}^2+(T_{[x,s]}f)^2 }-\frac{(\delta_{[x,s]}f/2) }{(s/2)^2+(\delta_{[x,s]}f/2)^2}\Big]\overline\omega(x-s)\, ds,\\[1ex]
 \bB_2(f)[\overline\omega](x)&:=\frac{1}{2\pi} \int_{-\pi}^{\pi}\frac{ t_{[s]}  (T_{[x,s]}f)^2 }{t_{[s]}^2+(T_{[x,s]}f)^2 }\overline\omega(x-s)\, ds\\[1ex]
  &\hspace{0.424cm}-\frac{1}{2\pi} \int_{-\pi}^{\pi}  \Big[\frac{t_{[s]} }{t_{[s]}^2+(T_{[x,s]}f)^2 }-\frac{s/2 }{(s/2)^2+(\delta_{[x,s]}f/2)^2}\Big]\overline\omega(x-s)\, ds,\\[1ex]
        \bB_3(f)[\overline\omega](x)&:=\frac{1}{\pi} \PV\int_{-\pi}^{\pi}  \frac{s+f'(x)(\delta_{[x,s]}f)  }{s^2+(\delta_{[x,s]}f)^2}\overline\omega(x-s)\, ds.
\end{align*}
Taking advantage of the relations
\begin{equation}\label{aa}
\begin{aligned}
&\tanh (x)\leq x ,\, \, x\geq0, \qquad x\leq \tan (x),\, \, x\in[0,\pi/2), \\[1ex]
&|\tanh(x)-x|\leq |x|^3 ,\, \, x\in\R, \qquad |\tan (x)-x|\leq |x|\tan^2(x),\, \, |x|<\pi/2, 
\end{aligned}
\end{equation}
it is easy to see that $\bB_i(f)  \in\kL(L_2(\s), L_\infty(\s))$ for  $i\in\{1,2\}$ (and that $\PV$ is not needed).
In fact these mappings are real-analytic, that is
\begin{align}\label{EG1}
 \bB_i\in {\rm C}^\omega(H^r(\s), \kL(L_2(\s), L_\infty(\s))),\quad i\in\{1,2\}.
\end{align}
Furthermore, given $\tau\in(1/2,1)$,   classical (but lengthy) arguments (see \cite[Lemmas 3.2-3.3]{MM17x} where similar integral operators
are discussed) show that 
\begin{align}\label{EG2}
 \bB_i\in {\rm C}^\omega(H^r(\s), \kL(H^\tau(\s), {\rm {\rm C}^1(\s)})),\quad i\in\{1,2\},
\end{align}
and we are left to consider the operator $\bB_3.$

Recalling Lemma \ref{L:A1}, we  see that
\[
\pi\bB_3(f)[\overline\omega]=C_{0,1}(f)[\overline\omega]+f'C_{1,1}(f)[f,\overline\omega],
\]
and Lemma \ref{L:A1} $(i)$ immediately yields $\bB_3(f)\in \kL(L_2(\s)).$
Moreover, arguing as in \cite[Section 5]{M16x}, it follows that 
\begin{align}\label{EG3}
 \bB_3\in {\rm C}^\omega(H^r(\s), \kL(L_2(\s))).
\end{align}
In order to prove that $\bB_3(f)\in \kL(H^1(\s)),$ when additionally $f\in H^2(\s)$, we let $\{\tau_\e\}_{\e\in\R}$ denote the 
 $C_0$-group of right translations, that is $\tau_\e h(x)=h(x-\e)$ for $x\in\R$ and $h\in L_2(\s)$.
 Given $\e>0$ and~$\overline\omega\in H^1(\s)$, it holds that 
  \begin{align*}
  \pi\frac{\tau_\e(\bB_3(f)[\overline\omega])-\bB_3(f)[\overline\omega]}{\e}&=C_{0,1}(\tau_\e f)\Big[\frac{\tau_\e \overline\omega-\overline\omega}{\e}\Big]-C_{2,2}(f,\tau_\e f)\Big[\frac{\tau_\e f-f}{\e}, \tau_\e f+f,\overline\omega\Big]\\[1ex]
  &\hspace{0.424cm}+\frac{\tau_\e f'-f'}{\e}C_{1,1}(\tau_\e f)[\tau_\e f,\tau_\e\overline\omega]+f'C_{1,1}(\tau_\e f)\Big[\tau_\e f, \frac{\tau_\e \overline\omega-\overline\omega}{\e}\Big] \\[1ex]
  &\hspace{0.424cm}+f' C_{1,1}(\tau_\e f)\Big[\frac{\tau_\e f-f}{\e},\overline\omega\Big]-f'C_{3,2}(f,\tau_\e f)\Big[\frac{\tau_\e f-f}{\e}, \tau_\e f+f,f,\overline\omega\Big].
  \end{align*}
  Since
  \begin{align*}
 & \tau_\e \overline\omega\underset{\e\to0}\longrightarrow\overline\omega\quad\text{in $H^1(\s)$,}&&  \tau_\e f\underset{\e\to0}\longrightarrow f\quad\text{in $H^2(\s)$,}\\[1ex]
 &  \frac{\tau_\e \overline\omega-\overline\omega}{\e}\underset{\e\to0}\longrightarrow -\overline\omega'\quad\text{in $L_2(\s)$,}&&  \frac{\tau_\e f-f}{\e}\underset{\e\to0}\longrightarrow -f'\quad\text{in $H^1(\s)$,}
  \end{align*}
  we may pass, in view of Lemma \ref{L:A1} $(i)-(iii)$, to the limit $\e\to0$ in the identity above to conclude that 
  \begin{align*}
  -\pi\frac{\tau_\e(\bB_3(f)[\overline\omega])-\bB_3(f)[\overline\omega]}{\e}&\underset{\e\to0}\longrightarrow
  C_{0,1}( f)[\overline\omega']-2C_{2,2}(f, f)[f', f,\overline\omega]+f''C_{1,1}(f)[f,\overline\omega]\\[1ex]
  &\hspace{0.624cm}+f'C_{1,1}( f) [  f, \overline\omega'] +f' C_{1,1}( f)[f',\overline\omega]-2f'C_{3,2}(f, f)[f', f,f,\overline\omega]
  \end{align*}
  in $L_2(\s).$ This proves that $\bB_3(f)[\overline\omega]\in H^1(\s),$ with 
\begin{equation}\label{Bprime} 
  \begin{aligned}
  \pi (\bB_3(f)[\overline\omega])'&=\pi\bB_3(f)[\overline\omega']-2C_{2,2}(f, f)[f', f,\overline\omega]+f''C_{1,1}(f)[f,\overline\omega]\\[1ex]
  &\hspace{0.424cm} +f' C_{1,1}( f)[f',\overline\omega]-2f'C_{3,2}(f, f)[f', f,f,\overline\omega].
  \end{aligned}
  \end{equation}
  Lemma \ref{L:A1} and the arguments in \cite[Section 5]{M16x}  finally lead us to
\begin{align}\label{EG4}
 \bB_3\in {\rm C}^\omega(H^2(\s), \kL(H^1(\s))),
\end{align}
and \eqref{DAR1} follows now from \eqref{EG1}-\eqref{EG3} and \eqref{EG4}. This completes the proof.
\end{proof}

We now study  the mapping properties of the operator $\bA$ introduced in \eqref{ADL}.

\begin{lemma}\label{L:31} 
 Let $r>3/2$ be given. It then holds  
 \begin{equation}\label{E:RA}
   \bA\in {\rm C}^\omega (H^r(\s),\kL(\wh L_2(\s)))\cap  {\rm C}^\omega (H^2(\s),\kL(\wh H^1(\s))).
 \end{equation}
\end{lemma}
\begin{proof}
Pick first $f,\, \overline\omega\in {\rm C}^\infty(\s)$ with $\langle \overline\omega\rangle=0$ and 
let $V_-\in {\rm C}(\ov\0_-)\cap {\rm C}^1(\0_-)$ be as defined in Lemma~\ref{L:220}.
It then holds
\[
2\langle V_-|_{[y=f(x)]}|(1,f')\rangle=(1+\bA(f))[\overline\omega]\in {\rm C}(\s),
\]
and therefore $\bA(f)[\overline\omega]\in \wh L_2(\s)$  if and only if $\langle V_-|_{[y=f(x)]}|(1,f')\rangle\in \wh L_2(\s)$.
The latter property follows from   the periodicity of $f$ and $P_-$, where $P_-\in {\rm C}^1(\ov\0_-)$ is given 
in \eqref{Pres}, with respect to $x$ and the relation
\[
V_-=-\frac{k}{\mu_-}(\nabla P_-+(0,\rho_-g))-(0,V)\quad\text{in $\ov\0_-$}.
\]

We are thus left to prove that 
 \begin{equation}\label{DIR2}
   \bA\in {\rm C}^\omega (H^r(\s),\kL(  L_2(\s)))\cap  {\rm C}^\omega (H^2(\s),\kL(  H^1(\s))).
 \end{equation}
 We proceed as in the previous lemma and write
 \begin{align}\label{FAA} 
 \bA(f)=-f'\bB_2(f)-\bB_1(f)+\bA_3(f),
 \end{align}
 where, using the notation introduced in Lemma \ref{L:A1}, we have
\[
\pi\bA_3(f)[\overline\omega]=f'C_{0,1}(f)[\overline\omega]-C_{1,1}(f)[f,\overline\omega].
\] 
Similarly as  in Lemma \ref{L:32}, we get   
\begin{align}\label{EGG2}
 \bA_3\in  {\rm C}^\omega(H^r(\s), \kL(L_2(\s)))\cap {\rm C}^\omega(H^2(\s), \kL(H^1(\s))),
\end{align}
with
\begin{equation}\label{Aprime} 
  \begin{aligned}
  \pi (\bA_3(f)[\overline\omega])'&=\pi\bA_3(f)[\overline\omega']+f''C_{0,1}(f)[\overline\omega]-2f'C_{2,2}(f, f)[f', f,\overline\omega]\\[1ex]
  &\hspace{0.424cm} - C_{1,1}( f)[f',\overline\omega]+2f'C_{3,2}(f, f)[f', f,f,\overline\omega].
  \end{aligned}
  \end{equation}
  The properties \eqref{EG1}, \eqref{EG2}, and \eqref{EGG2} combined imply \eqref{DIR2}, and the proof is complete.
\end{proof}

We now address the solvability of equation $\eqref{EP2}$.
To this end we first establish  the invertibility of $1+a_\mu\bA(f) $  in   $\kL(\wh L_2(\s))$.

\begin{thm}\label{T:I1}
 Let $r>3/2$ and $M>0$. Then, there exists a constant $C=C(M)>0$ such that 
 \begin{align}\label{EQ:I1}
  \|\overline\omega\|_2\leq C\|(\lambda-\bA(f))[\overline\omega]\|_2 
 \end{align}
for all $\lambda\in\R$ with $|\lambda|\geq1,$ $\overline\omega\in \wh L_2(\s),$ and $f\in H^r(\s)$ with $\|f'\|_{\infty}\leq M.$

In particular, $\{\lambda\in\R\,:\,|\lambda|\geq1\}$  is contained in the resolvent set  $\bA(f)\in \kL(\wh L_2(\s))$ for each $f\in H^r(\s)$. 
\end{thm}
\begin{proof} In view of Lemma \ref{L:31}, it suffices to establish the estimate \eqref{EQ:I1} for $\overline\omega,\, f\in {\rm C}^\infty(\s)$ with $\langle \overline\omega\rangle=0$ and $\|f'\|_{\infty}\leq M.$
Let   $V_\pm\in {\rm C}(\ov{\0}_\pm)\cap {\rm C}^1(\0_\pm)$ be as defined in Lemma \ref{L:220} and set
\begin{align}\label{TF}
 F_\pm:=(F^1_\pm, F^2_\pm):= V_\pm|_{[y=f(x)]}.
\end{align}
We denote by $\tau$ and $\nu$  the tangent and the  outward normal unit vectors at $\p\0_-$ and we decompose $F_\pm$ in tangential and normal components
$F_\pm=F_\pm^\tau+F_\pm^\nu,$ where   
\begin{equation}\label{NTC}
F_\pm^\tau=\mp\frac{(1\mp \bA(f))[\overline\omega]}{2(1+f'^2)}(1,f'),\qquad F^\nu_\pm=\frac{\bB(f)[\overline\omega]}{2(1+f'^2)} (-f',1),
\end{equation}
cf. \eqref{VFB}. Recalling the Lemmas \ref{L:32}-\ref{L:31}, we may view     $F_\pm^\tau$ and $F^\nu_\pm$ as being elements of $ L_2(\s,\R^2).$
  
  We next introduce the bilinear form 
 $\mathcal{B}:L_2(\mathbb{S},\mathbb{R}^2)\times L_2(\mathbb{S},\mathbb{R}^2)\to\mathbb{R}$ by the formula
\begin{equation*}
\mathcal{B}(F,G):=\int_{\s}G^2\langle F|(-f',1) \rangle +F^2\langle G|(-f',1) \rangle -\langle  F|G \rangle\, dx
\end{equation*}
 for $ F=(F^1,F^2),\, G=(G^1,G^2)\in L_2(\mathbb{S},\mathbb{R}^2).$  
Inserting the vector fields $F_\pm$ in \eqref{TF}, we find by using  Lebesgue's dominated convergence theorem, Stokes' formula,   and the Lemmas \ref{L:220}-\ref{L:221} 
that
\begin{align}\label{BB}
 \mathcal{B}(F_\pm,F_\pm)&=\int_{\G}\Big\langle\Big(\begin{array}{ccc}
                             2F_\pm^1F_\pm^2\\[0.5ex]
                             (F_\pm^2)^2-(F_\pm^1)^2
                             \end{array}\Big)\Big|\nu\Big\rangle\, d\sigma                            
                            =\mp\int_{\Omega_\pm}{\rm div\, }  \begin{pmatrix} 
                             2 V_\pm^1 V_\pm^2\\[0.5ex]
                             ( V_\pm^2)^2-( V_\pm^1)^2
                             \end{pmatrix} \, d(x,y)=0,
 \end{align}
 where $\Gamma$ denotes again a period of the graph $[y=f(x)].$
Moreover, in virtue of \eqref{NTC}, we may write  \eqref{BB} equivalently  as  
\begin{equation}\label{Ea1}
 \int_\mathbb{S}\frac{1}{1+f'^2}\Big[\big|\bB(f)[\overline\omega]\big|^2\mp2f'\big(\bB(f)[\overline\omega]\big)(1\mp\bA(f))[\overline\omega]-   \big|(1\mp\mathbb{A}(f))[\overline\omega]\big|^2\Big]\, dx=0,
 \end{equation}
 and, recalling that  $\|f'\|_{\infty}\leq M$, we infer from \eqref{Ea1} that  
\begin{align*}  
\|(1\pm\mathbb{A}(f))[\overline\omega]\|_2\leq C\| \bB(f)[\overline\omega]\|_2,
 \end{align*}
 with a positive constant $C=C(M)$.
In particular we get
\begin{align} \label{pm2}
  \|\overline\omega\|_2=\frac{1}{2}\|(1+\mathbb{A}(f))[\overline\omega]+(1-\mathbb{A}(f))[\overline\omega]\|_2\leq C\|\bB(f)[\overline\omega]\|_2.
\end{align}

Given $\lambda\in\R$ with $|\lambda|\geq1,$ it holds that 
\begin{align*}
  \big|(1\mp\mathbb{A}(f))[\overline\omega]\big|^2=  |(\lambda-\mathbb{A}(f))[\overline\omega] |^2
  -2(\lambda\mp1)\overline\omega(\lambda-\mathbb{A}(f))[\overline\omega] + (\lambda\mp 1)^2|\overline\omega|^2,
 \end{align*}
  and   eliminating the mixed term on the right hand side we obtain together with \eqref{Ea1}  that 
\begin{align*} 
\int_\mathbb{S}\frac{1}{1+f'^2}\Big[(\lambda^2-1)|\overline\omega|^2+ \big|\bB(f)[\overline\omega]\big|^2
 -|(\lambda-\mathbb{A}(f))[\overline\omega]|^2-2f' \big(\bB(f)[\overline\omega]\big)(\lambda-\mathbb{A}(f))[\overline\omega]\Big]\, dx=0,
 \end{align*}
 from where we conclude that  
 \[
 (\lambda^2-1)\|\overline\omega\|_2+\| \bB(f)[\overline\omega]\|_2\leq C \|(\lambda-\mathbb{A}(f))\overline\omega\|_2,
 \]
 with a constant  $C=C(M).$ 
 The latter estimate and  \eqref{pm2}   yield  \eqref{EQ:I1}.
 That   $\{\lambda\in\R\,:\,|\lambda|\geq1\}$ belongs to the resolvent set of $\bA(f)\in \kL(\wh L_2(\s))$ for all $f\in H^r(\s)$ is a straightforward consequence
  of \eqref{EQ:I1}, Lemma \ref{L:31},  and of the continuity method, 
 cf. e.g.\cite[Proposition I.1.1.1]{Am95}.
\end{proof}

The following remark is relevant in Section \ref{Sec:6} in the stability analysis of the Muskat problem.
\begin{rem}\label{R:2} The estimate 
 \begin{align*}  
\| \overline\omega \|_2\leq C\| \bB(f)[\overline\omega]\|_2
 \end{align*}
  derived in \eqref{pm2}  enables us to identify the equilibrium solutions to the Muskat problem \eqref{P} (see \eqref{EVE}) as being the  solutions  to the capillarity equation
 \begin{align}\label{CEqu}
 (\sigma \kappa(f)-\Theta f)'=0.
 \end{align}
\end{rem}

We now establish the invertibility of $1+a_\mu\bA(f) $  in the algebra  $\kL(\wh H^1(\s))$ under the assumption that  $f\in H^2(\s)$.

\begin{thm}\label{T:I2}
 Let   $M>0$. Then, there exists a constant $C=C(M)>0$ such that 
 \begin{align}\label{EQ:I2}
  \|\overline\omega\|_{H^1}\leq C\|(\lambda-\bA(f))[\overline\omega]\|_{H^1} 
 \end{align}
for all $\lambda\in\R$ with $|\lambda|\geq1,$ $\overline\omega\in \wh H^1(\s),$ and $f\in H^2(\s)$ with $\|f\|_{H^2}\leq M.$

In particular, $\{\lambda\in\R\,:\,|\lambda|\geq1\}$ is contained in the resolvent set  of $\bA(f)\in \kL(\wh H^1(\s))$ for each $f\in H^2(\s)$. 
\end{thm}
\begin{proof}
 Recalling \eqref{EQ:I1}, we are left  to estimate the term $\|((\lambda-\bA(f))[\overline\omega])'\|_{2} $ suitably.
 To this end, we infer from \eqref{FAA} and \eqref{Aprime} that 
 \begin{equation}\label{Aprime1} 
  \begin{aligned}
  (\bA(f)[\overline\omega])'&=\bA(f)[\overline\omega']+T^A_{\rm lot}(f)[\overline\omega],
  \end{aligned}
  \end{equation}
  where the operator $T^A_{\rm lot}(f)$  defined by
   \begin{equation}\label{Aprime2} 
  \begin{aligned}
  \pi T^A_{\rm lot}(f)[\overline\omega]&:= f''C_{0,1}(f)[\overline\omega]-2f'C_{2,2}(f, f)[f', f,\overline\omega] - C_{1,1}( f)[f',\overline\omega]\\[1ex]
  &\hspace{0.424cm}+2f'C_{3,2}(f, f)[f', f,f,\overline\omega] -\pi\big((f'\bB_2(f)[\overline\omega])'-f'\bB_2(f)[\overline\omega']\big)\\[1ex]
  &\hspace{0.424cm}-\pi\big((\bB_1(f)[\overline\omega])'-\bB_1(f)[\overline\omega']\big),
  \end{aligned}
  \end{equation}
   encompasses all lower order terms of  $(\bA_3(f)[\overline\omega])'$ with respect to $\overline\omega$ as, for each  $\tau\in(1/2,1)$ fixed, it holds
     \begin{equation}\label{Aprime3} 
  \|T^A_{\rm lot}(f)[\overline\omega]\|_2\leq C  \|\overline\omega\|_{H^\tau}\qquad\text{for all $\overline\omega\in \wh H^1(\s)$,}
  \end{equation}
  with $C=C(M)$.
   Indeed, letting $r:=(9-2\tau)/4,$ it  follows that $r\in(3/2,2) $ and $\tau\in(5/2-r,1)$, and  the estimates   \eqref{REF1} and \eqref{REF3} yield 
   \begin{equation*} 
   \|f''C_{0,1}(f)[\overline\omega]-2f'C_{2,2}(f, f)[f', f,\overline\omega] - C_{1,1}( f)[f',\overline\omega]
   +2f'C_{3,2}(f, f)[f', f,f,\overline\omega]\|_2\leq C\|\overline\omega\|_{H^\tau}.
   \end{equation*}
   Moreover, it follows from \eqref{EG1}-\eqref{EG2} and the compactness of the embedding $H^2(\s)\hookrightarrow H^r(\s)$ that also 
   \begin{equation*} 
   \|(f'\bB_2(f)[\overline\omega])'\|_2+\|(\bB_1(f)[\overline\omega])'\|_2\leq C\|\overline\omega\|_{H^\tau}.
   \end{equation*}
Finally, using  integration by parts in the formulas defining $\bB_1(f) $ and $\bB_2(f)$  we get
    \begin{equation*} 
   \|f'\bB_2(f)[\overline\omega'] \|_2+\|\bB_1(f)[\overline\omega']\|_2\leq C\|\overline\omega\|_{2},
   \end{equation*}
   and \eqref{Aprime3} follows.
   
   Invoking \eqref{EQ:I1} and \eqref{Aprime3} we find a constant  $c=c(M)\in(0,1)$ with
   \begin{align*}
  2\|(\lambda-\bA(f))[\overline\omega]\|_{H^1}& \geq   \|(\lambda-\bA(f))[\overline\omega]\|_{2}+\|((\lambda-\bA(f))[\overline\omega])'\|_{2} \\[1ex]
 &\geq \|(\lambda-\bA(f))[\overline\omega]\|_{2}+\|(\lambda-\bA(f))[\overline\omega']\|_{2}-\|T^A_{\rm lot}(f)[\overline\omega]\|_2\\[1ex]
 &\geq c\|\overline\omega\|_{H^1}- \frac{1}{c}\|\overline\omega\|_{H^\tau},
   \end{align*}
 and since by \eqref{IP}\footnote{Letting $[\,\cdot\,,\,\cdot\, ]_\theta $ denote complex interpolation functor, it is well-known that
\begin{align}\label{IP}
[H^{s_0}(\s),H^{s_1}(\s)]_\theta=H^{(1-\theta)s_0+\theta s_1}(\s),\qquad\theta\in(0,1),\, -\infty< s_0\leq s_1<\infty.
\end{align} } and Young's inequality
  $$\|\overline\omega\|_{H^\tau}\leq \|\overline\omega\|_2^{1-\tau}\|\overline\omega\|_{H^{1}}^\tau\leq \frac{c^2}{2}\|\overline\omega\|_{H^{1}}+C'\|\overline\omega\|_2,$$
  for some $C'=C'(M),$
  it follows that 
    \begin{align*}
  4\|(\lambda-\bA(f))[\overline\omega]\|_{H^1}\geq c\|\overline\omega\|_{H^1}- \frac{2C'}{c}\|\overline\omega\|_{2}.
   \end{align*}
   This estimate together with \eqref{EQ:I1} leads us to \eqref{EQ:I2} and the proof is complete.
\end{proof}

We conclude this section by considering the adjoints  of the operators defined in   \eqref{ADL} and \eqref{BDL}. 
Firstly we establish a similar estimate as in Theorem \ref{T:I2} for the operator $P(\bA(f))^*,$ where $(\bA(f))^*$ is the  double layer potential, cf. \eqref{DL},
and where  $P:L_2(\s)\to\wh L_2(\s)$, with $Ph:=h-\langle h\rangle,$ denotes the orthogonal projection on $\wh L_2(\s)$.
This estimate is important later  on in the uniqueness proof of Theorem \ref{MT:1}.  
Recalling that $(\bA(f))^*\in\kL(L_2(\s))$ is the $L_2$-adjoint of $\bA(f)\in\kL(L_2(\s))$, we obtain for $\overline\omega,\,\xi\in \wh L_2(\s)$ that
\[
\langle \bA(f)[\overline\omega],\xi \rangle_2=\langle \overline\omega,(\bA(f))^*[\xi] \rangle_2=\langle \overline\omega, P(\bA(f))^*[\xi] \rangle_2,
\]
meaning that the adjoint  $\big(\wh{\bA(f)}\big)^*:=\big(\bA(f)|_{\wh L_2(\s)}\big)^*\in\kL(\wh L_2(\s))$ is given by
$\big(\wh{\bA(f)}\big)^*=P(\bA(f))^*.$

\begin{thm}\label{T:I3}
 Let   $M>0$. Then, there exists a constant $C=C(M)>0$ such that 
 \begin{align}\label{EQ:I3}
  \|\xi\|_{H^1}\leq C\|\big(\lambda-\big(\wh{\bA(f)}\big)^*\big)[\xi]\|_{H^1} 
 \end{align}
for all $\lambda\in\R$ with $|\lambda|\geq1,$ $\xi\in \wh H^1(\s),$ and $f\in H^2(\s)$ with $\|f\|_{H^2}\leq M.$

In particular, $\{\lambda\in\R\,:\,|\lambda|\geq1\}$ is contained in the resolvent set  of $\big(\wh{\bA(f)}\big)^*\in \kL(\wh H^1(\s))$ for each $f\in H^2(\s)$. 
\end{thm}
\begin{proof}
Let $M>0$.  
Taking advantage of the fact that $\lambda-\big(\wh{\bA(f)}\big)^*$ is the $\kL(\wh L_2(\s))$-adjoint of $\lambda-\bA(f)$ for each
 $\lambda\in\R$ and $f\in H^r(\R)$, $r>3/2$, it follows from \eqref{EQ:I1}   there exists a constant $C=C(M)$ such that 
 \begin{align}\label{EQ:I3a}
  \|\xi\|_2\leq C\|\big(\lambda-\big(\wh{\bA(f)}\big)^*\big)[\xi]\|_2 
 \end{align}
for all $\lambda\in\R$ with $|\lambda|\geq1,$ $\xi\in \wh H^1(\s),$ and $f\in H^2(\s)$ with $\|f\|_{H^2}\leq M.$
In order to show that $\big(\wh{\bA(f)}\big)^*[\xi]\in \wh H^1(\s)$,  we note that
\[
\big(\wh{\bA(f)}\big)^*[\xi]=(\bA(f))^*[\xi]-\langle (\bA(f))^*[\xi]\rangle=\bB_1(f)[\xi]+\bB_2(f)[f'\xi]+ \bA_{3,*}(f)[\xi] -\langle (\bA(f))^*[\xi]\rangle,
\] 
where $\bB_1(f)$ and $\bB_2(f)$ are introduced in the proof of Lemma \ref{L:32} and where
\begin{align*}
\pi\bA_{3,*}(f)[\xi]:=C_{1,1}(f)[f,\xi]-C_{0,1}(f)[f'\xi].
\end{align*}
The arguments used to derive \eqref{Bprime} show that $\bA_{3,*}(f)[\xi]\in H^1(\s)$ with
\begin{align*}
\pi(\bA_{3,*}(f)[\xi])'&=\pi\bA_{3,*}(f)[\xi']+C_{1,1}(f)[f',\xi]-C_{0,1}(f)[f''\xi]+2C_{2,2}(f,f)[f',f,f'\xi]\\[1ex]
&\hspace{0.424cm}-2C_{3,2}(f,f)[f',f,f,\xi],
\end{align*}
and together with  \eqref{EG1}-\eqref{EG2} we conclude that indeed $\big(\wh{\bA(f)}\big)^*[\xi]\in \wh H^1(\s)$.
Proceeding as in  Theorem~\ref{T:I2}, we may write  
\begin{equation*} 
  \begin{aligned}
  \big(\big(\wh{\bA(f)}\big)^*[\xi]\big)'&=(\bA(f))^*[\xi']+T^{A^*}_{\rm lot}(f)[\xi]=\big(\wh{\bA(f)}\big)^*[\xi']+\langle(\bA(f))^*[\xi']\rangle+T^{A^*}_{\rm lot}(f)[\xi],
  \end{aligned}
  \end{equation*}
  with   
   \begin{equation*}
  \begin{aligned}
  \pi T^{A^*}_{\rm lot}(f)[\xi]&:= C_{1,1}(f)[f',\xi]-C_{0,1}(f)[f''\xi]+2C_{2,2}(f,f)[f',f,f'\xi]\\[1ex]
  &\hspace{0.424cm}-2C_{3,2}(f,f)[f',f,f,\xi] +\pi\big((\bB_2(f)[f'\xi])'-\bB_2(f)[(f'\xi)']\big)\\[1ex]
  &\hspace{0.424cm}+\pi\big((\bB_1(f)[\xi])'-\bB_1(f)[\xi']\big) 
  \end{aligned}
  \end{equation*}
 satisfying
     \begin{equation}\label{Aprime3*} 
  \|T^{A^*}_{\rm lot}(f)[\xi]\|_2\leq C  \|\xi\|_{H^\tau}\qquad\text{for all $\xi\in \wh H^1(\s)$,}
  \end{equation}
   for any fixed  $\tau\in(1/2,1)$ and  with a constant  $C=C(M)$.
   Moreover, since $\bA(f)[1]\in H^1(\s)$,  it follows that 
\begin{equation}\label{Aprime4*}    
   \begin{aligned}
   |\langle(\bA(f))^*[\xi']\rangle|&\leq |\langle(\bA(f))^*[\xi'],1\rangle_2|=|\langle \xi',\bA(f)[1]\rangle_2|=\Big|\int_{-\pi}^\pi\xi' \bA(f)[1]\, dx\Big|\\[1ex]
   &  =\Big|\int_{-\pi}^\pi\xi (\bA(f)[1])'\, dx\Big|\leq \|\xi\|_2\|\bA(f)[1]\|_{H^1}\leq C\|\xi\|_2,
   \end{aligned}
   \end{equation}
   again  with $C=C(M)$.
The desired claim \eqref{EQ:I3} follows now from \eqref{EQ:I3a}, \eqref{Aprime3*}, and \eqref{Aprime4*} by arguing as  in Theorem \ref{T:I2}.
\end{proof}

Finally, given $f\in H^r(\s)$, $r>3/2$,   let $(\bB(f))^*\in\kL(L_2(\s))$ denote the  adjoint of $\bB(f)\in \kL(L_2(\s))$.
The next lemma is also used later on in the uniqueness proof of Theorem \ref{MT:1}.  

\begin{lemma}\label{L:33}
Given $M>0$, there exists a constant $C=C(M)$ such that for all $f\in H^2(\s)$ with $\|f\|_{H^2}\leq M$ it holds that $(\bB(f))^*\in\kL(H^1(\s))$ and
\[
\|(\bB(f))^*\|_{\kL(H^1(\s))}\leq C.
\]
\end{lemma}
\begin{proof}
Given $f\in H^2(\s)$, it is not difficult to show that
\[
(\bB(f))^*[\xi]=-\bB_1(f)[f'\xi]+\bB_2(f)[\xi]-\frac{1}{\pi}\big(C_{0,1}(f)[\xi]+C_{1,1}(f)[f,f'\xi]\big),\qquad \xi\in L_2(\s).
\] 
The desired estimate follows now by arguing as in Lemma \ref{L:32}.
\end{proof}

 %%%%%%%%%%%%%%%%%%%%%%%%%%%%%%%%%%%%%%%%%%%%%%%%%%%%%%%%%%%%%%%%%%%
%%%%%%%%%%%%%%%%%%%%%%%%%%%%%%%%%%%%%%%%%%%%%%%%%%%%%%%%%%%%%%%%%%%%
%%%%%%%%%%%%%%%%%%%%%%%%%%%%%%%%%%%%%%%%%%%%%%%%%%%%%%%%%%%%%%%%%%%%
%%%%%%%%%%%%%%%%%%%%%%%%%%%%%%%%%%%%%%%%%%%%%%%%%%%%%%%%%%%%%%%%%%%
%%%%%%%%%%%%%%%%%%%%%%%%%%%%%%%%%%%%%%%%%%%%%%%%%%%%%%%%%%%%%%%%%%%%
%%%%%%%%%%%%%%%%%%%%%%%%%%%%%%%%%%%%%%%%%%%%%%%%%%%%%%%%%%%%%%%%%%%%
\section{The  Muskat problem with surface tension effects} \label{Sec:4}
 %%%%%%%%%%%%%%%%%%%%%%%%%%%%%%%%%%%%%%%%%%%%%%%%%%%%%%%%%%%%%%%%%%%
%%%%%%%%%%%%%%%%%%%%%%%%%%%%%%%%%%%%%%%%%%%%%%%%%%%%%%%%%%%%%%%%%%%%
%%%%%%%%%%%%%%%%%%%%%%%%%%%%%%%%%%%%%%%%%%%%%%%%%%%%%%%%%%%%%%%%%%%%
%%%%%%%%%%%%%%%%%%%%%%%%%%%%%%%%%%%%%%%%%%%%%%%%%%%%%%%%%%%%%%%%%%%
%%%%%%%%%%%%%%%%%%%%%%%%%%%%%%%%%%%%%%%%%%%%%%%%%%%%%%%%%%%%%%%%%%%%
%%%%%%%%%%%%%%%%%%%%%%%%%%%%%%%%%%%%%%%%%%%%%%%%%%%%%%%%%%%%%%%%%%%%
In this section we study the Muskat problem in the case when surface tension effects are included, that is for $\sigma>0$.
The main goal of this section  is to prove Theorem \ref{MT:1} which is postponed to the end of the section.
As a first step  we shall take advantage of the results established in the previous sections  to reexpress the contour integral formulation \eqref{P}
as an abstract   evolution equation of the form
\begin{align}\label{MPS}
\dot f(t) =\Phi_\sigma(f(t) )[f(t)],\quad t>0,\qquad f(0)=f_0,
\end{align}
 with an operator $[f\mapsto \Phi_\sigma(f)]: H^2(\s)\to\kL(H^3(\s), L_2(\s))$  defined in \eqref{PHS}.
 The quasilinear character of the  contour integral equation for $\sigma>0$ -- which is not obvious because of the coupling in $\eqref{EP}_2$ --   
 is expressed  in \eqref{MPS} by the fact that $\Phi_\sigma$  is nonlinear with respect to the first variable $f\in H^2(\s)$, but is linear with respect to the second variable $f\in H^3(\s)$ which corresponds to the third spatial derivatives of the function $f=f(t,x)$ in the curvature term in $\eqref{EP}_2$.        
A central part of the analysis in this section is devoted to showing that \eqref{MPS} is a parabolic problem in the
 sense that $\Phi_\sigma(f)$ -- viewed as an unbounded operator on $L_2(\s)$
with definition domain $H^3(\s)$ -- is, for each $f\in H^2(\s)$, the generator of a strongly  continuous and analytic semigroup in $\kL(L_2(\s))$, which we denote by writing 
\begin{align}\label{GP}
-\Phi_\sigma(f)\in\kH(H^3(\s), L_2(\s)).
\end{align}
This property   needs to be verified before applying the abstract quasilinear parabolic theory outlined in \cite{Am93, Am95, Am86, Am86b, Am88} (see also \cite{MW18x})
 in the particular  context of \eqref{MPS}.

We begin   by solving the equation $\eqref{EP}_2$ for $\overline\omega$. We shall rely on the invertibility properties provided in Theorems \ref{T:I1} and \ref{T:I2}
and the fact  that the Atwood number satisfies $|a_\mu|<1$.
 In  order to disclose the  quasilinear structure of  the Muskat problem with surface tension  we address at this point the solvability of the   equation
\begin{align}\label{CE}
 (1+a_\mu \bA(f))[\overline\omega] =b_\mu\Big[\sigma\frac{h'''}{(1+f'^2)^{3/2}}-3\sigma\frac{f'f''h''}{(1+f'^2)^{3/2}}-\Theta h'\Big] ,
\end{align}
which  for $h=f$   coincides, up to a factor of $2$, with $\eqref{EP}_2$.
The quasilinearity of the curvature term is essential  here.
For the sake of brevity  we introduce
\begin{equation}\label{bmu}
b_\mu:=\displaystyle\frac{k}{\mu_-+\mu_+}.
\end{equation}
Since  the values of $\sigma>0$ and $b_\mu>0$ are not  important in the proof of Theorem \ref{MT:1} we set in this section
\[
b_{\mu}=\sigma=1.
\]
The solvability result in Proposition \ref{P:I1} $(a)$ below is the main step towards writing \eqref{P} in the form \eqref{MPS}.
The decomposition of the solution operator provided at Proposition \ref{P:I1} $(b)$ is essential later on in the proof of the generator property, as it enables us to
 use integration by parts when estimating some terms of leading order.  

\begin{prop}\label{P:I1}
\begin{itemize}
\item[$(a)$]
Given $f\in H^2(\s)$ and $h\in H^3(\s)$, the function 
 \[
\overline\omega(f)[h]:=(1+a_\mu\bA(f))^{-1}\Big[\frac{h'''}{(1+f'^2)^{3/2}}-3\frac{f'f''h''}{(1+f'^2)^{3/2}}-\Theta h'\Big]
 \]
 is the unique solution to \eqref{CE} in $\wh L_2(\s)$ and 
\begin{align}\label{P:I1R}
\overline\omega\in {\rm C}^\omega(H^2(\s), \kL(H^3( \s), \wh L_2(\s))).
\end{align}
\item[$(b)$]   Given $f\in H^2(\s) $  and $h\in H^3(\s)$, let  
\begin{align*}
 \overline\omega_1(f)[h]&:=(1+a_\mu\bA(f))^{-1}\Big[\frac{h''}{(1+f'^2)^{3/2}}-\Big\langle\frac{h''}{(1+f'^2)^{3/2}}\Big\rangle\Big],\\[1ex]
 \overline\omega_2(f)[h]&:=(1+a_\mu\bA(f))^{-1}\big[-\Theta h'+a_\mu T_{\rm lot}^A(f)[\overline\omega_1(f)[h]]\big],
\end{align*}
where  $T_{\rm lot}^A $ is defined in \eqref{Aprime2}. 
Then:\\[-2ex]
\begin{itemize}
 \item[$(i)$] $\overline\omega_1\in {\rm C}^\omega(H^2(\s), \kL(H^3(\s),\wh  H^1(\s)))$ and $\overline\omega_2\in {\rm C}^\omega(H^2(\s), \kL(H^3(\s), \wh L_2(\s)));$
 \item[$(ii)$] \[\overline\omega(f)=\frac{d}{dx}\circ\overline\omega_1(f)+\overline\omega_2(f); \]
  \item[$(iii)$] Given $\tau\in (1/2,1)$, there exists a constant $C$ such that 
\begin{equation}\label{p1} 
\begin{aligned}
&\|\overline\omega_1(f)[h]\|_2\leq C\|h\|_{H^2}\\[1ex] 
&\|\overline\omega_1(f)[h]\|_{H^{\tau}}+\|\overline\omega_2(f)[h]\|_2\leq C\|h\|_{H^{2+\tau}} 
\end{aligned}\qquad\text{for all $h\in H^3(\s)$.}
\end{equation}
\end{itemize}
\end{itemize}
\end{prop}
 \begin{proof} 
 Observing that the right hand side of \eqref{CE} belongs to $\wh L_2(\s)$, the claim $(a)$   follows  from Theorems \ref{T:I1}.
 
 In order to prove $(b)$ we first note that 
 \begin{align*}
  \left[f\mapsto \Big[h\mapsto\frac{h''}{(1+f'^2)^{3/2}}-\Big\langle\frac{h''}{(1+f'^2)^{3/2}}\Big\rangle\Big]\right]\in {\rm C}^\omega(H^2(\s),\kL(H^3(\s),\wh H^1(\s))),
 \end{align*}
  and since by Theorem \ref{T:I2} $[f\mapsto (1+a_\mu\bA(f))^{-1}]\in {\rm C}^\omega(H^2(\s),\kL(\wh H^1(\s))),$ we conclude that 
  $\overline\omega_1$ is well-defined together with $\overline\omega_1\in {\rm C}^\omega(H^2(\s), \kL(H^3(\s), \wh H^1(\s)))$.
  Recalling \eqref{Aprime1} and \eqref{Aprime2}, it holds
  \begin{align*}
  (1+a_\mu\bA(f))[\overline\omega(f)[h]-(\overline\omega_1(f)[h])']=-\Theta h'+a_\mu T_{\rm lot}^A(f)[\overline\omega_1(f)[h]]=(1+a_\mu\bA(f))[\overline\omega_2(f)[h]].
  \end{align*}
  This proves  $\overline\omega_2\in {\rm C}^\omega(H^2(\s), \kL(H^3(\s), \wh L_2(\s))) $ together with the claim $(ii)$.
  
As for $(iii)$, we note that  the Theorems \ref{T:I1} and \ref{T:I2} imply that 
  \begin{equation*}
\|\overline\omega_1(f)[h]\|_2\leq C\|h\|_{H^2}\quad\text{and}\quad\|\overline\omega_1(f)[h]\|_{H^{1}}\leq C\|h\|_{H^{3}}\qquad \text{for all $h\in H^3(\s)$,}
\end{equation*}
and the  estimate  
$\|\overline\omega_1(f)[h]\|_{H^{\tau}}\leq C\|h\|_{H^{2+\tau}}$, $h\in H^3(\s)$, follows from the latter via interpolation. 
Finally, recalling Theorem \ref{T:I1} and \eqref{Aprime3}, it holds
\begin{align*}
\|\overline\omega_2(f)[h]\|_2&\leq C(\|h\|_{H^1}+\|T_{\rm lot}^A(f)[\overline\omega_1(f)[h]]\|_2)\leq C(\|h\|_{H^1}+\| \overline\omega_1(f)[h] \|_{H^\tau})\\[1ex]
&\leq C\|h\|_{H^{2+\tau}},
\end{align*}
and the proof is complete.
 \end{proof}
 
 Proposition \ref{P:I1} enables us to recast the contour integral formulation  \eqref{P} of the Muskat problem with surface tension
 as the abstract quasilinear evolution problem \eqref{MPS}, where  
\begin{equation}\label{PHS}
\Phi_\sigma(f)[h]:=\bB(f)[\overline\omega(f)[h]]\qquad\text{for $f\in H^2(\s)$ and $h\in H^3(\s)$.}
\end{equation}
Proposition \ref{P:I1} and Lemma \ref{L:32} imply that 
\begin{equation}\label{RegS}
\Phi_\sigma\in {\rm C}^\omega(  H^2(\s), \kL(  H^3(\s),  \wh L_2(\s)))\cap{\rm C}^\omega(H^2(\s), \kL(H^3(\s),  L_2(\s))).
\end{equation}

In the following $f\in H^2(\s)$ is kept fixed. 
In order to establish the generator property \eqref{GP} for $\Phi_\sigma(f)$ it is suitable to decompose this operator as the sum
\[
\Phi_\sigma(f)=\Phi_{\sigma,1}(f)+\Phi_{\sigma,2}(f),
\]
where 
\begin{align*}
\Phi_{\sigma,1}(f)[h]=\bB(f)[(\overline\omega_1(f)[h])']\qquad\text{and}\qquad \Phi_{\sigma,2}(f)[h]=\bB(f)[ \overline\omega_2(f)[h]].
\end{align*}
The operator $\Phi_{\sigma,1}(f)$ can be viewed as the leading order part of $\Phi_\sigma(f),$ while $\Phi_{\sigma,2}(f)$ is a lower order perturbation, 
see the proof of Theorem \ref{T:GP1}.
We study first the leading order part $\Phi_{\sigma,1}(f).$
In order to establish \eqref{GP} we follow a   direct and self-contained   approach pursued  previously in \cite{E94,ES95, ES97} and generalized more 
recently in \cite{EMW18, M16x, M17x, MM17x} in the context of the Muskat problem. 
 The proof of \eqref{GP} uses a localization procedure which necessitates the introduction of certain partitions of unity for the unit circle.
 
To proceed,  we choose for each  integer $p\geq 3$  a set   $\{\pi_j^p\,:\,{1\leq j\leq 2^{p+1}}\}\subset  {\rm C}^\infty(\s,[0,1])$, called {\em $p$-partition of unity}, such that
\begin{align*}
\bullet\,\,\,\, \,\, & \text{$\supp \pi_j^p=\bigcup _{n\in\Z} \big(2\pi n+ I_j^p\big)$ and $I_j^p:=[j-5/3,j-1/3] \frac{\pi}{2^p};$}\\[1ex]
 \bullet\,\,\,\, \,\, & \text{ $\sum_{j=1}^{2^{p+1}}\pi_j^p=1 $ in ${\rm C}(\s)$.}
\end{align*}
To each such $p$-partition of unity we associate a  set $\{\chi_j^p\,:\,{1\leq j\leq 2^{p+1}}\}\subset  {\rm C}^\infty(\s,[0,1])$ satisfying
\begin{align*}
\bullet\,\,\,\, \,\, & \text{$\supp \chi_j^p=\bigcup_{n\in\Z} \big(2\pi n+ J_j^p\big)$ with  $I_j^p\subset J_j^p:=[j-8/3,j+2/3] \frac{\pi}{2^p}$;}\\[1ex]
 \bullet\,\,\,\, \,\, & \text{ $\chi_j^p=1$ on $I_j^p$.}
\end{align*}

As a further step we  introduce  the continuous path 
\[[\tau\mapsto \Phi_{\sigma,1}(\tau f)]:[0,1]\to \kL(H^3(\s), L_2(\s)),\]
which connects the operator $\Phi_{\sigma,1}(f)$ with the Fourier multiplier 
\begin{align*}
 \Phi_{\sigma,1}(0)[h](x)=\bB(0)[h'''](x)=\frac{1}{2\pi}\PV\int_{-\pi}^\pi\frac{h'''(x-s)}{t_{[s]}}\, ds= H[h'''](x),
\end{align*}
where $H$ denotes as usually the periodic Hilbert transform.
Since $H$ is the Fourier multiplier with symbol $(-i\sign(k))_{k\in\Z},$ it follows that $\Phi_{\sigma,1}(0)=- (\p_x^4)^{3/4},$ that is  
 the symbol of $\Phi_{\sigma,1}(0)$ is $(- |k|^3)_{k\in\Z}.$
In Theorem \ref{T:K1}, which is the key argument in the proof of \eqref{GP}, we establish some  commutator type estimates relating
   $\Phi_{\sigma,1}(\tau f)$   locally to some  explicit Fourier multipliers.
   The proof of this result is quite technical and lengthy and uses to a large extent the outcome of Lemma \ref{L:A1}.

\begin{thm}\label{T:K1} 
Let  $f\in H^2(\s)$   and     $\mu>0$ be given.
Then, there exist $p\geq3$, a  $p$-partition of unity  $\{\pi_j^p\,:\, 1\leq j\leq 2^{p+1}\} $, a constant $K=K(p)$, and for each  $ j\in\{1,\ldots,2^{p+1}\}$ and $\tau\in[0,1]$ there 
exist   operators $$\bA_{j,\tau}\in\kL(H^3(\s), L_2(\s))$$
 such that 
 \begin{equation}\label{DE3}
  \|\pi_j^p\Phi_{\sigma,1}(\tau f)[h]-\bA_{j,\tau}[\pi^p_j h]\|_{2}\leq \mu \|\pi_j^p h\|_{H^3}+K\|  h\|_{H^{11/4}}
 \end{equation}
 for all $ j\in\{1,\ldots, 2^{p+1}\}$, $\tau\in[0,1],$ and  $h\in H^3(\s)$. 
 The operator  $\bA_{j,\tau}$ is defined  by 
  \begin{align} 
 \bA_{j,\tau }:=&-\frac{1}{(1+f_\tau'^2(x_j^p))^{3/2}}(\p_x^4)^{3/4},\label{FM1}
 \end{align}
 where $x_j^p\in I_j^p$ is arbitrary, but fixed.
\end{thm}
\begin{proof} 
 Let   $p\geq 3$ be an integer which we fix later on in this proof and let  $\{\pi_j^p\,:\, 1\leq j\leq 2^{p+1}\}$ be a $p$-partition of unity, respectively, 
  let $\{\chi_j^p\,:\, 1\leq j\leq 2^{p+1}\} $ be a family associated to this $p$-partition of unity as described above.
 In the following, we denote by $C$   constants which are
independent of $p\in\N$, $h\in H^3(\s)$, $\tau\in [0,1]$, and $j \in \{1, \ldots, 2^{p+1}\}$, while the constants  denoted by $K$ may depend only on $p.$\medskip

\noindent{\em Step 1: The lower order terms.} Using the decomposition provided in the proof of Lemma \ref{L:32} for the operator $\bB$, we write
\begin{align}\label{MMM0}
\Phi_{\sigma,1}(\tau f)[h]&=f_\tau'\bB_1(f_\tau)[\overline\omega_1']-\bB_2(f_\tau)[\overline\omega_1']+\frac{1}{\pi}C_{0,1}(f_\tau)[\overline\omega_1']+\frac{1}{\pi}f_\tau'C_{1,1}(f_\tau)[f_\tau,\overline\omega_1'],
\end{align}
 where, for the sake of brevity, we have set 
 \[\text{$\overline\omega_1:=\overline\omega_1(\tau f)[h]$ \qquad and\qquad $f_\tau:=\tau f.$}\]
 Using integration by parts, we infer from \eqref{p1} that
 \begin{align}\label{MMM1}
 \|\pi_j^p\big[f_\tau'\bB_1(f_\tau)[\overline\omega_1']-\bB_2(f_\tau)[\overline\omega_1']\big]\|_2\leq C\|\overline\omega_1\|_2\leq C\|h\|_{H^2},
 \end{align}
 and we are left to consider the last two terms in \eqref{MMM0}.\medskip

\noindent{\em Step 2: The first leading  order term.}
Given $1\leq j\leq 2^{p+1}$ and $\tau \in[0,1]$, let
\[
\bA_{j,\tau}^1:=-\frac{f_\tau'^2(x_j^p)}{(1+f_\tau'^2(x_j^p))^{5/2}}(\p_x^4)^{3/4},
\]
where $x_j^p\in I_j^p$.
In this step we show that if $p$ is sufficiently large, then
  \begin{equation}\label{MMM2}
  \| \pi_j^p f_\tau'C_{1,1}(f_\tau)[f_\tau,\overline\omega_1']-\pi\bA^1_{j,\tau}[\pi^p_j h]|_{2}\leq \frac{\mu}{2} \|\pi_j^p h\|_{H^3}+K\|  h\|_{H^{11/4}}
 \end{equation}
 for all $ j\in\{1,\ldots, 2^{p+1}\}$, $\tau\in[0,1],$ and  $h\in H^3(\s)$. 
 To this end we write
 \[
 \pi_j^pf_\tau'C_{1,1}(f_\tau)[f_\tau,\overline\omega_1']-\pi\bA^1_{j,\tau}[\pi^p_j h]=T_1[h]+T_2[h]+T_3[h],
 \]
 where
 \begin{align*}
 T_1[h]&:=\pi_j^pf_\tau'C_{1,1}(f_\tau)[f_\tau,\overline\omega_1']- f_\tau'(x_j^p)C_{1,1}(f_\tau)[f_\tau,\pi_j^p\overline\omega_1'],\\[1ex]
  T_2[h]&:=  f_\tau'(x_j^p)C_{1,1}(f_\tau)[f_\tau,\pi_j^p\overline\omega_1']- \frac{f_\tau'^2(x_j^p)}{1+f_\tau'^2(x_j^p)}C_{0,0}[\pi_j^p\overline\omega_1'],\\[1ex]
  T_3[h]&:=   \frac{f_\tau'^2(x_j^p)}{1+f_\tau'^2(x_j^p)}\Big[C_{0,0}[\pi_j^p\overline\omega_1']+\frac{\pi}{(1+f_\tau'^2(x_j^p))^{3/2}}(\p_x^4)^{3/4}[\pi_j^p h]\Big].
 \end{align*}
 
 We first consider $T_1[h].$ Recalling that $\chi_j^p\pi_j^p=\pi_j^p$, algebraic manipulations lead us to
 \begin{align*}
 T_1[h]&:=\chi_j^p  (f'_\tau-f'_\tau(x_j^p))C_{1,1}(f_\tau)[f_\tau,\pi_j^p\overline\omega_1']+T_{11}[h],
 \end{align*}
and the term $T_{11}[h]$ may be expressed, after  integrating  by parts, as
 \begin{align*}
 T_{11}[h]&=f_{\tau}'C_{1,1}(f_\tau)  [f_{\tau},(\pi_j^p)'\overline\omega_1 ]-2f_{\tau}'C_{2,1}( f_\tau)  [\pi_j^p,f_{\tau},\overline\omega_1]+f_\tau'C_{1,1}(f_\tau)  [\pi_j^p,f_\tau'\overline\omega_1]\\[1ex]
 &\hspace{0.424cm}-2 f_{\tau}'C_{3,2}(f_\tau, f_\tau)  [\pi_j^p,f_\tau,f_\tau,f_\tau' \overline\omega_1]+2 f_\tau'C_{4,2}(f_\tau, f_\tau)  [\pi_j^p,f_\tau,f_\tau,f_\tau, \overline\omega_1]\\[1ex]
 &\hspace{0.424cm}+(f_\tau'(x_j^p)-f_\tau')(1-\chi_j^p)C_{1,1}(f_\tau)  [f_\tau,(\pi_j^p)'\overline\omega_1]\\[1ex]
 &\hspace{0.424cm}+(f_\tau'(x_j^p)-f_\tau')C_{1,1}(f_\tau)  [\chi_j^p,\pi_j^p f_\tau'\overline\omega_1]-2(f_\tau'(x_j^p)-f_\tau')C_{2,1}( f_\tau)  [\chi_j^p,f_\tau,\pi_j^p \overline\omega_1]\\[1ex]
 &\hspace{0.424cm}-2 (f_\tau'(x_j^p)-f_\tau')C_{3,2}(f_\tau, f_\tau)  [\chi_j^p,f_\tau,f_\tau,\pi_j^p f_\tau' \overline\omega_1]\\[1ex]
 &\hspace{0.424cm}+2(f_\tau'(x_j^p)-f_\tau')C_{4,2}(f_\tau, f_\tau)  [\chi_j^p,f_\tau,f_\tau,f_\tau,\pi_j^p \overline\omega_1].
\end{align*}
Lemma \ref{L:A1} $(i)$  together with \eqref{p1} yields
 \begin{align}\label{MM3a-}
 \|T_{11}[h]\|_2\leq K\|\overline\omega_1\|_2\leq K\|h\|_{H^2},
 \end{align}
 and
 \begin{align}\label{MM3a}
 \|C_{1,1}(f_\tau)[f_\tau,\pi_j^p\overline\omega_1']\|_2\leq C\|\pi_j^p\overline\omega_1'\|_2.
 \end{align}
Hence,  we need to estimate the term $\|\pi_j^p\overline\omega_1'\|_2$ appropriately. 
The relation \eqref{Aprime1} and the definition of $\overline\omega_1 $ (see Proposition \ref{P:I1} $(b)$), yield
\begin{equation}\label{MM3b}
\begin{aligned}
(1+a_\mu\bA(f_\tau))[(\pi_j^p\overline\omega_1)']&=\frac{\pi_j^ph'''}{(1+f_\tau'^2)^{3/2}}-\frac{3\pi_j^pf_\tau'f_\tau''h''}{(1+f_\tau'^2)^{5/2}}-a_\mu\pi_j^pT_{\rm lot}^A(f_\tau)[\overline\omega_1]\\[1ex]
&\hspace{0.424cm}+(1+a_\mu\bA(f_\tau))[(\pi_j^p)'\overline\omega_1]+a_\mu\big(\bA(f_\tau)[\pi_j^p\overline\omega_1']-\pi_j^p\bA(f_\tau)[\overline\omega_1']\big),
\end{aligned}
\end{equation}
and   the  last  term  on the right hand side of \eqref{MM3b} can be recast    as
\begin{align*}
\pi(\bA(f_\tau)[\pi_j^p\overline\omega_1']-\pi_j^p\bA(f_\tau)[\overline\omega_1'])
&= \pi f_\tau'(\pi_j^p\bB_2(f_\tau)[\overline\omega_1']-\bB_2(f_\tau)[\pi_j^p\overline\omega_1'])\\[1ex]
& \hspace{0.424cm}+\pi(\pi_j^p\bB_1(f_\tau)[\overline\omega_1']-\bB_1(f_\tau)[\pi_j^p\overline\omega_1'])\\[1ex]
&\hspace{0.424cm}+f_\tau'(C_{0,1}(f_\tau)[\pi_j^p\overline\omega_1']- \pi_j^pC_{0,1}(f_\tau)[\overline\omega_1'])\\[1ex]
&\hspace{0.424cm}-(C_{1,1}(f_\tau)[f_\tau,\pi_j^p\overline\omega_1']- \pi_j^pC_{1,1}(f_\tau)[f_\tau,\overline\omega_1']).
\end{align*}
Integration  by parts and   Lemma  \ref{L:A1} $(i)$ lead us to 
\begin{equation}\label{MM3b'}
\begin{aligned}
&\| f_\tau'\bB_2(f_\tau)[\pi_j^p\overline\omega_1']\|_2+\|f_\tau'\pi_j^p\bB_2(f_\tau)[\overline\omega_1']\|_2 +\|\pi_j^p\bB_1(f_\tau)[\overline\omega_1']\|_2+\|\bB_1(f_\tau)[\pi_j^p\overline\omega_1']\|_2\\[1ex]
&\hspace{0.424cm}+\|f_\tau'(C_{0,1}(f_\tau)[\pi_j^p\overline\omega_1']- \pi_j^pC_{0,1}(f_\tau)[\overline\omega_1'])\|_2+\|C_{1,1}(f_\tau)[f_\tau,\pi_j^p\overline\omega_1']- \pi_j^pC_{1,1}(f_\tau)[f_\tau,\overline\omega_1']\|_2\\[1ex]
&\hspace{3cm}\leq K\|\overline\omega_1\|_2\leq K\|h\|_{H^2}.
\end{aligned}
\end{equation}
Theorem \ref{T:I1}, Lemma \ref{L:31} (which can be applied as $(\pi_j^p\overline\omega_1)'\in\wh L^2(\s)$),   \eqref{Aprime3} and  \eqref{p1} (both for $\tau=3/4$), and 
\eqref{MM3b}-\eqref{MM3b'} combined yield
\begin{equation*} 
\|(\pi_j^p\overline\omega_1)'\|_2\leq C\|\pi_j^ph\|_{H^3}+K\|h\|_{H^{11/12}}+K\|\overline\omega_1\|_{H^{3/4}}\leq C\|\pi_j^ph\|_{H^3}+K\|h\|_{H^{11/12}},
\end{equation*}
and \eqref{p1} now entails  
\begin{equation}\label{MM3c}
\begin{aligned}
\|\pi_j^p\overline\omega_1'\|_2&\leq \|(\pi_j^p\overline\omega_1)'\|_2+\|(\pi_j^p)'\overline\omega_1\|_2\leq C\|\pi_j^ph\|_{H^3}+K\|h\|_{H^{11/12}}.
\end{aligned}
\end{equation}
Recalling that  $x_j^p\in I_j^p\subset J_j^p$  and  $\supp \chi_j^p=\cup_{n\in\Z}(2\pi n+J_j^p ), $ the embedding $H^{1}(\s)\hookrightarrow {\rm C}^{1/2}(\s)$ together with \eqref{MM3a-} \eqref{MM3a}, and \eqref{MM3c} finally yield
\begin{equation}\label{MM3}
\begin{aligned}
\|T_1[h]\|_2&\leq C\|\chi_j^p  (f'-f'(x_j^p))\|_\infty\|\pi_j^ph\|_{H^3}+ K\|h\|_{H^{11/12}}\leq \frac{C}{2^{p/2}}\|\pi_j^ph\|_{H^3}+ K\|h\|_{H^{11/12}}\\[1ex]
&\leq \frac{\mu}{6}\|\pi_j^ph\|_{H^3}+ K\|h\|_{H^{11/12}},
\end{aligned}
\end{equation}
provided that $p$ is sufficiently large. \medskip

Noticing that 
\[\frac{f_\tau'(x_j^p)}{1+f_\tau'^2(x_j^p)}C_{0,0}[\pi_j^p\overline\omega_1']=C_{1,1}(f_\tau'(x_j^p){\rm id}_\R)[f_\tau'(x_j^p){\rm id}_\R,\pi_j^p\overline\omega_1']\]
we write  the term $T_2[h]$  as
\[
T_2[h]=f_\tau'(x_j^p)T_{21}[h]-\frac{f_\tau'^2(x_j^p)}{1+f_\tau'^2(x_j^p)}T_{22}[h],
\]
where
\begin{align*}
T_{21}[h]&:=C_{1,1}(f_\tau)[f_\tau-f_\tau'(x_j^p){\rm id}_\R,\pi_j^p\overline\omega_1'],\\[1ex]
T_{22}[h]&:=C_{2,1}(f_\tau)[f_\tau-f_\tau'(x_j^p){\rm id}_\R,f_\tau+f_\tau'(x_j^p){\rm id}_\R,\pi_j^p\overline\omega_1'].
\end{align*}
Though $f_\tau'(x_j^p){\rm id}_\R$ is not $ 2\pi$-periodic, it is easy to see that the functions  $T_{2i}[h]$  still belong  to $L_2(\s) $ for $i\in\{1,2\}$.
Since $\chi_j^p\pi_j^p=\pi_j^p$, we have $$T_{21}[h]:= T_{21a}[h]+T_{21b}[h],$$
 where
\begin{align*}
T_{21a}[h]&:= \chi_j^p\PV\int_{|s|<\frac{\pi}{2^p}}\frac{\delta_{[\cdot,s]} (f_\tau-f_\tau'(x_j^p){\rm id}_\R)/s}
{1+\big(\delta_{[\cdot,s]} f_\tau/s\big)^2}\frac{(\pi_j^p\overline\omega_1')(\cdot-s)}{s}\, ds,\\[1ex]
T_{21b}[h]&:=\chi_j^p\PV\int_{\frac{\pi}{2^p}<|s|<\pi}\frac{\delta_{[\cdot,s]} (f_\tau-f_\tau'(x_j^p){\rm id}_\R)/s}
{1+\big(\delta_{[\cdot,s]} f_\tau/s\big)^2}\frac{(\pi_j^p\overline\omega_1')(\cdot-s)}{s}\, ds\\[1ex]
&\hspace{0.424cm}-\int_{-\pi}^\pi\frac{\big(\delta_{[\cdot,s]} (f_\tau-f_\tau'(x_j^p){\rm id}_\R)/s\big)
\big(\delta_{[\cdot,s]}\chi_j^p/s\big)}{1+\big(\delta_{[\cdot,s]} f_\tau/s\big)^2}(\pi_j^p\overline\omega_1')(\cdot-s)\, ds.
\end{align*}
Integrating by parts we obtain in view of   \eqref{p1} that 
\begin{align*}
\|T_{21b}[h]\|_2 \leq K\|\overline\omega_1\|_{2}\leq K\|h\|_{H^2}.
\end{align*}
Since $T_{21a}[h]\in L_2(\s),$ it holds   $\|T_{21a}[h]\|_2=\| T_{21a}[h]\|_{L_2((-\pi,\pi))}$.
Clearly, if $x\in\supp ({\bf 1}_{(-\pi,\pi)}T_{21a}[h]),$  then 
\[x\in(-\pi,\pi)\cap \big(\cup_{n\in\Z}(2n\pi+J_j^p)\big).\]
 Letting $J_j^p:=[a_j^p,b_j^p]$, $p\geq 3$, $1\leq j\leq 2^{p+1}$,  we distinguish three cases.
\begin{itemize}
 \item[$(i)$] If $1\leq j \leq 2^{p}-1$, then   $(-\pi,\pi)\cap  \big(2\pi n+ J_j^p\big)\neq\emptyset$  if and only if $n=0$ and
\[
(-\pi,\pi)\cap   J_j^p=[a_j^p,b_j^p].
\]
\item[$(ii)$] If $2^p+3\leq j\leq 2^{p+1}$, then  $(-\pi,\pi)\cap  \big(2\pi n+ J_j^p\big)\neq\emptyset$ if and only if $n=-1$ and 
\[
(-\pi,\pi)\cap (-2\pi+  J_j^p) =[a_j^p-2\pi ,b_j^p-2\pi].
\]
\item[$(iii)$] If  $j \in\{2^p, 2^{p}+1, 2^{p}+2\}$, then
  $(-\pi,\pi)\cap  \big(2\pi n+ J_j^p\big)\neq\emptyset$ if and only if $n \in\{-1,0\}$,  and 
\[
(-\pi,\pi)\cap    J_j^p  =[a_j^p,\pi)\qquad\text{and}\qquad(-\pi,\pi)\cap (-2\pi+  J_j^p) =(-\pi ,-2\pi+b_j^p].
\]
\end{itemize}
 Assume   that we are in the first case, that is $1\leq j\leq 2^p-1.$ 
 Let $F_{\tau,j}$ be the Lipschitz continuous function given by
\begin{equation*}
 F_{\tau,j}=f_\tau \quad   \text{on $[a_j^p,b_j^p]$}, \qquad F_{\tau,j}'=f'_\tau(x_j^p) \quad   \text{on $\R\setminus [a_j^p,b_j^p]$.} 
\end{equation*}
Then  $\|F_{\tau,j}'\|_\infty\leq \|f'\|_\infty$. 
Taking into account that $(\supp \pi_j^p)\cap[a_j^p-\pi/2^p,b_j^p+\pi/2^p]\subset[a_j^p,b_j^p],$ it follows that 
 \begin{align*}
{\bf 1}_{(-\pi,\pi)}T_{21a}[h]&= {\bf 1}_{(-\pi,\pi)}\chi_j^p\PV\int_{|s|<\frac{\pi}{2^p}}\frac{\delta_{[\cdot,s]} (F_{\tau,j}-f_\tau'(x_j^p){\rm id}_\R)/s}
{1+\big(\delta_{[\cdot,s]} f_\tau/s\big)^2}\frac{(\pi_j^p\overline\omega_1')(\cdot-s)}{s}\, ds,\\[1ex] 
&={\bf 1}_{(-\pi,\pi)}\chi_j^p C_{1,1}(f_\tau)[F_{\tau,j}-f_\tau'(x_j^p){\rm id}_\R,\pi_j^p\overline\omega_1']\\[1ex]
&\hspace{0.424cm}-{\bf 1}_{(-\pi,\pi)}\chi_j^p\PV\int_{\frac{\pi}{2^p}<|s|<\pi}\frac{\delta_{[\cdot,s]} (F_{\tau,j}-f_\tau'(x_j^p){\rm id}_\R)/s}
{1+\big(\delta_{[\cdot,s]} f_\tau/s\big)^2}\frac{(\pi_j^p\overline\omega_1')(\cdot-s)}{s}\, ds,
\end{align*}
and,  using integration by parts and \eqref{p1},  we arrive at  
 \begin{align*}
\Big\|\chi_j^p\PV\int_{\frac{\pi}{2^p}<|s|<\pi}\frac{\delta_{[\cdot,s]} (F_{\tau,j}-f_\tau'(x_j^p){\rm id}_\R)/s}
{1+\big(\delta_{[\cdot,s]} f_\tau/s\big)^2}\frac{(\pi_j^p\overline\omega_1')(\cdot-s)}{s}\, ds\Big\|_{L_2((-\pi,\pi))}
\leq K\|\overline\omega_1\|_{2}\leq
K\|h\|_{H^2}.
\end{align*}
Moreover, combining  Lemma \ref{L:A1} $(i)$ and \eqref{MM3c}, we find that  
\begin{align*}
\|\chi_j^p C_{1,1}(f_\tau)[F_{\tau,j}-f_\tau'(x_j^p){\rm id}_\R,\pi_j^p\overline\omega_1']\|_{L_2((-\pi,\pi))}&\leq C\|F'_{\tau,j}-f_\tau'(x_j^p)\|_\infty\|\pi_j^p\overline\omega_1'\|_2\\[1ex]
&\leq C\|f'-f'(x_j^p)\|_{L_\infty((a_j^p,b_j^p))}\big(\|\pi_j^p h\|_{H^3}+K\|h\|_{H^{11/12}}\big)\\[1ex]
&\leq \frac{\mu}{12}\|\pi_j^ph\|_{H^3}+ K\|h\|_{H^{11/12}},
\end{align*}
provided that $p$ is sufficiently large. 
Altogether, we conclude that for $1\leq j\leq 2^{p}-1$ it holds 
\begin{equation}\label{MM4a}
\|T_{21}[h]\|_2\leq \frac{\mu}{12}\|\pi_j^ph\|_{H^3}+ K\|h\|_{H^{11/12}}.
\end{equation}
Similar arguments apply also in the cases  $(ii)$ and $(iii)$, and therefore  the latter  estimate actually holds for all $1\leq j\leq 2^{p+1}$. 
Since  $T_{22}[h]$ can be estimated in the same way, we obtain that
\begin{equation}\label{MM4}
\|T_{2}[h]\|_2\leq \frac{\mu}{6}\|\pi_j^ph\|_{H^3}+ K\|h\|_{H^{11/12}}, 
\end{equation}
provided that $p$ is sufficiently large.

With regard to $T_3[h],$ it holds
\[
\|T_3[h]\|_2\leq \Big\|C_{0,0}[\pi_j^p\overline\omega_1']+ \frac{\pi}{(1+f_\tau'^2(x_j^p))^{3/2}}(\p_x^4)^{3/4}[\pi_j^p h]\Big\|_2,
\]
with
\begin{align*}
C_{0,0}[\pi_j^p\overline\omega_1']+ \frac{\pi}{(1+f_\tau'^2(x_j^p))^{3/2}}(\p_x^4)^{3/4}[\pi_j^p h]&=C_{0,0}\Big[\pi_j^p\overline\omega_1'-\frac{ \pi_j^p h'''}{(1+f_\tau'^2(x_j^p))^{3/2}}\Big]\\[1ex]
&\hspace{0.424cm}-\frac{1}{(1+f_\tau'^2(x_j^p))^{3/2}}C_{0,0}[3(\pi_j^p)'h''+3(\pi_j^p)''h'+(\pi_j^p)'''h]\\[1ex]
&\hspace{0.424cm}-\frac{1}{2(1+f_\tau'^2(x_j^p))^{3/2}}\int_{-\pi}^\pi\Big[\frac{1}{t_{[s]}}-\frac{1}{s/2}\Big](\pi_j^ph)'''(\cdot-s)\, ds.
\end{align*}
 Integration by parts and  Lemma \ref{L:A1} $(i)$ lead us to
   \begin{align*}
\Big\|C_{0,0}[\pi_j^p\overline\omega_1']+ \frac{\pi}{(1+f_\tau'^2(x_j^p))^{3/2}}(\p_x^4)^{3/4}[\pi_j^p h]\Big\|_2\leq 
C\Big\|\pi_j^p\overline\omega_1'-\frac{ \pi_j^p h'''}{(1+f_\tau'^2(x_j^p))^{3/2}}\Big\|_2+K\|h\|_{H^2}.
\end{align*}
A straight forward consequence of \eqref{MM3b} is the following identity
\begin{align*}
\pi_j^p\overline\omega_1'-\frac{ \pi_j^p h'''}{(1+f_\tau'^2(x_j^p))^{3/2}}=&\Big[\frac{1}{(1+f_\tau'^2)^{3/2}}-\frac{1}{(1+f_\tau'^2(x_j^p))^{3/2}}\Big]\pi_j^ph'''-a_\mu\bA(f_\tau)[\pi_j^p\overline\omega_1']\\[1ex]
&\hspace{0.424cm}-\frac{3\pi_j^pf_\tau'f_\tau''h''}{(1+f_\tau'^2)^{5/2}}-a_\mu\pi_j^pT_{\rm lot}^A(f_\tau)[\overline\omega_1]+a_\mu\big(\bA(f_\tau)[\pi_j^p\overline\omega_1']-\pi_j^p\bA(f_\tau)[\overline\omega_1']\big).
\end{align*}
Using once more the H\"older continuity of $f'$,  \eqref{Aprime3} and \eqref{p1} (both with $\tau=3/4$) together with \eqref{MM3b'} 
yields that for  $p$ sufficiently large  
 \begin{align}\label{MM5a}
\|T_3[h]\|_2\leq C\| f_\tau'C_{0,1}(f_\tau)[\pi_j^p\overline\omega_1']-C_{1,1}(f_\tau)[f_\tau,\pi_j^p\overline\omega_1']\|_2+\frac{\mu}{24}\|\pi_j^ph\|_{H^3}+ K\|h\|_{H^{11/12}}.
\end{align}
We are left with the term
\[f_\tau'C_{0,1}(f_\tau)[\pi_j^p\overline\omega_1']-C_{1,1}(f_\tau)[f_\tau,\pi_j^p\overline\omega_1']=T_{31}[h]-T_{21}[h],\]
with $T_{21}[h]$ defined above and with
\begin{align*}
T_{31}[h]&:=(f_\tau'-f_\tau'(x_j^p))C_{0,1}(f_\tau)[\pi_j^p\overline\omega_1'].
\end{align*}
Since 
\begin{align*}
T_{31}[h]&=\chi_j^p(f_\tau'-f_\tau'(x_j^p))C_{0,1}(f_\tau)[\pi_j^p\overline\omega_1']
-(f_\tau'-f_\tau'(x_j^p))\int_{-\pi}^\pi\frac{\delta_{[\cdot,s]}\chi_j^p/s}{1+\big(\delta_{[\cdot,s]}f_\tau/s\big)^2}(\pi_j^p\overline\omega_1')(\cdot-s)\, ds,
\end{align*}
the estimate \eqref{MM3c} and Lemma \ref{L:A1} $(i)$ for the first term, respectively
integration by parts for the second term lead us, for $p$ sufficiently large, to
\begin{align}\label{MM5b}
C\|T_{31}[h]\|_2\leq  \frac{\mu}{24}\|\pi_j^ph\|_{H^3}+ K\|h\|_{H^{11/12}}.
\end{align}
Gathering \eqref{MM4a} (which is valid also for $C\|T_{21}[h]\|_2$ provided that we choose a larger $p$ if required), 
\eqref{MM5a}, and \eqref{MM5b}, we conclude that  
 \begin{align}\label{MM5}
\|T_3[h]\|_2\leq \frac{\mu}{6}\|\pi_j^ph\|_{H^3}+ K\|h\|_{H^{11/12}},
\end{align}
provided that  $p$ is sufficiently large. 
The estimate \eqref{MMM2} follows now from \eqref{MM3}, \eqref{MM4}, and \eqref{MM5}.\medskip

\noindent{\em Step 2: The second leading  order term.}
Given $1\leq j\leq 2^{p+1}$ and $\tau \in[0,1]$, let
\[
\bA_{j,\tau}^2:=-\frac{1}{(1+f_\tau'^2(x_j^p))^{5/2}}(\p_x^4)^{3/4},
\]
where $x_j^p\in I_j^p$.
Similarly as in the previous  step, it follows that
  \begin{equation}\label{MMM3}
  \Big\| \pi_j^p  C_{0,1}(f_\tau)[\overline\omega_1']-\pi\bA^2_{j,\tau}[\pi^p_j h]\Big|_{2}\leq \frac{\mu}{2} \|\pi_j^p h\|_{H^3}+K\|  h\|_{H^{11/4}}
 \end{equation}
 for all $ j\in\{1,\ldots, 2^{p+1}\}$, $\tau\in[0,1],$ and  $h\in H^3(\s)$, provided that $p$ is sufficiently large. \medskip
 
 The desired claim \eqref{DE3} follows from \eqref{MMM0}, \eqref{MMM1}, \eqref{MMM2}, and \eqref{MMM3}.

 \end{proof}

We are now in a position to prove \eqref{GP}.
\begin{thm}\label{T:GP1}
Given $f\in H^2(\s)$, it holds that 
\[
-\Phi_\sigma(f)\in\kH(H^3(\s), L_2(\s)).
\]
\end{thm}
\begin{proof}
Let $\Phi^c_{\sigma}(f)=\Phi^c_{\sigma,1}(f)+\Phi^c_{\sigma,2}(f)$ denote the complexification of $\Phi_\sigma(f)$ (the Sobolev spaces where $\Phi^c_{\sigma}(f)$
acts are now complex valued).  
In view of  \cite[Corollary~2.1.3]{L95} is suffices to show that $-\Phi_\sigma^c(f)\in\kH(H^3(\s), L_2(\s)).$
Moreover,  for the choice $\tau=3/4$ in Proposition \ref{P:I1} $(b)$, we obtain together with Lemma \ref{L:32}, that $\Phi_{\sigma,2}^c(f)\in\kL(H^{11/4}(\s), L_2(\s)).$
Since $[L_2(\s), H^3(\s)]_{11/12}=H^{11/4}(\s),$ cf. \eqref{IP}, by \cite[Theorem I.1.3.1 (ii)]{Am95}  we only need to show  that 
\begin{align}\label{GPS1}
-\Phi_{\sigma,1}^c(f)\in\kH(H^3(\s), L_2(\s)).
\end{align}
Recalling   \cite[Remark I.1.21 (a) ]{Am95}, we are left to find constants $\omega>0$ and $\kappa\geq1$
such that 
\begin{align}
&\omega-\Phi^c_{\sigma,1}(f)\in {\rm Isom}(H^3(\s), L_2(\s)), \label{TBS1}\\[1ex]
&\kappa\|(\lambda-\Phi^c_{\sigma,1}(f))[h]\|_2\geq |\lambda|\cdot \|h\|_{2}+\|h\|_{H^3}\qquad\text{$\forall$ $h\in H^3(\s)$  and $\re\lambda\geq\omega$.}\label{TBS2}
\end{align}
Let $a>1$ be chosen  such that
\[
\frac{1}{a}\leq \frac{1}{(1+\|f'\|^2_\infty)^{3/2}}\leq a.
\]
For each $\alpha\in[a^{-1},a]$, let $\bA_\alpha:H^3(\s)\to L_2(\s)$ denote operator  $\bA_\alpha:=-\alpha(\p_x^4)^{3/4}.$ 
Then it is easy to see that  for $\kappa':=1+a$ the following hold 
\begin{align}
&\lambda-\bA_\alpha\in {\rm Isom}(H^3(\s), L_2(\s))\qquad\text{$\forall$  $\re\lambda\geq1$,} \label{TBS1a}\\[1ex]
&\kappa'\|(\lambda-\bA_\alpha)[h]\|_2\geq |\lambda|\cdot \|h\|_{2}+\|h\|_{H^3}\qquad\text{$\forall$ $h\in H^3(\s)$  and $\re\lambda\geq1$.}\label{TBS2a}
\end{align}
Taking  $\mu:=1/(2\kappa')$  in Theorem \ref{T:K1}, we find  $p\geq3$, a   $p$-partition of unity  $\{\pi_j^p\,:\, 1\leq j\leq 2^{p+1}\} $, a constant $K=K(p)$, and for each 
 $j\in\{1,\ldots,2^{p+1}\}$ and $\tau\in[0,1]$ operators $\bA_{j,\tau}^c\in\kL(H^3(\s), L_2(\s))$ ($\bA_{j,\tau}^c$ is the complexification of $\bA_{j,\tau}$ defined in \eqref{FM1})
 such that 
 \begin{equation}\label{DE3a}
  \|\pi_j^p\Phi_{\sigma,1}^c(\tau f)[h]-\bA_{j,\tau}^c[\pi^p_j h]\|_{2}\leq \mu \|\pi_j^p h\|_{H^3}+K\|  h\|_{H^{11/4}}
 \end{equation}
 for all $ j\in\{1,\ldots, 2^{p+1}\}$, $\tau\in[0,1],$ and  $h\in H^3(\s)$. 
 We note that  the relations \eqref{TBS1a} and \eqref{TBS2a} are both valid for $\bA_{j,\tau}^c$ as $\bA_{j,\tau}^c\in\{\bA_\alpha\,:\,\alpha\in[a^{-1},a]\}$.
 It now follows from  \eqref{TBS2a} and \eqref{DE3a} that 
 \begin{align*}
 \kappa'\|\pi_j^p(\lambda-\Phi_{\sigma,1}^c(\tau f))[h]\|_2&\geq \kappa'\| (\lambda- \bA_{j,\tau}^c(f))[\pi_j^p h]\|_2
 -\kappa'\|\pi_j^p\Phi_{\sigma,1}^c(\tau f)[h]-\bA_{j,\tau}^c[\pi^p_j h]\|_2\\[1ex]
 &\geq |\lambda|\cdot \|\pi^p_j h\|_{2}+\frac{1}{2}\|\pi^p_j h\|_{H^3} -\kappa'K\|  h\|_{H^{11/4}} 
 \end{align*}
for all   $ j\in\{1,\ldots, 2^{p+1}\}$, $\tau\in[0,1],$ and  $h\in H^3(\s)$. 
  Since for each $k\in\N$ 
$$\Big[h\mapsto \max_{1\leq j\leq 2^{p+1}} \|\pi_j^p h\|_{H^k}\Big]: H^k(\s)\to\R,$$
defines a  norm equivalent to the standard  $H^k(\s)$-norm, cf. \cite[Remark 4.1]{MM17x},
Young's inequality together with \eqref{IP} enables us to conclude from the previous inequality the existence of constants $\omega>1$ and $\kappa\geq1$
with
\begin{align}
&\kappa\|(\lambda-\Phi^c_{\sigma,1}(\tau f))[h]\|_2\geq |\lambda|\cdot \|h\|_{2}+\|h\|_{H^3}\qquad\text{$\forall$ $h\in H^3(\s)$, $\tau\in[0,1]$,  and $\re\lambda\geq\omega$.}\label{TBS2'}
\end{align}
Choosing $\tau=1$ in \eqref{TBS2'} we obtain \eqref{TBS2}.
Moreover, the estimate \eqref{TBS2'} for $\lambda=\omega$, \eqref{TBS1a} ($\Phi^c_{\sigma,1}(\tau f)=\bA_1$ for $\tau=0$), and the method of continuity \cite[Proposition I.1.1.1]{Am95}
ensure that the property \eqref{TBS1} also holds and the proof is complete.
\end{proof}\medskip

We now come to the proof our first main result which uses on  the one hand the abstract theory for quasilinear parabolic problems outlined 
 in \cite{Am86, Am86b, Am88, Am93, Am95} (see also \cite[Theorem 1.1]{MW18x}), and on the other hand a parameter trick which has been employed in various 
versions  in \cite{An90, ES96, PSS15, M16x, M17x, MM17x} in the context of improving the regularity of solutions to certain parabolic evolution equations.    
We point out that the parameter trick can only be used because the uniqueness claim of Theorem \ref{MT:1} holds in the setting of classical solution (the solutions in Theorem \ref{MT:1}
possess though additional H\"older regularity properties, see the proof of Theorem \ref{MT:1}).   

\begin{proof}[Proof of Theorem \ref{MT:1}] Let $\E_1:=H^3(\s)$, $\E_0:=L_2(\s)$, $\beta:=2/3$ and $\alpha:=r/3$. Then $\E_1\hookrightarrow \E_0$ is a  compact embedding,
 $0<\beta<\alpha<1$, and it follows from Theorem \ref{T:GP1} and \eqref{RegS} that the abstract result \cite[Theorem 1.1]{MW18x} may be applied in the context of the Muskat problem \eqref{MPS}.
Hence,  given $f_0\in H^r(\s)=[L_2(\s), H^3(\s)]_\alpha$, \eqref{MPS} possesses a unique classical solution 
$f=f(\,\cdot\,; f_0)$, that is
\[
f\in C([0,T_+(f_0)), H^r(\s))\cap C((0,T_+(f_0)), H^3(\s))\cap C^1((0,T_+(f_0)), L_2(\s)),
\]
 where $T_+(f_0)\leq\infty$,  which has the property that
\[
f\in C^{\alpha-\beta}([0,T], H^2(\s)) \qquad\text{for all $T<T_+(f_0)$.}
\]
Concerning the uniqueness statement of Theorem \ref{MT:1} $(i)$, it suffices to prove   that if $T>0$ and  
\begin{equation}\label{UNIV}
f\in C([0,T], H^r(\s))\cap C((0,T], H^3(\s))\cap C^1((0,T], L_2(\s))
\end{equation}
solves \eqref{MPS} pointwise, then
\begin{equation}\label{UNIC}
f\in C^{\eta}([0,T], H^2(\s))  \qquad\text{for $\eta:=\frac{r-2}{r+1}$,}
\end{equation}
cf. \cite[Theorem 1.1]{MW18x}.
Let thus $f$ be a solution to \eqref{MPS} which  satisfies \eqref{UNIV}.
Since  $f\in C([0,T], H^r(\s))$ and $r>2$, we deduce from the Theorems \ref{T:I1} and \ref{T:I2} via interpolation that
\[
\sup_{t\in[0,T]}\|(1+a_\mu\bA(f))^{-1}\|_{\kL(\wh H^{r-2}(\s))}\leq C.
\] 
Since $\langle\kappa(f)\rangle=0$ and $\sup_{t\in[0,T]}\|\kappa(f)\|_{H^{r-2}}\leq C,$ it follows for $\overline\omega_1:=\overline\omega_1(f)[f]=(1+a_\mu \bA(f))^{-1}[\kappa(f)]$ (see Proposition \ref{P:I1}) that
\begin{align}\label{UNI0}
\sup_{t\in[0,T]}\|\overline\omega_1 \|_{ H^{r-2} }\leq C.
\end{align}
We next show that
\begin{align}\label{UNI1}
\sup_{t\in(0,T]}\|\Phi_{\sigma,1}(f)[f]\|_{H^{-1}}+\sup_{t\in(0,T]}\|\Phi_{\sigma,2}(f)[f]\|_{H^{-1}}\leq C.
\end{align}
It follows from the definitions of $\Phi_{\sigma,1}$ and $\overline\omega_1$ that
\begin{align*}
\Phi_{\sigma,1}(f)[f]=f'\bB_1(f)[\overline\omega_1']-\bB_2(f)[\overline\omega_1']+\frac{1}{\pi}\big(C_{0,1}(f)[\overline\omega_1']+f'C_{1,1}(f)[f,\overline\omega_1']\big),\qquad t\in(0,T].
\end{align*}
Using integration by parts, it is not difficult to  derive, with the help of \eqref{UNI0}, the estimate
\begin{align}\label{UNI1A}
\sup_{t\in(0,T]}\|f'\bB_1(f)[\overline\omega_1']\|_{2}+\sup_{t\in(0,T]}\|\bB_2(f)[\overline\omega_1']\|_{2}\leq C,
\end{align}
and we are left to consider the terms $C_{0,1}(f)[\overline\omega_1']$ and $f'C_{1,1}(f)[f,\overline\omega_1']$.
Since $\overline\omega_1\in H^1(\s)$ for $t\in(0,T],$ it is shown in Lemma \ref{L:32} that $C_{1,1}(f)[f,\overline\omega_1]\in H^1(\s)$  with
 \[
 f'C_{1,1}(f)[f,\overline\omega_1']=f'(C_{1,1}(f)[f,\overline\omega_1])'-f'C_{1,1}(f)[f',\overline\omega_1]+2f'C_{3,2}(f,f)[f',f,f,\overline\omega_1].
 \]
We estimate the terms on the right hand side of the latter identity in the $H^{-1}$-norm one by one.
 Given $\varphi\in H^1(\s)$, integration by parts, \eqref{UNI0},  and Lemma \ref{L:A1} $(i)$  yield 
 \begin{align*}
 \Big|\int_{-\pi}^\pi f'(C_{1,1}(f)[f,\overline\omega_1])'\varphi\, dx\Big|&\leq \Big|\int_{-\pi}^\pi f''C_{1,1}(f)[f,\overline\omega_1]\varphi\, dx\Big|
 +\Big|\int_{-\pi}^\pi f'C_{1,1}(f)[f,\overline\omega_1]\varphi'\, dx\Big|\leq C\|\varphi\|_{H^1},
 \end{align*}
 and therewith
 \begin{align}\label{UNI1b}
\sup_{t\in(0,T]}\|f'(C_{1,1}(f)[f,\overline\omega_1])'\|_{H^{-1}}\leq C.
\end{align}
In order to estimate $f'C_{1,1}(f)[f',\overline\omega_1]$ we write
\[
 C_{1,1}(f)[f',\overline\omega_1]=T_1-T_2-T_3,
\]
where
\begin{align*}
T_1&:=\int_{0}^\pi\frac{\delta_{[x,s]}f'/s}{1+\big(\delta_{[x,s]}f/s\big)^2}\frac{\overline\omega_1(x-s)-\overline\omega_1(x+s)}{s}\, ds,\\[1ex]
T_2&:=\int_{0}^\pi\frac{1}{1+\big(\delta_{[x,s]}f/s\big)^2}\frac{f'(x+s)-2f'(x)+f'(x-s)}{s}\frac{\overline\omega_1(x+s)}{s}\, ds,\\[1ex]
T_3&:=\int_{0}^\pi\frac{\big[(\delta_{[x,s]}f/s)-(\delta_{[x,-s]}f/s)\big]\big(\delta_{[x,-s]}f'/s\big)}{\big[1+\big(\delta_{[x,s]}f/s\big)^2\big]\big[1+\big(\delta_{[x,-s]}f/s\big)^2\big]}\frac{f(x+s)-2f(x)+f(x-s)}{s}\frac{\overline\omega_1(x+s)}{s}\, ds.
\end{align*}
   Given $\varphi\in H^1(\s)$,        Fubini's theorem  yields for $t\in(0,T]$
\begin{align*}
 \Big|\int_{-\pi}^\pi f'T_1\varphi\, dx\Big|
 &\leq C\|\varphi\|_{H^1}\int_0^\pi\int_{-\pi}^\pi\Big|\frac{\delta_{[x,s]}f'}{s}\Big|\cdot \Big|\frac{\overline\omega_1(x-s)-\overline\omega_1(x+s)}{s}\Big|\, dx\, ds\\[1ex]
 &\leq C\|\varphi\|_{H^1}\int_0^\pi\frac{1}{s^2}\Big(\int_{-\pi}^\pi|f'-\tau_s f'|^2\, dx\Big)^{1/2} \Big(\int_{-\pi}^\pi|\tau_s \overline\omega_1-\tau_{-s} \overline\omega_1|^2\, dx\Big)^{1/2}\, ds\\[1ex]
  &\leq C\|\varphi\|_{H^1}\int_0^\pi\frac{1}{s^2}\Big(\sum_{k\in\Z}|k|^2|\wh f(k)|^2|e^{iks}-1|^2\Big)^{1/2} \Big(\sum_{k\in\Z}|\wh{ \overline\omega}_1(k)|^2|e^{i2ks}-1|^2 \Big)^{1/2}\, ds,
\end{align*}
and since $|e^{i\xi}-1|\leq C|\xi|$, respectively  $|e^{i\xi}-1|\leq C|\xi|^{r-2}$, for all $\xi\in\R,$ the latter inequality together with \eqref{UNI0} leads   to
\[
\|f'T_1\|_{H^{-1}}\leq C\|f\|_{H^2}\|\overline\omega_1\|_{H^{r-2}}\int_0^\pi s^{r-3} \, ds \leq C.
\]
Arguing along the same lines we  find for $t\in(0,T]$, in view of $|e^{i\xi}-2+e^{-i\xi }|\leq C|\xi|^{r-1}$ for all $\xi\in\R,$ that 
\begin{align*}
 \Big|\int_{-\pi}^\pi f'T_2\varphi\, dx\Big|&\leq C\|\varphi\|_{H^1}\|\overline\omega_1\|_2\int_0^\pi\frac{1}{s^2}\Big(\int_{-\pi}^\pi|\tau_{-s} f'-2f'+\tau_{s} f'|^2\, dx\Big)^{1/2}\, ds\\[1ex]
  &\leq C\|\varphi\|_{H^1}\int_0^\pi\frac{1}{s^2}\Big(\sum_{k\in\Z}|k|^2|\wh f(k)|^2|e^{iks}-2+e^{-iks}|^2\Big)^{1/2}  \, ds \\[1ex]
  &\leq C\|\varphi\|_{H^1}\|f\|_{H^r}\int_0^\pi s^{r-3}  \, ds,
\end{align*}
and therewith 
\[
\|f'T_2\|_{H^{-1}}\leq C.
\]
Finally, the inequality $|e^{i\xi}-2+e^{-i\xi }|\leq C|\xi|^{2}$ for all $\xi\in\R $ together with the Sobolev embedding 
$H^{r-1}(\s)\hookrightarrow {\rm C}^{r-3/2}(\s)$ for $r\neq 5/2,$ yield  for $t\in(0,T]$ that
\begin{align*}
 \Big|\int_{-\pi}^\pi f'T_3\varphi\, dx\Big|&\leq C\|\varphi\|_{H^1}\|\overline\omega_1\|_2\|f\|_{H^r}\int_0^\pi s^{\min\{-2,r-9/2\}}\Big(\int_{-\pi}^\pi|\tau_{-s} f-2f+\tau_{s} f|^2\, dx\Big)^{1/2}\, ds\\[1ex]
  &\leq C\|\varphi\|_{H^1}\int_0^\pi s^{\min\{-2,r-9/2\}}\Big(\sum_{k\in\Z} |\wh f(k)|^2|e^{iks}-2+e^{-iks}|^2\Big)^{1/2}  \, ds \\[1ex]
  &\leq C\|\varphi\|_{H^1}\|f\|_{H^2}\int_0^\pi s^{\min\{0,r-5/2\}}  \, ds,
\end{align*}
hence
\[
\|f'T_3\|_{H^{-1}}\leq C.
\]
The latter estimate clearly holds also for $r=5/2$.
We have thus shown that  
\begin{align}\label{UNI1c}
\sup_{t\in(0,T]}\|f'C_{1,1}(f)[f',\overline\omega_1]\|_{H^{-1}}\leq C 
\end{align} 
holds true. Similarly  
 \begin{align}\label{UNI1d}
\sup_{t\in(0,T]}\|f'C_{3,2}(f,f)[f',f,f,\overline\omega_1]\|_{H^{-1}}\leq C.
\end{align}
Gathering \eqref{UNI1b}-\eqref{UNI1d}, it follows that 
\begin{align}\label{UNI1B}
\sup_{t\in(0,T]}\|f'C_{1,1}(f)[f,\overline\omega_1']\|_{H^{-1}}\leq C.
\end{align} 
Similarly, we get
\begin{align}\label{UNI1C}
\sup_{t\in(0,T]}\| C_{0,1}(f)[\overline\omega_1']\|_{H^{-1}}\leq C,
\end{align} 
and \eqref{UNI1A}, \eqref{UNI1B}, and \eqref{UNI1C} lead to 
\begin{align}\label{UNI1x}
\sup_{t\in(0,T]}\|\Phi_{\sigma,1}(f)[f]\|_{H^{-1}}\leq C.
\end{align}

We now consider the second term $\Phi_{\sigma,2}.$ Given $t\in(0,T]$, it holds 
\[
\Phi_{\sigma,2}(f)[f]=\bB(f)[\overline\omega_2(f)[f]]=-\Theta\bB(f)[(1+a_\mu\bA(f))^{-1}[f']]+a_\mu\bB(f)[(1+a_\mu\bA(f))^{-1}[T_{\rm lot}^A(f)[\overline\omega_1]]],
\] 
and Lemma \ref{L:32} together with Theorem \ref{T:I1} yields
\begin{align}\label{UNI2A1}
\|\bB(f)[(1+a_\mu\bA(f))^{-1}[f']]\|_2\leq C\|(1+a_\mu\bA(f))^{-1}[f']\|_2\leq C\|f'\|_2\leq C \qquad\text{$\forall$ $t\in[0,T]$.}
\end{align}
We now estimate $\|\bB(f)[\overline\omega_3]\|_{H^{-1}},$ where $\overline\omega_3:=\overline\omega_3(f):=(1+a_\mu\bA(f))^{-1}[T_{\rm lot}^A(f)[\overline\omega_1]]\in \wh L_2(\s)$ for  $t\in(0,T]$.
We begin by showing that the function $T_{\rm lot}^A(f)[\overline\omega_1]\in \wh L_2(\s)$, see \eqref{Aprime2},  satisfies
  \begin{align}\label{UNI2H}
\sup_{t\in(0,T]}\|T_{\rm lot}^A(f)[\overline\omega_1]\|_1\leq C.
\end{align}
Firstly we   consider the difference $(f'\bB_2(f)[\overline\omega_1])'-f'\bB_2(f)[\overline\omega_1'],$  which we estimate,
in view of \eqref{UNI0} and Lemma \ref{L:31}, as follows
 \begin{align*}
 \|(f'\bB_2(f)[\overline\omega_1])'-f'\bB_2(f)[\overline\omega_1']\|_1&\leq\|f''\bB_2(f)[\overline\omega_1]\|_1+\|f'\|_\infty\|(\bB_2(f)[\overline\omega_1])'-\bB_2(f)[\overline\omega_1']\|_1\\[1ex]
 &\leq \|f''\|_2\|\bB_2(f)[\overline\omega_1]\|_2+C\|(\bB_2(f)[\overline\omega_1])'-\bB_2(f)[\overline\omega_1']\|_1\\[1ex]
 &\leq C(1+\|(\bB_2(f)[\overline\omega_1])'-\bB_2(f)[\overline\omega_1']\|_1).
 \end{align*}
 Secondly, it is not difficult to see that 
\begin{align*}
 \|(\bB_1(f)[\overline\omega_1])'-\bB_1(f)[\overline\omega_1']\|_1+\|(\bB_2(f)[\overline\omega_1])'-\bB_2(f)[\overline\omega_1']\|_1\leq C\|\overline\omega_1\|_2\leq C.
 \end{align*}

 We still need to estimate  the terms  of $T_{\rm lot}^A(f)[\overline\omega_1]$  defined by means of the operators $C_{n,m}$ introduced in Lemma \ref{L:A1}.
 This is done  as follows
 \begin{align*}
 &\|f''C_{0,1}(f)[\overline\omega_1]\|_1\leq \|f''\|_2\|C_{0,1}(f)[\overline\omega_1]\|_2\leq C\|\overline\omega_1\|_2\leq C,\\[1ex]
 &\|f'C_{2,2}(f,f)[f',f,\overline\omega_1]\|_1 +\|C_{1,1}(f)[f',\overline\omega_1]\|_1+\|f'C_{3,2}(f,f)[f',f,f,\overline\omega_1]\|_1\leq C,
\end{align*}  
  the last estimate following in a similar way as  \eqref{UNI1c}. 
  Altogether,  \eqref{UNI2H} holds true.

Given $t\in(0,T]$,  we compute for $\varphi\in H^1(\s)$ that
\begin{align*}
\Big|\int_{-\pi}^\pi\overline\omega_3\varphi\, dx\Big|&=\Big|\int_{-\pi}^\pi(1+a_\mu\bA(f))^{-1}[T_{\rm lot}^A(f)[\overline\omega_1]]P\varphi\, dx\Big|
\\[1ex]
&=\Big|\int_{-\pi}^\pi T_{\rm lot}^A(f)[\overline\omega_1](1+a_\mu\big(\wh{\bA(f)}\big)^*)^{-1}[P\varphi]\, dx\Big|\\[1ex]
&\leq \|T_{\rm lot}^A(f)[\overline\omega_1]\|_1\|(1+a_\mu\big(\wh{\bA(f)}\big)^*)^{-1}\|_{\kL(\wh H^1(\s))}\|P\varphi\|_{H^1},
\end{align*}
 where   $P $ is the orthogonal projection on $\wh L_2(\s).$
 This inequality together with Theorem \ref{T:I3} and  \eqref{UNI2H} implies 
  \begin{align}\label{UNI2G}
\sup_{t\in(0,T]}\|\overline\omega_3\|_{H^{-1}}\leq C.
\end{align}
Since for $t\in(0,T] $ and $\varphi\in H^1(\s)$  
\begin{align*}
\Big|\int_{-\pi}^\pi\bB(f)[\overline\omega_3]\varphi\, dx\Big|&=\Big|\int_{-\pi}^\pi \overline\omega_3(\bB(f))^*[\varphi]\, dx\Big|
\leq \|\overline\omega_3\|_{H^{-1}}\|(\bB(f))^*\|_{\kL(\wh H^1(\s))}\|\varphi\|_{H^1},
\end{align*}
 Lemma \ref{L:33} together with \eqref{UNI2G} lead us to
 \begin{align}\label{UNI2A2}
\sup_{t\in(0,T]}\|\bB(f)[\overline\omega_3]\|_{H^{-1}}\leq C.
\end{align}

In view of  \eqref{UNI2A1} and  \eqref{UNI2A2}   we conclude that 
 \begin{align}\label{UNI1y}
\sup_{t\in(0,T]}\|\Phi_{\sigma,2}(f)[f]\|_{H^{-1}}\leq C,
\end{align}
and the claim \eqref{UNI1} follows from \eqref{UNI1x} and \eqref{UNI1y}.

Recalling that $f\in {\rm C}^1((0,T], L_2(\s))\cap {\rm C}([0,T], H^r(\s)),$   \eqref{UNI1}  
yields $f\in {\rm BC}^1((0,T], H^{-1}(\s))$ and the property \eqref{UNIC}
is now a straight forward consequence of \eqref{IP}. This proves the uniqueness claim in Theorem~\ref{MT:1} and herewith the assertion $(i)$.
The claim $(ii)$  follows directly from \cite[Theorem~1.1]{MW18x}, while the parabolic smoothing property stated at $(iii)$ is obtain by using 
a parameter trick in the same way as in the proof of \cite[Theorem 1.3]{M16x}. The proof of Theorem \ref{MT:1} is now complete.
 \end{proof}

 %%%%%%%%%%%%%%%%%%%%%%%%%%%%%%%%%%%%%%%%%%%%%%%%%%%%%%%%%%%%%%%%%%%
%%%%%%%%%%%%%%%%%%%%%%%%%%%%%%%%%%%%%%%%%%%%%%%%%%%%%%%%%%%%%%%%%%%%
%%%%%%%%%%%%%%%%%%%%%%%%%%%%%%%%%%%%%%%%%%%%%%%%%%%%%%%%%%%%%%%%%%%%
%%%%%%%%%%%%%%%%%%%%%%%%%%%%%%%%%%%%%%%%%%%%%%%%%%%%%%%%%%%%%%%%%%%
%%%%%%%%%%%%%%%%%%%%%%%%%%%%%%%%%%%%%%%%%%%%%%%%%%%%%%%%%%%%%%%%%%%%
%%%%%%%%%%%%%%%%%%%%%%%%%%%%%%%%%%%%%%%%%%%%%%%%%%%%%%%%%%%%%%%%%%%%
\section{The  Muskat problem without surface tension effects}\label{Sec:5}
 %%%%%%%%%%%%%%%%%%%%%%%%%%%%%%%%%%%%%%%%%%%%%%%%%%%%%%%%%%%%%%%%%%%
%%%%%%%%%%%%%%%%%%%%%%%%%%%%%%%%%%%%%%%%%%%%%%%%%%%%%%%%%%%%%%%%%%%%
%%%%%%%%%%%%%%%%%%%%%%%%%%%%%%%%%%%%%%%%%%%%%%%%%%%%%%%%%%%%%%%%%%%%
%%%%%%%%%%%%%%%%%%%%%%%%%%%%%%%%%%%%%%%%%%%%%%%%%%%%%%%%%%%%%%%%%%%
%%%%%%%%%%%%%%%%%%%%%%%%%%%%%%%%%%%%%%%%%%%%%%%%%%%%%%%%%%%%%%%%%%%%
%%%%%%%%%%%%%%%%%%%%%%%%%%%%%%%%%%%%%%%%%%%%%%%%%%%%%%%%%%%%%%%%%%%%
 We now investigate the evolution problem \eqref{P} in the  absence of the surface tension effects, that is for $\sigma=0$. 
  One of the main features of the Muskat problem with surface tension, namely  the quasilinear character,  seems to be lost  as the curvature term disappears from the equations. 
  Nevertheless, we show below that  \eqref{P} can be recast as a fully nonlinear and  nonlocal evolution problem  
\begin{align}\label{MPOS}
\dot f(t) =\Phi (f(t) ),\quad t\geq0,\qquad f(0)=f_0,
\end{align}
 with   $[f\mapsto \Phi(f)]\in C^\omega( H^2(\s), H^1(\s))$ defined in \eqref{PHOS}.
While the Muskat problem with surface tension is parabolic  regardless of the initial data that are considered, in the case when $\sigma=0$ we can prove that the 
Fr\' echet derivative  $\p\Phi(f_0)$ 
generates  a strongly  continuous and analytic semigroup in $\kL(H^1(\s))$, more precisely that
\begin{align}\label{GPOS}
-\p\Phi(f_0)\in\kH(H^2(\s), H^1(\s)),
\end{align}
only when  requiring that 
the initial data $f_0\in H^2(\s)$ are chosen such that the Rayleigh-Taylor condition is satisfied. 
Establishing \eqref{GPOS} is the first goal of this section and this necessitates some preparations.

To begin, we solve the  equation $\eqref{EP}_2$, which is, up to a factor of 2, equivalent to
\begin{align}\label{CEOS}
 (1+a_\mu \bA(f))[\overline\omega] =-c_\Theta f',
\end{align}
where  
\[
c_{\Theta}:=\displaystyle\frac{k\Theta}{\mu_-+\mu_+}.
\]
It is worth mentioning that in order to solve \eqref{CEOS} for $\overline\omega$ in $\wh H^1(\s)$  it is required in Theorem \ref{T:I2} that the left hand side belongs to $\wh H^1(\s),$ 
that is $f\in H^2(\s)$, and this is precisely the regularity required also for the function in the argument of $\bA$. 
Hence, \eqref{CEOS} is no longer quasilinear,
 unless $a_\mu=0$, see \cite{MM17x}. 
 
 \begin{prop}\label{P:O}
 Given $f\in H^2(\s)$, there exists a unique solution $\overline\omega:=\overline\omega(f)\in \wh H^1(\s)$ to \eqref{CEOS} and
 \begin{align}\label{oo}
 \overline\omega\in C^\omega(H^2(\s), \wh H^1(\s)).
 \end{align}
\end{prop}
\begin{proof}
 Theorem \ref{T:I2} implies that
\[
\overline\omega(f):=-c_\Theta(1+a_\mu \bA(f))^{-1}[f']
\] 
  is the unique solution to \eqref{CEOS} in $\wh H^1(\s),$ and the regularity property \eqref{oo} follows from Lemma~\ref{L:31}.
\end{proof}

In view of Proposition \ref{P:O},    \eqref{P} is equivalent to the  equation \eqref{MPOS}, where $\Phi:H^2(\s)\to \wh H^1(\s)$ is given by
\begin{equation}\label{PHOS}
\Phi(f):=\bB(f)[\overline\omega(f)]=-c_\Theta\bB(f)[(1+a_\mu\bA(f))^{-1}[f']],
\end{equation} 
and it satisfies
\begin{align}\label{REOS}
\Phi\in C^\omega( H^2(\s),  \wh H^1(\s))\cap C^\omega(H^2(\s),  H^1(\s)),
\end{align}
cf. \eqref{E:RB} and \eqref{oo}. With respect to our goal of proving  Theorem \ref{MT:2}, the fact that $\Phi$ maps in $\wh H^1(\s)$ is not relevant, and therefore we shall not rely in this part on this property, but consider instead $\Phi$ as a  mapping in $H^1(\s)$. 
In view of Lemma \ref{L:221} and Proposition \ref{P:O}  the Rayleigh-Taylor condition \eqref{RT} can be reformulated as
   \begin{equation}\label{RTC}
   a_{\text{\tiny RT}}:=c_{\Theta} + a_\mu \Phi(f_0)>0.
  \end{equation}
 Since $\Phi(f_0)\in \wh H^1(\s),$ it follows that \eqref{RTC} can hold only if $\Theta>0$.
We also note  that \eqref{REOS} ensures that the set $\cO$ of all initial data that satisfy the Rayleigh-Taylor condition \eqref{RTC}, that is
\[
\cO=\{f_0\in H^2(\s)\,:\, c_{\Theta} + a_\mu \Phi(f_0)>0\}
\]
is an open subset of $H^2(\s)$ which is nonempty as  it contains for example all constant functions. 

In the following we fix an arbitrary  $f_0\in\cO$ and prove the generator property  \eqref{GPOS} for the operator
\begin{align}\label{PPH}
\p\Phi(f_0)[f]=\p\bB(f_0)[f][\overline\omega_0]+\bB(f_0)[\p\overline\omega(f_0)[f]],
\end{align}
where 
\begin{align}\label{ooo}
\overline\omega_0:=\overline\omega(f_0)
\end{align}
is defined in Proposition \ref{P:O}.
In view of \eqref{CEOS} and of Proposition \ref{P:O}, we determine $\p\overline\omega(f_0)[f]$ as the solution to the equation
     \[
(1+a_\mu \bA(f_0))[\p\overline\omega(f_0)[f]] =-c_\Theta f'-a_\mu\p\bA(f_0)[f][\overline\omega_0],
  \] 
 where, combining the Lemmas \ref{L:31} and \ref{L:A1} $(i)$,  we get
\begin{equation}\label{PDA} 
 \begin{aligned}
 \p\bA(f_0)[f][\overline\omega_0]&=-f'\bB_2(f_0)[\overline\omega_0]-f_0'\p\bB_2(f_0)[f][\overline\omega_0]-\p\bB_1(f_0)[f][\overline\omega_0]\\[1ex]
 &\hspace{0.424cm}+{\pi}^{-1}\big[f'C_{0,1}(f_0)[\overline\omega_0]-2f_0'C_{2,2}(f_0,f_0)[f,f_0,\overline\omega_0]-C_{1,1}(f_0)[f,\overline\omega_0]\\[1ex]
 &\hspace{0.424cm}+2C_{3,2}(f_0,f_0)[f,f_0,f_0,\overline\omega_0]\big],\qquad f\in H^2(\s).
 \end{aligned}
 \end{equation}
 
 Establishing \eqref{GPOS} is now more difficult than for the Muskat problem with surface tension, because  there are several 
 leading order terms to
  be considered when dealing with $\p\Phi(f_0)$, see the proof of Theorem \ref{T:K2}.
Besides, the Rayleigh-Taylor condition \eqref{RTC} does not appear in a natural way in the analysis and it has to be 
artificially built in instead.
Indeed, let us first conclude from the Lemmas~\ref{L:32} and \ref{L:A1} that
\begin{equation}\label{PDH}
\begin{aligned}
 \p\bB(f_0)[f][\overline\omega_0]&=f'\bB_1(f_0)[\overline\omega_0]+f_0'\p\bB_1(f_0)[f][\overline\omega_0]-\p\bB_2(f_0)[f][\overline\omega_0]\\[1ex]
 &\hspace{0.424cm}-2\pi^{-1}C_{2,2}(f_0,f_0)[f,f_0,\overline\omega_0]+\pi^{-1}f'C_{1,1}(f_0)[f_0,\overline\omega_0]\\[1ex]
 &\hspace{0.424cm}+\pi^{-1}f'_0C_{1,1}(f_0)[f,\overline\omega_0]-2\pi^{-1}f_0'C_{3,2}(f_0,f_0)[f,f_0,f_0,\overline\omega_0],
\end{aligned}
\end{equation} 
and let   
\[[\tau\mapsto\Psi(\tau)]:[0,1]\to\kL(H^2(\s), \wh H^1(\s)),\] 
   denote the continuous path defined by
   \[
\Psi(\tau)[f]:=\tau\p\bB(f_0)[f][\overline\omega_0] +\bB(\tau f_0)[w(\tau)[f]],   
   \]
where
\begin{equation}\label{WTF}
\begin{aligned}
w(\tau)[f]&:=-(1+a_\mu\bA(\tau f_0))^{-1}\big[c_\Theta f'+\tau a_\mu\p\bA(f_0)[f][\overline\omega_0]\\[1ex]
&\hspace{3.95cm}+(1-\tau)a_\mu \big(f'\Phi(f_0)-\langle f'\Phi(f_0)\rangle\big)\big].
\end{aligned}
\end{equation}
The  function defined in \eqref{WTF} is   related   to $\p\overline\omega(f_0)[f]$.
We emphasize that the last term on the right hand side of \eqref{WTF}   has been introduced artificially with the purpose   of 
  identifying the function $ a_{\text{\tiny RT}}$ when setting $\tau=0$, but also when 
relating $\Psi(\tau)$ locally to certain Fourier multipliers, see Theorem~\ref{T:K2} below. 
If $\tau=1$, it follows that $\Psi(1)=\p\Phi(f_0),$ while for $\tau=0$ we get  
\begin{align}\label{Psi0}
\Psi(0)[f]=\bB(0)[w(0)[f]]=-H[f'a_{\text{\tiny RT}}-\langle f'a_{\text{\tiny RT}}\rangle]=-H[f'a_{\text{\tiny RT}}],
\end{align}
where we used once more the relation  $\bB(0)=H.$ We note that, since $a_{\text{\tiny RT}}$ is in general 
not constant, the operator $\Psi(0)$ is in general not a Fourier multiplier. 
However, we may benefit from the simpler structure of $\Psi(0)$, compared to 
that of $\p\Phi(f_0)$, and the fact that the Rayleigh-Taylor condition 
holds to show that  large real numbers belong to the spectrum 
 of $\Psi(0)$, see Proposition \ref{P:GOS}.

We now derive some estimates for the operator $w\in C([0,1], \kL(H^2(\s), \wh H^1(\s)))$, which are needed later on in the analysis.  
Let therefore $\tau'\in(1/2,1)$.
Since $\Phi(f_0)\in H^1(\s)$, it follows from Theorem~\ref{T:I1}  and \eqref{DIR2} (with $r=1+\tau'$) 
    there exists a constant $C>0$ such that 
  \begin{align}
  \|w(\tau)[f]\|_2&\leq C\|f\|_{H^{1+\tau'}} \label{DFF1}
\end{align}
for all $f\in H^2(\s)$ and  $\tau\in[0,1].$ Furthermore, Theorem \ref{T:I2}  and \eqref{DIR2} show that additionally
 \begin{align}
  \|w(\tau)[f]\|_{H^1}&\leq C\|f\|_{H^{2}}. \label{DFF2}
\end{align}
 Using the interpolation property \eqref{IP}, we conclude from \eqref{DFF1}-\eqref{DFF2} that 
 \begin{align}
 \|w(\tau)[f]\|_{H^{\tau'}}&\leq C\|f\|_{H^{1+2{\tau'}-{\tau'}^2}} \label{DFF3}
\end{align}
for all $f\in H^2(\s)$ and $\tau\in[0,1].$

The following result is the main step towards proving the generator property \eqref{GPOS}.
Below $(-\p_x^2)^{1/2} $ stands for the Fourier multiplier with symbol $(|k|)_{k\in\Z}$, and the following identity is used
\[
(-\p_x^2)^{1/2}[f]=H[f']=\bB(0)[f']\qquad \text{for all $f\in H^1(\s)$.}
\]

\begin{thm}\label{T:K2} 
Let  $f_0\in H^2(\s)$   and     $\mu>0$ be given.
Then, there exist $p\geq3$, a   $p$-partition of unity  $\{\pi_j^p\,:\, 1\leq j\leq 2^{p+1}\} $, a constant $K=K(p)$, and for each  $ j\in\{1,\ldots,2^{p+1}\}$ and $\tau\in[0,1]$ there 
exist   operators $$\bA_{j,\tau}\in\kL(H^2(\s), H^1(\s))$$
 such that 
 \begin{equation}\label{DEOS}
  \|\pi_j^p\Psi(\tau )[f]-\bA_{j,\tau}[\pi^p_j f]\|_{H^1}\leq \mu \|\pi_j^p f\|_{H^2}+K\|  f\|_{H^{31/16}}
 \end{equation}
 for all $ j\in\{1,\ldots, 2^{p+1}\}$, $\tau\in[0,1],$ and  $f\in H^2(\s)$. 
 The operator $\bA_{j,\tau}$ is defined  by 
  \begin{align} 
 \bA_{j,\tau }:=&-  \alpha_\tau (x_j^p)(-\p_x^2)^{1/2}+\beta_\tau(x_j^p)\p_x,\label{FM2}
 \end{align}
 where $x_j^p\in I_j^p$ is arbitrary, but fixed, and where
 \begin{equation*}
 \alpha_\tau:= \frac{1+(1-\tau) f_0'^2}{1+f_0'^2} a_{\text{\tiny{\em RT}}}\qquad\text{and}\qquad\beta_\tau:=\tau\Big(\bB_1(f_0)[\overline\omega_0]+\pi^{-1}C_{1,1}(f_0)[f_0,\overline\omega_0]
 +a_\mu\frac{\overline\omega_0 }{1+\tau^2f_0'^2}\Big).
 \end{equation*}
 \end{thm}
\begin{proof} 
 Let   $p\geq 3$ be an integer which we fix later on in this proof and let  $\{\pi_j^p\,:\, 1\leq j\leq 2^{p+1}\}$ be a 
 $p$-partition of unity, respectively 
  let $\{\chi_j^p\,:\, 1\leq j\leq 2^{p+1}\} $ be a family associated to this  partition.
We denote by $C$  constants which are
independent of $p\in\N$, $f\in H^2(\s)$, $\tau\in [0,1]$, and $j \in \{1, \ldots, 2^{p+1}\}$, while the constants  denoted by $K$ may depend only upon $p.$\medskip

\noindent{\em The lower order terms.} 
We first note that
\begin{align*} 
  \|\pi_j^p\Psi(\tau )[f]-\bA_{j,\tau}[\pi^p_j f]\|_{H^1}&\leq \|\pi_j^p\Psi(\tau )[f]-\bA_{j,\tau}[\pi^p_j f]\|_{2}+\|(\pi_j^p\Psi(\tau )[f]-\bA_{j,\tau}[\pi^p_j f])'\|_{2}\\[1ex]
  &\leq \|\pi_j^p\Psi(\tau )[f]-\bA_{j,\tau}[\pi^p_j f]\|_{2}+\|(\pi_j^p)'\Psi(\tau )[f]\|_2\\[1ex]
  &\hspace{0.424cm}+\|\pi_j^p(\Psi(\tau )[f])'-\bA_{j,\tau}[(\pi^p_j f)']\|_{2}.
 \end{align*}
The relations \eqref{E:RB} (with $r=7/4$) and \eqref{DFF1} (with $\tau'=3/4$) yield
 \[
 \|\pi_j^p\Psi(\tau )[f]\|_2+ \|(\pi_j^p)'\Psi(\tau )[f]\|_2\leq K\| \Psi(\tau )[f]\|_2\leq K\|f\|_{H^{7/4}},
 \]
 and since $\max_{ \tau\in[0,1]}(\|\alpha_\tau\|_{H^1}+\|\beta_\tau\|_{H^1})\leq C,$ it also holds that
  \[
 \|\bA_{j,\tau}[\pi^p_j f]\|_2 \leq K\|f\|_{H^{1}}.
 \]
 Therewith we get
 \begin{align*} 
  \|\pi_j^p\Psi(\tau )[f]-\bA_{j,\tau}[\pi^p_j f]\|_{H^1}&\leq  \|\pi_j^p(\Psi(\tau )[f])'-\bA_{j,\tau}[(\pi^p_j f)']\|_{2}+K\|f\|_{H^{7/4}}.
 \end{align*}
 Moreover,  combining  \eqref{PDH}, \eqref{EG2} (with $r=7/4$ and $\tau=3/4$), Lemma \ref{L:A1} $(ii)$ (with $\tau=3/4$ and $r=15/8$), and \eqref{DFF3} (with $\tau'=3/4$),
 we may write
 \begin{align*}
 (\Psi(\tau)[f])'&=\bB_3(\tau f_0)[(w(\tau)[f])']+\tau f''\big(\bB_1(f_0)[\overline\omega_0]+\pi^{-1}C_{1,1}(f_0)[f_0,\overline\omega_0]\big)+\tau\pi^{-1}f_0'C_{0,1}(f_0)[(f'\overline\omega_0)']\\[1ex]
 &\hspace{0.424cm}-2\tau\pi^{-1}C_{1,2}(f_0,f_0)[f_0,(f'\overline\omega_0)']-2\tau\pi^{-1}f_0'C_{2,2}(f_0,f_0)[f_0,f_0,(f'\overline\omega_0)'] +T_{\rm lot}^{\Psi,\tau}[f],
 \end{align*}
 where
 \[
\|T_{\rm lot}^{\Psi,\tau}[f]\|_2\leq C\|f\|_{H^{31/16}}.
 \]
 Consequently, we are left to estimate the $L_2$-norm of the difference
 \begin{align*}
 & \pi^p_j \bB_3(\tau f_0)[(w(\tau)[f])']+\tau\pi^p_j  f''\big(\bB_1(f_0)[\overline\omega_0]+\pi^{-1}C_{1,1}(f_0)[f_0,\overline\omega_0]\big)+\tau\pi^{-1}\pi^p_jf_0'C_{0,1}(f_0)[(f'\overline\omega_0)']\\[1ex]
 &\hspace{0.424cm}-2\tau\pi^{-1}\pi^p_j C_{1,2}(f_0,f_0)[f_0,(f'\overline\omega_0)'] -2\tau\pi^{-1}\pi^p_jf_0'C_{2,2}(f_0,f_0)[f_0,f_0,(f'\overline\omega_0)'] -\bA_{j,\tau}[(\pi^p_j f)'].
 \end{align*}
 \noindent{\em Higher order terms I.}
 Given $1\leq j\leq 2^{p+1}, $ we set 
 \begin{align*} 
 \bA^1_{j,\tau}:=  \big(\bB_1(f_0)[\overline\omega_0]+\pi^{-1}C_{1,1}(f_0)[f_0,\overline\omega_0]\big)(x_j^p)\p_x.
 \end{align*}
 Since $\bB_1(f_0)[\overline\omega_0],\, C_{1,1}(f_0)[f_0,\overline\omega_0]\in H^1(\s)\hookrightarrow{\rm C}^{1/2}(\s)$ and $\chi_j^p\pi_j^p=\pi_j^p$, it follows that 
\begin{equation}\label{ABA1} 
 \begin{aligned}
 &\hspace{-1cm}\| \pi^p_j f''\big(\bB_1(f_0)[\overline\omega_0]+\pi^{-1}C_{1,1}(f_0)[f_0,\overline\omega_0]\big)-\bA^1_{j,\tau}[(\pi_j^pf)']\|_2\\[1ex]
 &\leq K\|f\|_{H^1}+\|\chi_j^p(\bB_1(f_0)[\overline\omega_0]-\bB_1(f_0)[\overline\omega_0](x_j^p))\|_\infty\|\pi_j^p f\|_{H^2}\\[1ex]
 &\hspace{0.424cm}+\|\chi_j^p(C_{1,1}(f_0)[f_0,\overline\omega_0]-C_{1,1}(f_0)[f_0,\overline\omega_0](x_j^p))\|_\infty\|\pi_j^p f\|_{H^2}\\[1ex]
 &\leq \frac{\mu}{4}\|\pi_j^pf\|_{H^2}+K\|f\|_{H^1},
 \end{aligned}
 \end{equation}
 provided that $p$ is sufficiently large.\medskip
 
 \noindent{\em Higher order terms II.} Letting 
 \begin{align*}
 \bA^2_{j,\tau}:= \frac{\overline\omega_0(x_j^p)f_0'(x_j^p)}{(1+f_0'^2(x_j^p))^2} (-\p_x^2)^{-1/2},
 \end{align*}
 it holds that
 \[
 \pi^p_j C_{1,2}(f_0,f_0)[f_0,(f'\overline\omega_0)']-\pi\bA^2_{j,\tau}[(\pi_j^p f)'] =T_1[f]+T_2[f]+T_3[f],
 \]
 where
 \begin{align*}
 T_1[f]&=\pi^p_j C_{1,2}(f_0,f_0)[f_0,(f'\overline\omega_0)']-C_{1,2}(f_0,f_0)[f_0,\pi_j^p(f'\overline\omega_0)'],\\[1ex]
  T_2[f]&= C_{1,2}(f_0,f_0)[f_0,\pi_j^p(f'\overline\omega_0)']- \frac{ f_0'(x_j^p)}{(1+f_0'^2(x_j^p))^2}C_{0,0}[\pi_j^p(f'\overline\omega_0)'],\\[1ex]
  T_3[f]&=  \frac{ f_0'(x_j^p)}{(1+f_0'^2(x_j^p))^2}C_{0,0}[\pi_j^p(f'\overline\omega_0)']-\pi\bA^2_{j,\tau}[(\pi_j^p f)'].
\end{align*}  
 The first term may be estimated, by using integration by parts, in a similar way as the term $T_{11}[h]$ 
 in the proof of Theorem \ref{T:K1}, that is 
 \begin{align*}
 \|T_1[f]\|_2\leq K\|f'\overline\omega_0\|_2\leq K\|f\|_{H^1}.
 \end{align*}
 Besides, the same arguments used to derive \eqref{MM4} show that for $p$ sufficiently large 
 \begin{align*}
 \|T_2[f]\|_2\leq \frac{\mu}{16}\|\pi_j^pf\|_{H^2}+K\|f\|_{H^1}.
 \end{align*}
 Finally, it holds that
 \begin{align*}
 \|T_3[f]\|_2&\leq \|C_{0,0}[(\pi_j f)''(\overline\omega_0-\overline\omega_0(x_j^p))]\|_2+\|C_{0,0}[\pi_j^p f'
 \overline\omega_0']\|_2+\|C_{0,0}[((\pi_j^p)'' f+2(\pi_j^p)'f')\overline\omega_0]\|_2\\[1ex]
 &\hspace{0.424cm}+\|\overline\omega_0\|_\infty\Big\|\int_{-\pi}^\pi\Big[\frac{1}{t_{[s]}}-\frac{1}{s/2}\Big](\pi_j f)''(\cdot-s)\, ds\Big\|_2,
 \end{align*}
 and, recalling that $\chi_j^p=1$ on $\supp \pi_j^p$, we obtain, by using integration by parts, Lemma \ref{L:A1} $(i)$, and 
 the fact that $\overline\omega_0\in H^1(\s)\hookrightarrow{\rm C}^{1/2}(\s)$    the estimate
  \begin{align*}
 \|T_3[f]\|_2&\leq \|C_{0,0}[(\pi_j f)''\chi_j^p(\overline\omega_0-\overline\omega_0(x_j^p))]\|_2+ K\|f\|_{H^1}\leq C\|\pi_j f\|_{H^2}\|\chi_j^p(\overline\omega_0
 -\overline\omega_0(x_j^p))\|_\infty+ K\|f\|_{H^1}\\[1ex]
 &\leq \frac{\mu}{16}\|\pi_j^pf\|_{H^2}+K\|f\|_{H^1},
 \end{align*}
provided $p$ is sufficiently large.
 Summarizing, we have shown that 
 \begin{equation}\label{ABA2} 
 2\|\pi^p_j C_{1,2}(f_0,f_0)[f_0,(f'\overline\omega_0)']-\pi\bA^2_{j,\tau}[(\pi_j^p f)']\|_{2}\leq \frac{\mu}{4}\|\pi_j^pf\|_{H^2}+K\|f\|_{H^1}
 \end{equation}
and similarly we get
  \begin{equation}\label{ABA3} 
  \begin{aligned}
  & \|\pi^p_j f_0'C_{0,1}(f_0)[(f'\overline\omega_0)']-\pi(1+f_0'^2(x_j^p))\bA^2_{j,\tau}[(\pi_j^p f)']\|_{2}\\[1ex]
 &+2\|\pi^p_j f_0'C_{2,2}(f_0,f_0)[f_0,f_0(f'\overline\omega_0)']-\pi f_0'^2(x_j^p)\bA^2_{j,\tau}[(\pi_j^p f)']\|_{2}\\[1ex]
 &\hspace{1cm}\leq \frac{\mu}{4}\|\pi_j^pf\|_{H^2}+K\|f\|_{H^1}.
  \end{aligned}
 \end{equation}
 
 \noindent{\em Higher order terms III.}   We are left to consider the function 
 \begin{align}\label{ST12}
 \pi^p_j \bB_3(\tau f_0)[(w(\tau)[f])']=\pi^{-1}\pi_j^p\big( C_{0,1}(f_{\tau_0})[w']+  f_{\tau_0}'C_{1,1}(f_{\tau_0})[f_{\tau_0}, w']\big),
 \end{align}
where, for the sake of brevity, 
we have set 
\[\text{$ f_{\tau_0}:=\tau f_{0} $ \qquad\text{and}\qquad $ w:=w(\tau)[f] $}.\]
 Let further
 \[
\phi_\tau:=a_{\text{\tiny RT}}- \tau a_\mu f_0'\big(\bB_1(f_0)[\overline\omega_0]+\pi^{-1}C_{1,1}(f_0)[f_0,\overline\omega_0]\big)\in H^1(\s).
 \]
 We first derive an estimate  for  the $L_2$-norm of $\pi_j^p w'.$
 To this end we differentiate \eqref{WTF} once  to obtain, in view of \eqref{Aprime1}, \eqref{PDA}, and Lemma \ref{L:A1} $(i)$-$(ii)$,
that 
\begin{equation}\label{wta}
\begin{aligned}
(1+a_\mu\bA(f_{\tau_0}))[(\pi_j^p w)']&=-\phi_\tau\pi_j^p f''+T^{w,j,\tau}_{\rm lot}[f]+\tau a_\mu\pi^{-1}\big(2f_0'C_{1,2}(f_0,f_0)[f_0,\pi_j^p(f'\overline\omega_0)']\\[1ex]
&\hspace{0.424cm}+C_{0,1}(f_0)[\pi_j^p(f'\overline\omega_0)'] -2C_{2,2}(f_0,f_0)[f_0,f_0,\pi_j^p(f'\overline\omega_0)']\big).
\end{aligned}
\end{equation}
Combining \eqref{EG2} (with $\tau=3/4$ and $r=7/4$), \eqref{Aprime3} (with $\tau=3/4$), \eqref{MM3b'}, \eqref{DFF1} and \eqref{DFF3} (both with $\tau'=3/4$), 
 and Lemma \ref{L:A1} $(i)$-$(ii)$ (with $\tau=3/4$ and $r=15/8$)  we get that
 \begin{align}\label{wta'}
\|T^{w,j,\tau}_{\rm lot}[f]\|_2 \leq K\|f\|_{H^{31/16}}.
 \end{align}
The relation \eqref{wta} together with Theorem \ref{T:I1}, Lemma \ref{L:A1} $(i)$, and \eqref{DFF1} (with $\tau'=3/4$)  now yields
\begin{equation}\label{WTA}
\|\pi_j^p w'\|_2\leq \|(\pi_j^pw)'\|_2+\|(\pi_j^p)'w\|_2\leq C\|\pi_j^p f\|_{H^2}+K\|f\|_{H^{31/16}}.
\end{equation}
 
 We now consider the second term on the right hand side of \eqref{ST12}.
 Letting 
 \begin{align*} 
 \bA^3_{j,\tau}:= \frac{ f_{\tau_0}'^2(x_j^p)}{1+f_{\tau_0}'^2(x_j^p)}\Big[-\phi_\tau(x_j^p)(-\p_x^2)^{1/2}-\tau a_\mu \frac{\overline\omega_0(x_j^p)}{1+f_{\tau_0}'^2(x_j^p)}\p_x\Big],
 \end{align*}
we write
 \[
 \pi^p_j f_{\tau_0}'C_{1,1}(f_{\tau_0})[f_{\tau_0}, w']-\pi\bA^3_{j,\tau}[(\pi_j^p f)'] =T_4[f]+T_5[f]+T_6[f],
 \]
 where
 \begin{align*}
 T_4[f]&=\pi^p_j f_{\tau_0}'C_{1,1}(f_{\tau_0})[f_{\tau_0}, w']- f_{\tau_0}'(x_j^p)C_{1,1}(f_{\tau_0})[f_{\tau_0},\pi^p_j w'],\\[1ex]
  T_5[f]&=f_{\tau_0}'(x_j^p)C_{1,1}(f_{\tau_0})[f_{\tau_0},\pi^p_j w']- \frac{ f_{\tau_0}'^2(x_j^p)}{1+f_{\tau_0}'^2(x_j^p)}C_{0,0}[\pi_j^p w'],\\[1ex]
  T_6[f]&=  \frac{ f_{\tau_0}'^2(x_j^p)}{ 1+f_{\tau_0}'^2(x_j^p) }C_{0,0}[\pi_j^p w']-\pi\bA^3_{j,\tau}[(\pi_j^p f)'].
\end{align*}  
The arguments that led to \eqref{MM3} together with \eqref{WTA} show that
\begin{align*}
\|T_4[f]\|_2\leq \frac{\mu}{24}\|\pi_j^pf\|_{H^2}+K\|f\|_{H^{31/16}},
\end{align*}
provided that $p$ is sufficiently large, while arguing as in the derivation of \eqref{MM4} we obtain that
 \begin{align*}
\|T_5[f]\|_2\leq \frac{\mu}{24}\|\pi_j^pf\|_{H^2}+K\|f\|_{H^{31/16}}.
\end{align*}
Concerning $T_6[f]$, we find, by using fact that the  Hilbert transform satisfies $H^2=-{\rm id}_{L_2(\s)}$, the following relation  
\begin{align*}
\|T_6[f]\|_2&\leq \Big\|C_{0,0}[\pi_j^p w']+\pi \phi_\tau(x_j^p)H[(\pi_j^p f)'']-\tau a_\mu\pi \frac{\overline\omega_0(x_j^p)}{1+f_{\tau_0}'^2(x_j^p)}H^2[(\pi_j^pf)'']\Big\|_2,
\end{align*}
and, since integration by parts and \eqref{DFF1} (with $\tau'=3/4$) yield
\begin{align*}
\|C_{0,0}[\pi_j^p w']-\pi H[\pi_j^p w']\|_2+\|C_{0,0}[\pi_j^p f'']-\pi H[(\pi_j^p f)'']\|_2\leq K\|w\|_2 \leq K\|f\|_{H^{7/4}},
\end{align*}
we conclude that
\begin{align*}
\|T_6[f]\|_2&\leq \Big\| \pi_j^p w' +  \phi_\tau(x_j^p) (\pi_j^p f)'' -\tau a_\mu\frac{\overline\omega_0(x_j^p)}{1+f_{\tau_0}'^2(x_j^p)}H[(\pi_j^pf)'']\Big\|_2+K\|f\|_{H^{7/4}}\\[1ex]
 &\leq \Big\| \pi_j^p w' +  \phi_\tau(x_j^p)  \pi_j^p f '' -\frac{\tau a_\mu}{\pi} \frac{\overline\omega_0(x_j^p)}{1+f_{\tau_0}'^2(x_j^p)}C_{0,0}[\pi_j^pf'']\Big\|_2+K\|f\|_{H^{7/4}}.
\end{align*}
Combining \eqref{FAA} and \eqref{wta}, we further get 
\begin{align*}
&\hspace{-1cm}\Big\| \pi_j^p w' +  \phi_\tau(x_j^p)  \pi_j^p f '' -\frac{\tau a_\mu}{\pi} \frac{\overline\omega_0(x_j^p)}{1+f_{\tau_0}'^2(x_j^p)}C_{0,0}[\pi_j^pf'']\Big\|_2\\[1ex]
&\leq \|\chi_j^p(\phi_\tau-\phi_\tau(x_j^p))\|_\infty\|\pi_j^pf''\|_2+\|(1+a_\mu \bA(f_{\tau_0})[(\pi_j^p)'w]\|_2\\[1ex]
&\hspace{0.424cm}+  \Big\| f_0'C_{1,2}(f_0,f_0)[f_0,\pi_j^p(f'\overline\omega_0)']- \frac{ \overline\omega_0(x_j^p)f_0'^2(x_j^p)}{ (1+f_0'^2(x_j^p) )^2}C_{0,0}[\pi_j^p f'']\Big\|_2\\[1ex]
&\hspace{0.424cm}+  \Big\| C_{0,1}(f_0)[\pi_j^p(f'\overline\omega_0)']- \frac{ \overline\omega_0(x_j^p)}{ 1+f_0'^2(x_j^p)}C_{0,0}[\pi_j^p f'']\Big\|_2\\[1ex]
&\hspace{0.424cm}+  \Big\| C_{2,2}(f_0,f_0)[f_0,f_0,\pi_j^p(f'\overline\omega_0)']- \frac{ \overline\omega_0(x_j^p)f_0'^2(x_j^p)}{ (1+f_0'^2(x_j^p))^2}C_{0,0}[\pi_j^p f'']\Big\|_2\\[1ex]
&\hspace{0.424cm}+\|T_{\rm lot}^{w,j,\tau}[f]\|_2+\|f_{\tau_0}'\bB_2(f_{\tau_0})[\pi_j^pw']\|_2+\| \bB_1(f_{\tau_0})[\pi_j^pw']\|_2\\[1ex]
&\hspace{0.424cm}+  \Big\| f_{\tau_0}'C_{0,1}(f_{\tau_0})[\pi_j^pw']- \frac{ f_{\tau_0}'(x_j^p)}{ 1+f_{\tau_0}'^2(x_j^p)}C_{0,0}[\pi_j^p w']\Big\|_2\\[1ex]
&\hspace{0.424cm}+  \Big\|  C_{1,1}(f_{\tau_0})[f_{\tau_0},\pi_j^pw']- \frac{ f_{\tau_0}'(x_j^p)}{ 1+f_{\tau_0}'^2(x_j^p)}C_{0,0}[\pi_j^p w']\Big\|_2,
\end{align*}
and the estimates \eqref{MM3b'}, \eqref{DFF1} (with $\tau'=3/4$), \eqref{wta'}, together with the arguments used to estimate $\|T_2[f]\|_2$ show, for $p$ sufficiently large, that 
 \begin{align*}
 \|T_6[f]\|_2\leq   \frac{\mu}{24}\|\pi_j^pf\|_{H^2}+K\|f\|_{H^{31/16}}.
 \end{align*}
 Altogether, we have shown that 
  \begin{align*}
 \|\pi^p_j f_{\tau_0}'C_{1,1}(f_{\tau_0})[f_{\tau_0}, w']-\pi\bA^3_{j,\tau}[(\pi_j^p f)'] \|_2\leq   \frac{\mu}{8}\|\pi_j^pf\|_{H^2}+K\|f\|_{H^{31/16}}.
 \end{align*}
 
 Letting 
 \begin{align*} 
 \bA^4_{j,\tau}:= \frac{ 1}{1+f_{\tau_0}'^2(x_j^p)}\Big[-\phi_\tau(x_j^p)(-\p_x^2)^{1/2}-\tau a_\mu \frac{\overline\omega_0(x_j^p)}{1+f_{\tau_0}'^2(x_j^p)}\p_x\Big],
 \end{align*}
 we obtain in a similar way, that 
 \begin{align*}
 \|\pi^p_j  C_{0,1}(f_{\tau_0})[  w']-\pi\bA^4_{j,\tau}[(\pi_j^p f)'] \|_2\leq   \frac{\mu}{8}\|\pi_j^pf\|_{H^2}+K\|f\|_{H^{31/16}},
 \end{align*}
 provided that $p$  is sufficiently large, and therewith we conclude that 
\begin{align}\label{ABA4}
 \|\pi^p_j \bB_3(\tau f_{0})[ (w(\tau)[f])']-  (\bA^3_{j,\tau}+ \bA^4_{j,\tau})[(\pi_j^p f)'] \|_2\leq   \frac{\mu}{4}\|\pi_j^pf\|_{H^2}+K\|f\|_{H^{31/16}}.
 \end{align}
 
 \noindent{\em Final step.} Using  the identity $\overline\omega_0=-c_\Theta f_0'-a_\mu\bA(f_0)[\overline\omega_0],$ it is not difficult to see that  
 \[
 \bA_{j,\tau} =\tau\big[\bA^1_{j,p}-2\bA^2_{j,p}+(1+f_0'^2(x_j^p))\bA^2_{j,p} -2f_0'^2(x_j^p)\bA^2_{j,p}\big]+\bA^3_{j,p}+\bA^4_{j,p},
 \]  
 and \eqref{ABA1}, \eqref{ABA2}, \eqref{ABA3}, and \eqref{ABA4} immediately yield \eqref{DEOS}.
\end{proof}

Making use of the fact that for  $f_0\in\cO$  the Rayleigh-Taylor condition $a_{\text{\tiny RT}}>0$ is satisfied, it follows from the general result in Proposition \ref{P:GOS} below
  that $\Psi(0)$ contains in its resolvent set all sufficiently large
 real numbers.   

\begin{prop}\label{P:GOS}
Let $a\in H^1(\s)$ be a positive function. 
Then, there exists $\omega_0\geq1$ with the property that $\lambda+H[a\p_x]\in {\rm Isom} (H^2(\s), H^1(\s))$ 
for all $\lambda\in[\omega_0,\infty)$.
\end{prop}
\begin{proof}
Let $m:=\min_{\s} a>0$. We introduce the continuous path $[\tau\mapsto B(\tau)]:[0,1]\to \kL(H^2(\s), H^1(\s))$ via
\[
B(\tau):=H[a_\tau \p_x]  \qquad \text{with}\quad  a_\tau:=(1-\tau)m+\tau a\geq m. 
\]
Since $\lambda+B(0)$ is the Fourier multiplier with symbol $(\lambda+m |k|)_{k\in\Z}$, it is obvious that $\lambda+B(0)$ is invertible for all $\lambda>0$.
If $\lambda$ is sufficiently large, we show below that   $\lambda+B(1)=\lambda+H[a\p_x]$ has this property too.
To this end we prove that for each  $\mu>0$ there exists  $p\geq 3$, a
$p$-partition of unity  $\{\pi_j^p\,:\, 1\leq j\leq 2^{p+1}\}$, a constant $K=K(p)$, and for each  $ j\in\{1,\ldots,2^{p+1}\}$ and $\tau\in[0,1]$ there 
exist   operators $$\bB_{j,\tau}\in\kL(H^2(\s), H^1(\s))$$
 such that 
 \begin{equation}\label{DEOSa}
  \|\pi_j^pB(\tau )[f]-\bB_{j,\tau}[\pi^p_j f]\|_{H^1}\leq \mu \|\pi_j^p f\|_{H^2}+K\|  f\|_{H^{7/4}}
 \end{equation}
 for all $ j\in\{1,\ldots, 2^{p+1}\}$, $\tau\in[0,1],$ and  $f\in H^2(\s)$. The operators $\bB_{j,\tau}$ are the Fourier multipliers
 \[
\bB_{j,\tau}:=a_\tau(x_j^p)(-\p_x^2)^{1/2}
 \] 
 with $x_j^p\in I_j^p$.
Indeed, given $p\geq 3$, let $\{\pi_j^p\,:\, 1\leq j\leq 2^{p+1}\}$ be a $p$-partition of unity  and  let
   $\{\chi_j^p\,:\, 1\leq j\leq 2^{p+1}\} $ be a family associated to this  partition.
  Integrating by parts we get
\begin{align*}
 \|\pi_j^pB(\tau )[f]-\bB_{j,\tau}[\pi^p_j f]\|_{H^1}&\leq \|\pi_j^pB(\tau )[f]-\bB_{j,\tau}[\pi^p_j f]\|_{2}  +\|(\pi_j^p)'B(\tau )[f]\|_{2}
 +\|\bB_{j,\tau}[(\pi^p_j)' f]\|_{2}\\[1ex]
&\hspace{0.424cm} +\|\pi_j^p(B(\tau )[f])'-\bB_{j,\tau}[\pi^p_j f']\|_{2}\\[1ex]
&\leq K\|f\|_{H^1}+ \|\pi_j^pH[(a_\tau f')']-a_\tau(x_j^p)H[(\pi^p_j f')']\|_{2}\\[1ex]
&\leq K\|f\|_{H^{7/4}}+ \|\pi_j^pH[a_\tau f''] -a_\tau(x_j^p)H[\pi^p_j f'']\|_{2}\\[1ex]
&\leq K\|f\|_{H^{7/4}}+\|\pi_j^pH[a_\tau f''] - H[a_\tau \pi^p_j f'']\|_{2}+ \|H[(a_\tau-a_\tau(x_j^p)) \pi_j^pf''] \|_{2}\\[1ex]
&\leq K\|f\|_{H^{7/4}}+ \|(a_\tau-a_\tau(x_j^p))\chi_j^p\|_\infty\| \pi_j^pf''] \|_{2}\\[1ex]
&\leq \mu \|\pi_j^p f\|_{H^2}+K\|  f\|_{H^{7/4}}
\end{align*}
provided that $p$ is sufficiently large, and  \eqref{DEOSa} follows.

A simple computation shows   that  there exists  $\kappa\geq1$ such that
\begin{align}\label{baba}
\kappa\|(\lambda+\alpha(-\p_x^2)^{1/2})[f]\|_{H^1}\geq \lambda\cdot\|f\|_{H^1}+\|f\|_{H^2} 
\end{align}
for all $f\in H^2(\s)$, $\alpha\geq m$, and $\lambda\in[1,\infty)$.
Set $\mu:=1/2\kappa$ in \eqref{DEOSa}.
Since $a_\tau\geq m$, it follows from \eqref{DEOSa} and \eqref{baba} that 
\begin{align*}
\kappa\|\pi_j^p(\lambda+B(\tau))[f]\|_{H^1}&\geq\kappa\| (\lambda+\bB_{j,\tau})[\pi_j^pf]\|_{H^1}-\kappa\|\pi_j^pB(\tau )[f]-\bB_{j,\tau}[\pi^p_j f]\|_{H^1}\\[1ex]
&\geq \lambda\cdot\|\pi_j^pf\|_{H^1}+\frac{1}{2}\|\pi_j^pf\|_{H^2}-\kappa K\|  f\|_{H^{7/4}} 
\end{align*}
for all $f\in H^2(\s)$, $\lambda\geq1$, $\tau\in[0,1]$,  and $  j\in\{1,\ldots, 2^{p+1}\}$.
The arguments at the very and of the proof of Theorem \ref{T:GP1} enable us to conclude the existence of  constants $\beta\in(0,1)$ and $\omega_0\geq1$ with 
\begin{align*}
 \| (\lambda+B(\tau))[f]\|_{H^1}&\geq \beta\| f\|_{H^2} 
\end{align*}
for all $f\in H^2(\s)$, $\lambda\geq\omega_0$, and  $\tau\in[0,1]$. The continuity method  
 \cite[Proposition I.1.1.1]{Am95} and the previous observation that  $\lambda+B(0)\in{\rm Isom}(H^2(\s), H^1(\s))$ for $\lambda>0$ yield the desired conclusion.
\end{proof}

We are now in a position to  derive the desired generator property \eqref{GPOS}.

\begin{thm}\label{T:GP2}
Given $f_0\in\cO$, it holds that
\[
-\p\Phi(f_0)\in\kH(H^2(\s), H^1(\s)).
\] 
\end{thm}
\begin{proof}
Given $f_0\in \cO$ and $\tau\in[0,1],$ let $\alpha_\tau$ and $\beta_\tau$ denote the functions introduced in Theorem \ref{T:K2}.
The Rayleigh-Taylor condition $a_{\text{\tiny RT}}>0$ ensures there exists a constant  $\eta\in(0,1)$ such that
\[
\eta\leq \alpha_\tau \leq \frac{1}{\eta}\quad\text{and}\quad \qquad |\beta_\tau|\leq \frac{1}{\eta}
\]
for all $\tau\in[0,1]$.
Given $\alpha\in[\eta,1/\eta]$ and $|\beta|<1/\eta$, let $\bA_{\alpha,\beta}$ denote the Fourier multiplier
\[
\bA_{\alpha,\beta}:=-\alpha (-\p_x^2)^{1/2}+\beta\p_x.
\]
It is not difficult to prove there exists $\kappa_0\geq1 $ such that the complexification of $\bA_{\alpha,\beta}$ (denoted again by $\bA_{\alpha,\beta}$) satisfies 
\begin{align}\label{FMOS}
\kappa_0\|(\lambda-\bA_{\alpha,\beta})[f]\|_{H^1}\geq |\lambda|\cdot\|f\|_{H^1}+\|f\|_{H^2}
\end{align}
for all $\alpha\in[\eta,1/\eta]$, $|\beta|<1/\eta$, $\re \lambda\geq 1$, and $f\in H^2(\s)$. 
Observing that the operators $\bA_{j,\tau}$ found in Theorem \ref{T:K2} belong to the family $\{\bA_{\alpha,\beta}\,:\,\text{$\alpha\in[\eta,1/\eta]$, $|\beta|<1/\eta$}\} $
and that $$\lambda-\Psi(0)=\lambda+H[ a_{\text{\tiny RT}}\p_x]\in{\rm Isom}(H^2(\s), H^1(\s))$$ for all $\lambda\in\R$ which are sufficiently large, cf. Proposition \ref{P:GOS},
the arguments in the proof of Theorem \ref{T:GP1} together with \eqref{FMOS} and Theorem \ref{T:K2} lead us to the desired claim.
\end{proof}
 
We conclude this section with the proof of Theorem \ref{MT:2}.

\begin{proof}[Proof of Theorem \ref{MT:2}]
The proof follows by using the fully nonlinear parabolic theory  in \cite[Chapter 8]{L95}, \eqref{REOS}, and  Theorem \ref{T:GP2}.
The details of proof are identical to those in the nonperiodic case, cf. \cite[Theorem 1.2]{M17x}, and therefore we omit them.
\end{proof}

\pagebreak

 %%%%%%%%%%%%%%%%%%%%%%%%%%%%%%%%%%%%%%%%%%%%%%%%%%%%%%%%%%%%%%%%%%%
%%%%%%%%%%%%%%%%%%%%%%%%%%%%%%%%%%%%%%%%%%%%%%%%%%%%%%%%%%%%%%%%%%%%
%%%%%%%%%%%%%%%%%%%%%%%%%%%%%%%%%%%%%%%%%%%%%%%%%%%%%%%%%%%%%%%%%%%%
%%%%%%%%%%%%%%%%%%%%%%%%%%%%%%%%%%%%%%%%%%%%%%%%%%%%%%%%%%%%%%%%%%%
%%%%%%%%%%%%%%%%%%%%%%%%%%%%%%%%%%%%%%%%%%%%%%%%%%%%%%%%%%%%%%%%%%%%
%%%%%%%%%%%%%%%%%%%%%%%%%%%%%%%%%%%%%%%%%%%%%%%%%%%%%%%%%%%%%%%%%%%%
\section{Stability analysis}\label{Sec:6}
 %%%%%%%%%%%%%%%%%%%%%%%%%%%%%%%%%%%%%%%%%%%%%%%%%%%%%%%%%%%%%%%%%%%
%%%%%%%%%%%%%%%%%%%%%%%%%%%%%%%%%%%%%%%%%%%%%%%%%%%%%%%%%%%%%%%%%%%%
%%%%%%%%%%%%%%%%%%%%%%%%%%%%%%%%%%%%%%%%%%%%%%%%%%%%%%%%%%%%%%%%%%%%
%%%%%%%%%%%%%%%%%%%%%%%%%%%%%%%%%%%%%%%%%%%%%%%%%%%%%%%%%%%%%%%%%%%
%%%%%%%%%%%%%%%%%%%%%%%%%%%%%%%%%%%%%%%%%%%%%%%%%%%%%%%%%%%%%%%%%%%%
%%%%%%%%%%%%%%%%%%%%%%%%%%%%%%%%%%%%%%%%%%%%%%%%%%%%%%%%%%%%%%%%%%%%

In this section we identify the equilibria of the Muskat problem \eqref{P} and study their stability properties. 
 \medskip

\paragraph{\bf The  Muskat problem without surface tension}
We first infer from Remark \ref{R:2}  that  $f\in H^2(\s)$ is a stationary solution to \eqref{P} (with $\sigma=0$) if and only if $f$ is  constant also with respect to $x$.
Besides, as pointed out in Section \ref{Sec:5}, if $f$ is a solution to \eqref{P} as found in Theorem \ref{MT:2}, then  $\Phi(f(t))\in \wh H^1(\s)$
 for all  $t$ in the existence interval of $f$, hence the 
mean integral of   the initial datum  is preserved by the flow. 
Recalling also the invariance property \eqref{HVT}, we shall only address the stability issue for the  $0$ equilibrium under perturbed initial data  with zero integral mean.
Hence, we are led to consider the evolution problem 
\begin{align}\label{MPOS'}
\dot f(t) =\Phi (f(t) ),\quad t\geq0,\qquad f(0)=f_0,
\end{align}
where 
\begin{align}\label{R111}
\Phi\in C^\omega( \wh H^2(\s),\wh  H^1(\s))
\end{align}
is the restriction of  the operator defined in  \eqref{PHOS}.
Recalling \eqref{PPH}, it follows from the relations $\overline\omega(0)=0$, $\bA(0)=0$, and $\bB(0)=H$, that
\[\p\Phi(0)=-c_\Theta H\circ\p_x=-c_\Theta(-\p_x^2)^{1/2}\in\kL( \wh H^2(\s),\wh  H^1(\s)),\]
which identifies the spectrum $\sigma(\p\Phi(0))$ as being the set 
\[\sigma(\p\Phi(0))=\{-c_\Theta |k|\,:\,k\in\Z\setminus\{0\}\}.\] 
Moreover, it is easy to verify that this Fourier multiplier  is the generator of a strongly continuous and analytic semigroup in $\kL(\wh  H^1(\s)).$
This enable us to use the fully nonlinear principle of linearized stability, cf. \cite[Theorem 9.1.1]{L95}, and prove in this way the exponential stability of the zero solution.

\begin{proof}[Proof of Theorem \ref{MT:3}]
The claim follows from \eqref{R111}, the property  $-\p\Phi(0)\in\kH(\wh H^2(\s),\wh  H^1(\s))$, 
and the fact that $\re\lambda\leq-c_\Theta$ for all $\lambda\in\sigma(\p\Phi(0))$ via
     \cite[Theorem 9.1.1]{L95}. 
\end{proof}

\medskip

\paragraph{\bf The  Muskat problem with surface tension}
For $\sigma>0$ the stability analysis is  more intricate. 
Before presenting the complete picture of the equilibria 
we notice that also in this case the mean value of the initial data is preserved by the flow. 
This aspect and  the invariance property \eqref{HVT} 
enable us to restrict our stability analysis to the setting of solutions with zero integral mean.

In view of Remark \ref{R:2}, a function $f\in \wh H^3(\s)$ is a stationary solution to \eqref{P}
if and only if it solves the capillarity equation 
\begin{align}\label{CapEq}
  \frac{f''}{(1+f'^2)^{3/2}}+\lambda f =0\qquad\text{where $\lambda:=-\frac{\Theta}{\sigma}$.}
\end{align} 
This equation has been discussed in detail in \cite{EEM09c}.
If $\lambda\leq 0$, the equation \eqref{CapEq} has by the elliptic maximum principle a unique solution in  $\wh H^3(\s)$, the trivial equilibrium $f=0$.
However, if $\lambda>0$, there may exist also finger-shaped solutions to \eqref{CapEq}, see Figure \ref{Fig:1}, which are all symmetric with respect to the horizontal lines through 
the extrema but also with respect to the points where they intersect the $x$-axis.
In particular, each equilibrium in $\wh H^3(\s)$ is the horizontal translation of an even equilibrium.
We now view $\lambda>0$ as a bifurcation parameter in the equation \eqref{CapEq} and we shall refer to $(\lambda,f)$ as being the solution to \eqref{CapEq}.
 The following theorem provides a complete description of the set of even equilibria to the Muskat problem with surface tension
  (and in virtue of \eqref{HVT} also of the set of all equilibria).
 
  \begin{figure}[h]
 \includegraphics[width=1\textwidth, angle=0]{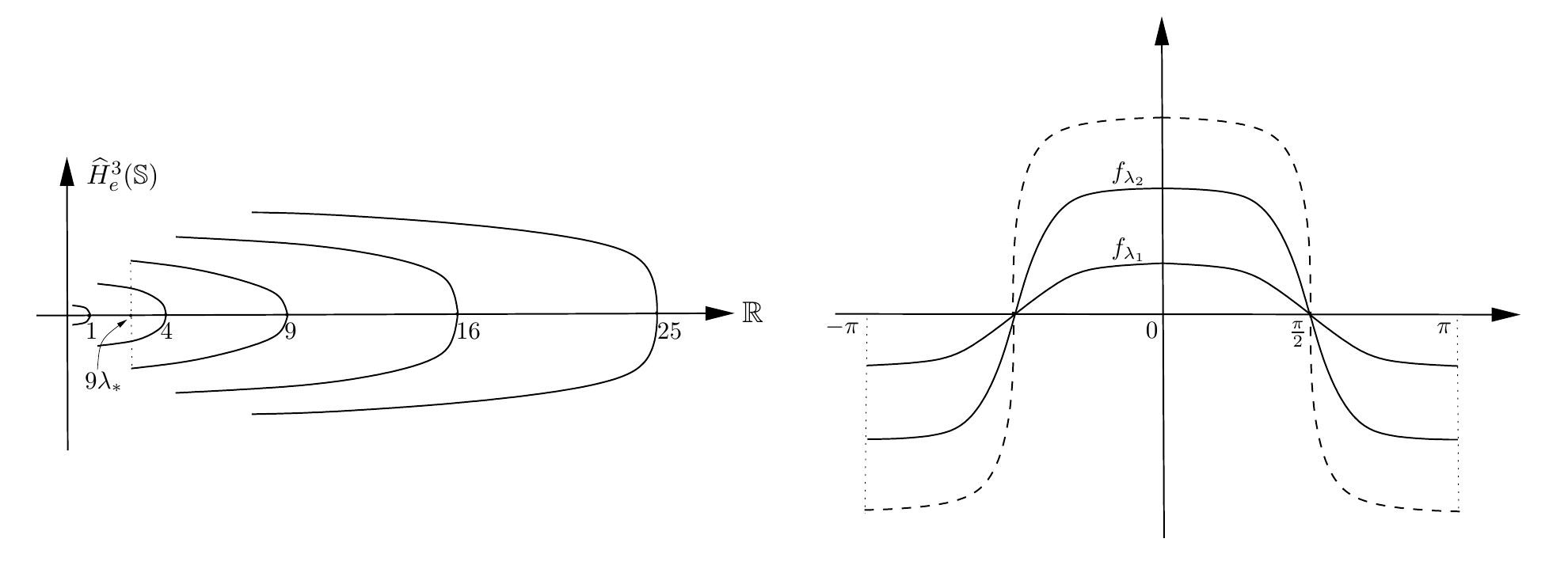}
	\caption{The subcritical global bifurcation branches of \eqref{CapEq} found in Theorem \ref{T:GBP} (left) and the behavior of  the finger-shaped solutions along the first bifurcation branch (right) 
	($\lambda_*<\lambda_2<\lambda_1<1$). The dashed curve is the graph of the function $ \lim_{\lambda\searrow\lambda_*} f_\lambda$ and it has unbounded slope at $x=\pi/2$ and height
	  $\sqrt{2/\lambda_*}.$}
	\label{Fig:1}
\end{figure}
 
 \begin{thm}\label{T:GBP}
 Let 
 $$\lambda_*:=\frac{1}{2\pi^2}B^2\Big(\frac{3}{4},\frac{1}{2}\Big),
 $$
 where $B$ is the beta function.
 The even solutions to \eqref{CapEq} are organized as follows.
 \begin{itemize}
 \item[$(a)$] If $\lambda\leq \lambda_*$,\footnote{A  rough estimate for $\lambda_*$ is $\lambda_*\approx 0.3$.} then \eqref{CapEq} has only the trivial solution.\\[-2ex]
 \item[$(b)$] Let $\lambda>\lambda_* $.
 \begin{itemize}
  \item[$(i)$] The equation \eqref{CapEq} has even solutions  of minimal period $2\pi$ if and only if $\lambda_*<\lambda<1$.  More precisely, for each 
    $\lambda\in(\lambda_*,1)$, \eqref{CapEq} has exactly two even solutions $(\lambda,\pm f_\lambda)$  of minimal period $2\pi$. 
  These solutions are real-analytic, $|f_{\lambda_1}|\leq |f_{\lambda_2}|$ for $\lambda_2<\lambda_1$, $\|f_\lambda\|_\infty\to0$ for $\lambda\nearrow 1$,  and 
  \[
 \text{$\|f_\lambda\|_\infty=|f_\lambda(0)|\nearrow \sqrt{2/\lambda_*}$,\quad $\|f_\lambda'\|_\infty=|f_\lambda'(\pi/2)|\nearrow \infty$
  \quad for $\lambda\searrow \lambda_*.$}
  \]
  \item[$(ii)$] The equation \eqref{CapEq} has even solutions  of minimal period $2\pi/\ell$, $2\leq \ell\in\N$,  if and only if $\ell^{2}\lambda_*<\lambda<\ell^{2}$.
    More precisely, for each 
    $\lambda\in(\ell^{2}\lambda_*,\ell^{2}),$ \eqref{CapEq} has exactly two even solutions $(\lambda,\pm f_\lambda)$ of minimal period $2\pi/\ell$ and
    \[f_\lambda=\ell^{-1}f_{\lambda\ell^{-2}}(\ell \, \cdot\, ) \]
    where $f_{\lambda\ell^{-2}}(\ell \, \cdot\, )$ is the function identified at $(ii)$.
 \end{itemize}
 
 \item[$(c)$] If we  consider \eqref{CapEq} as an abstract bifurcation problem in $\R\times \wh H^3_e(\s),$
  where $$\wh H^3_e(\s):=\{f\in H^3_e(\s)\,:\, \text{$f$ is even}\},$$ then the global bifurcation curve  
 arising from  $(\ell^2,0),$ $1\leq \ell\in\N,$ and described at $(b)$, admits in a neighborhood of $(\ell^2,0)$ a real-analytic parametrization 
  $$(\lambda_\ell,f_\ell):(-\e_\ell,\e_\ell)\to (0,\infty)\times \wh H^3_e(\s)$$ such that
  \[
  \left\{
\begin{array}{lll}
\lambda_\ell(s)=\ell^2-\cfrac{3\ell^4}{8}s^2+O(s^4)  \quad \text{in $\R$,}\\[2ex]
f_\ell(s)= s\cos(\ell x)+O(s^2)  \quad \text{in $ \wh H^3_e(\s)$}
\end{array}
\right.\qquad\text{for $s\to0$.}
  \]
  \end{itemize}
 \end{thm}
 \begin{proof}
 The claims $(a)$ and $(b)$ are established in \cite{EEM09c}.
  The last claim follows by applying   the theorem on bifurcations from simple eigenvalues due to Crandall and Rabinowitz, cf. \cite{CR71}.
   The details are similar to those in the proof of \cite[Theorem 6.1]{EM11a}.
 \end{proof}

With respect to Theorem \ref{T:GBP} we add the following remark.
\begin{rem}\label{R:3}
\begin{itemize}
\item[$(i)$] Because $\lambda_*\approx 0.3$, for certain $\lambda\in(\ell^2\lambda_*,\ell^2)$ with $\ell\geq2$  there exist nontrivial solutions to \eqref{CapEq} with 
minimal period different than $2\pi/\ell$, see Figure \ref{Fig:1}.
\item[$(ii)$] As pointed out in \cite{EEM09c}, these finger-shaped equilibria are in correspondence to certain solutions to the mathematical pendulum equation
\[\theta''+\lambda\sin\theta=0.\] 
\item[$(iii)$] The global bifurcation curves may be continued beyond $\lambda_*\ell^2$, but outside the setting of interfaces parametrized as graphs.
\item[$(iv)$] Because $\lambda_\ell'(0)=0>\lambda_\ell''(0)$, we may assume that $s\lambda_\ell'(s)<0$  for all   $s\in(-\e_\ell,\e_\ell)\setminus\{0\}$.
This aspect is of relevance when studying the stability properties of the finger-shaped equilibria identified above.
\end{itemize}
\end{rem}

 In order to address the stability properties of  the  equilibria to \eqref{P}, we first reformulate the problem by incorporating  $\lambda$ as a parameter.
 To this end we  define  $\Phi:\R\times   \wh H^2(\s)\to \kL(\wh H^3(\s),\wh L_{2}(\s))$  according to
 \begin{align}\label{PLH}
 \Phi(\lambda,f)[h]:=\sigma b_\mu \bB(f)\big[(1+a_\mu\bA(f))^{-1}\Big[\frac{h'''}{(1+f'^2)^{3/2}}-3\frac{f'f''h''}{(1+f'^2)^{3/2}}+\lambda h'\Big],
 \end{align}
where $b_\mu$ is the constant introduced in \eqref{bmu}.
Then, it follows from the analysis in Section \ref{Sec:4} that $\Phi\in C^\omega(\R\times \wh H^2(\s),\kL(\wh H^3(\s), \wh L_{2}(\s))),$ and 
the problem \eqref{P} is equivalent, for  solutions with zero integral mean,  to the quasilinear  evolution problem
\begin{align}\label{MPSL}
\dot f(t) =\Phi(\lambda,f(t))[f(t)],\quad t>0,\qquad f(0)=f_0.
\end{align}
It is not difficult to see that the linearization  $ \Phi(\lambda,0)\in\kL(\wh H^3(\s), \wh L_{2}(\s))$  is a Fourier multiplier 
with spectrum $\sigma( \Phi(\lambda,0))$ that consists only of the eigenvalues  $\{-\sigma b_\mu(|k|^3-\lambda |k|)\,:\,k\in\Z\setminus\{0\}\}.$ 
Moreover, $ \Phi(\lambda,0)$ generates a strongly continuous and analytic semigroup in $\kL( \wh L_{2}(\s)) $ for all $\lambda\in\R$.
We are now in a position to prove Theorem \ref{MT:4} where we exploit the quasilinear principle of linearized stability in \cite[Theorem 1.3]{MW18x}.

\begin{proof}[Proof of Theorem \ref{MT:4}] 
We first address the stability of the zero solution $f=0$ to \eqref{P}.
Assume first that  $\lambda<1$. In this case  all eigenvalues of $\Phi(\lambda,0)$ are negative,
 more precisely $\re z\leq -\sigma b_\mu (1-\lambda)<0$ for all $z\in\sigma(\Phi(\lambda,0))$.
 The quasilinear principle of linearized stability \cite[Theorem 1.3]{MW18x} applied to \eqref{MPSL} yields the first claim of Theorem \ref{MT:4}.
 
In the second case when $\lambda>1$, the intersection $ \sigma(\p_f\Phi(\lambda,0))\cap[\re\lambda>0]$ consists of a finite number of positive
 eigenvalues and we may apply  the instability result in 
\cite[Theorem 1.4]{MW18x} to derive the assertion $(ii)$ in Theorem \ref{MT:4}. 

In the remaining part we discuss the stability properties of small finger-shaped solutions. 
To this end we denote by $\bA_\ell(s)$ the linearized operator
 $$\bA_\ell(s):=\Phi(\lambda_\ell(s),f_\ell(s))+(\p_f \Phi(\lambda_\ell(s),f_\ell(s))[\cdot])[f_\ell(s)]\in\kL(\wh H^3(\s),\wh L_2(\s)),$$
where $\p_f \Phi\in\kL( \wh H^2(\s),\kL(\wh H^3(\s), \wh L_{2}(\s)))$ is the   Fr\'echet derivative  of the mapping $ \Phi $ with respect to the variable $f$.  
We point out that $\bA_\ell(0)=\Phi(\ell^2,0).$

Let us first note that for $\ell\geq2$  the spectrum  $ \sigma(\bA_\ell(0))$ contains a finite number of positive eigenvalues (this number increases with $\ell$).
Since a set consisting of finitely many eigenvalues of $\bA_\ell(s)$ changes continuously with $s\in(-\e_\ell,\e_\ell)$,
cf. \cite[Chapter IV]{Ka95}, we infer from  \cite[Theorem~I.1.3.1~(i)]{Am95} that $-\bA(s)\in\kH(\wh H^3(\s),\wh L_2(\s))$ and
that $\sigma(\bA(s))$ contains only  finitely many  eigenvalues
 with positive real part if $\e_\ell$ is sufficiently small. 
Thus, we may appeal to \cite[Theorem 1.4]{MW18x} to conclude that if $\lambda=\lambda_\ell(s),$ $0<|s|<\e_\ell,$ $\ell\geq 2$,  then $f_\ell(s)$ is an unstable equilibrium to \eqref{P}.

The situation when $\ell=1$ is special, because $\sigma(\bA_1(s))$ has for $s=0$, excepting for the eigenvalue $0$,  only negative eigenvalues. 
We show below that when letting $s$ vary in $(-\e_1,\e_1)$ the operator  $\bA_1(s)$, $0<|s|<1$, has 
a positive eigenvalue $z(s)$ which corresponds to the zero eigenvalue of 
 $ \bA_1(0)$.
To this end we associate to a periodic  function $h$  the function  $\check{h}$ defined by
\[
\check{h}(x):=h(-x),\qquad x\in\R.
\]
Observing that  $(\bB(f)[\overline\omega])\!\check{\phantom{a}}=-\bB(\check{f})[\check{\overline\omega}]$ and $(\bA(f)[\overline\omega])\!\check{\phantom{a}}=\bA(\check{f})[\check{\overline\omega}]$, $f\in\wh  H^2(\s)$, $\overline\omega\in \wh L_2(\s)$,  and that 
\[
\overline\omega(\check{f})[\check{h}]=- (\overline\omega(f)[h])\!\check{\phantom{a}},\qquad\text{$f\in \wh H^2(\s),$ $h\in \wh H^3(\s)$,}
\]
cf. Proposition \ref{P:I1},
it follows that the operator  $\Phi$ introduced in \eqref{PLH} satisfies
\begin{align*}
(\Phi(\lambda,f)[h])\!\check{\phantom{a}}=\Phi(\lambda,\check{f})[\check{h}]\qquad\text{for $\lambda\in\R$, $f\in \wh H^2(\s)$, $h\in \wh H^3(\s)$}.
\end{align*}
Hence, letting $\wh L_{2,e}(\s):=\{f\in \wh L_2(\s)\,:\, \text{$f$ is even}\}$ and $\wh H^r_e(\s):=\wh H^r(\s)\cap L_{2,e}(\s),$ $r\geq0$, it follows that
 $\Phi\in C^\omega(\R\times \wh H^2_e(\s), \kL(\wh H^3_e(\s),\wh L_{2,e}(\s))),$
the linearization   $ \bA_1(0)\in\kL(\wh H^3_e(\s), \wh L_{2,e}(\s))$  being the Fourier multiplier 
\[
\bA_1(0)\sum_{k=1}^\infty a_k\cos (kx)=-\sigma b_\mu\sum_{k=1}^\infty (k^3-\lambda k)a_k\cos (kx).
\]
Let $\Psi:\R\times \wh H^3_e(\s)\to\wh L_{2,e}(\s)$ be the real-analytic mapping defined by $\Psi(\lambda,f):=\Phi(\lambda,f)[f]$.
Noticing that $\p_f\Psi(\lambda_1(s),f_1(s))=\bA_1(s)$,  it follows that $0$ is a simple eigenvalue of $\p_f\Psi(1,0)$ and  ${\rm Ker\,}  \p_f\Psi(1,0)={\rm span}\{\cos(x)\}.$
Since additionally $\p_{\lambda f}\Psi(1,0)[\cos (x)]=\sigma b_\mu \cos(x)\not\in\im \p_f\Psi(1,0)$, the principle of exchange of stability, cf. \cite[Theorem 1.16]{CR73},
 together with Remark \ref{R:3} $(iv)$ implies that the zero eigenvalue of $\p_f\Psi(1,0)$ perturbs along the bifurcation curve through $(\lambda_1,f_1)$ into a positive eigenvalue
 $z(s)$ of  $\bA_1(s)$, $0<|s|<\e_1$, 
 and moreover
 \[
\lim_{s\to 0}\frac{-s\lambda'_1(s)}{z(s)}=\frac{1}{\sigma b_\mu}. 
 \]
 Hence,  if $\e_1$ is 
 sufficiently small, the  operator $ \bA_1(s)$, $0<|s|<\e_1$,  has a positive eigenvalue $z(s)$. 
 Moreover, $ \bA_1(s)$ has at most two eigenvalues with positive real part.
 \cite[Theorem 1.4]{MW18x} yields now that if $\lambda=\lambda_1(s),$ $0<|s|<\e_1$,  then $f_1(s)$ is an unstable equilibrium.
\end{proof}

 %%%%%%%%%%%%%%%%%%%%%%%%%%%%%%%%%%%%%%%%%%%%%%%%%%%%%%%%%%%%%%%%%%%
%%%%%%%%%%%%%%%%%%%%%%%%%%%%%%%%%%%%%%%%%%%%%%%%%%%%%%%%%%%%%%%%%%%%
%%%%%%%%%%%%%%%%%%%%%%%%%%%%%%%%%%%%%%%%%%%%%%%%%%%%%%%%%%%%%%%%%%%%
%%%%%%%%%%%%%%%%%%%%%%%%%%%%%%%%%%%%%%%%%%%%%%%%%%%%%%%%%%%%%%%%%%%
%%%%%%%%%%%%%%%%%%%%%%%%%%%%%%%%%%%%%%%%%%%%%%%%%%%%%%%%%%%%%%%%%%%%
%%%%%%%%%%%%%%%%%%%%%%%%%%%%%%%%%%%%%%%%%%%%%%%%%%%%%%%%%%%%%%%%%%%%
\appendix
\section{Some technical results}\label{S:A}
 %%%%%%%%%%%%%%%%%%%%%%%%%%%%%%%%%%%%%%%%%%%%%%%%%%%%%%%%%%%%%%%%%%%
%%%%%%%%%%%%%%%%%%%%%%%%%%%%%%%%%%%%%%%%%%%%%%%%%%%%%%%%%%%%%%%%%%%%
%%%%%%%%%%%%%%%%%%%%%%%%%%%%%%%%%%%%%%%%%%%%%%%%%%%%%%%%%%%%%%%%%%%%
%%%%%%%%%%%%%%%%%%%%%%%%%%%%%%%%%%%%%%%%%%%%%%%%%%%%%%%%%%%%%%%%%%%
%%%%%%%%%%%%%%%%%%%%%%%%%%%%%%%%%%%%%%%%%%%%%%%%%%%%%%%%%%%%%%%%%%%%
%%%%%%%%%%%%%%%%%%%%%%%%%%%%%%%%%%%%%%%%%%%%%%%%%%%%%%%%%%%%%%%%%%%%

In Lemma \ref{L:A1} we establish  the boundedness of a family of   multilinear singular integral operators in certain settings that are motivated by the analysis in the previous sections.
The  nonperiodic counterparts   
of the estimates derived below have been obtained  previously in     \cite{M16x, M17x}\footnote{In \cite{M16x, M17x} the operators
  \begin{equation*} 
B_{n,m}(a_1,\ldots, a_{m})[b_1,\ldots,b_n,\overline\omega](x):=\PV\int_{\R}  \frac{\overline\omega(x-s)}{s}\cfrac{\prod_{i=1}^{n}\big(\delta_{[x,s]} b_i /s\big)}{\prod_{i=1}^{m}\big[1+\big(\delta_{[x,s]}  a_i /s\big)^2\big]}\, ds 
\end{equation*}
are considered.
 The functions $a_1,\ldots, a_{m},\, b_1, \ldots, b_n:\R\to\R$    are Lipschitz functions  and $\overline\omega\in L_2(\R).$ 
 It is shown in \cite{M16x, M17x} that these operators  extend to  bounded multilinear operators on certain products of Sobolev spaces on $\R$.}. 

  \begin{lemma}\label{L:A1}
  \begin{itemize}
  \item[$(i)$] Given $m,\, n\in\N$ and   Lipschitz functions $a_1,\ldots, a_{m},\, b_1, \ldots, b_n:\R\to\R $, the singular integral operator 
  $C_{n,m}(a_1,\ldots, a_{m})[b_1,\ldots,b_n,\,\cdot\,]$ defined by
\[
C_{n,m}(a_1,\ldots, a_{m})[b_1,\ldots,b_n,\overline\omega](x):=\PV\int_{-\pi}^\pi  \frac{\overline\omega(x-s)}{s}
\cfrac{\prod_{i=1}^{n}\big(\delta_{[x,s]} b_i /s\big)}{\prod_{i=1}^{m}\big[1+\big(\delta_{[x,s]}  a_i /s\big)^2\big]}\, ds 
\]
satisfies $\|C_{n,m}(a_1,\ldots, a_{m})[b_1,\ldots,b_n,\,\cdot\,]\|_{\kL(L_2(\s),L_2((-\pi,\pi)))}\leq C\prod_{i=1}^{n} \|b_i'\|_{\infty},$ 
with a constant $C$ that  depends only 
on $n,\, m$ and $\max_{i=1,\ldots, m}\|a_i'\|_{\infty}.$

In particular,   $C_{n,m}\in {\rm C}^{1-}((W^1_\infty(\mathbb{S}))^{m},
\mathcal{L}_{n+1}( (W^1_\infty(\mathbb{S}))^{n}\times L_2(\mathbb{S}), L_2(\s))).$\\[-2ex]

\item[$(ii)$] Let $m\in\N$, $1\leq n\in\N$, $r\in(3/2,2)$, and $\tau\in(5/2-r,1).$ Then:\\[-2ex]
 \begin{itemize}
\item[$(ii1)$] Given $a_1,\ldots, a_m\in H^r(\mathbb{S})$ and $ b_1,\ldots, b_n,\, \overline\omega\in {\rm C}^\infty(\s)$,
 there exists a constant $C$ that depends only on $n, $ $m$, $r$, $\tau$, and $\max_{i=1,\ldots, m}\|a_i\|_{H^r(\s)}$ such that
 \begin{align}  
&\|C_{n,m}(a_1,\ldots, a_{m})[b_1,\ldots, b_n,\overline\omega]\|_{L_2(\s)}\leq C\|\overline\omega\|_{H^\tau(\s)}\|b_1\|_{H^1(\s)}\prod_{i=2}^n\|b_i\|_{H^r(\s)} \label{REF1}
\end{align}
and
\begin{equation}\label{REF2}
\begin{aligned} 
&\|C_{n,m}(a_1,\ldots, a_{m})[b_1,\ldots, b_n, \overline\omega]-C_{n-1,m}(a_1,\ldots, a_{m})[b_2,\ldots, b_n,b_1' \overline\omega]\|_{L_2(\s)}\\[1ex]
&\hspace{2.6cm}\leq C\|b_1\|_{H^{\tau}(\s)}\| \overline\omega\|_{H^1(\s)}\prod_{i=2}^n\|b_i\|_{H^r(\s)}.
\end{aligned}
 \end{equation}
 In particular,   $C_{n,m}(a_{1}, \ldots, a_{m})$ has an extension in  
$$  \mathcal{L}_{n+1}(H^1(\mathbb{S})\times  (H^r(\mathbb{S}))^{n-1}\times H^{\tau}(\mathbb{S}), L_2(\s)).$$
  \item[$(ii2)$]  $C_{n,m}\in {\rm C}^{1-}((H^r(\mathbb{S}))^m,\mathcal{L}_{n+1}(H^1(\mathbb{S})\times 
  (H^{r}(\mathbb{S}))^{n-1}\times H^{\tau}(\mathbb{S}), L_2(\s))).$\\[-2ex]
 \end{itemize} 
 \item[$(iii)$] Let $m,\, n\in\N$, $r\in(3/2,2)$, and $\tau\in(1/2,1).$ Then:\\[-2ex]
 \begin{itemize}
\item[$(iii1)$] Given $a_1,\ldots, a_m\in H^r(\mathbb{S})$ and $ b_1,\ldots, b_n,\, \overline\omega\in {\rm C}^\infty(\s),$
 there exists a constant $C$ that depends only on $n, $ $m$, $r$, $\tau$, and $\max_{i=1,\ldots, m}\|a_i\|_{H^r(\s)}$ such that
 \begin{align} 
&\|C_{n,m}(a_1,\ldots, a_{m})[b_1,\ldots, b_n,\overline\omega]\|_{\infty}\leq C\|\overline\omega\|_{H^\tau(\s)} \prod_{i=1}^n\|b_i\|_{H^r(\s)}. \label{REF3}
\end{align}
 In particular,   $C_{n,m}(a_{1}, \ldots, a_{m})$ has an extension in
$  \mathcal{L}_{n+1}(  (H^r(\mathbb{S}))^{n}\times H^{\tau}(\mathbb{S}), L_\infty(\s)).$\\[-2ex]
  \item[$(iii2)$]  $C_{n,m}\in {\rm C}^{1-}((H^r(\mathbb{S}))^m,\mathcal{L}_{n+1}( (H^{r}(\mathbb{S}))^{n}\times H^{\tau}(\mathbb{S}), L_\infty(\s))).$
 \end{itemize}
  \end{itemize}
\end{lemma}

\begin{proof}
We first address $(i)$. 
To this end we    fix  $\varphi\in {\rm C}^\infty_0(\R,[0,1])$  with $\varphi=1$ for $|x|\leq 2\pi$ and $\varphi=0$
 for $|x|\geq 4\pi$.
 Then, it is easy to see that 
\begin{align}\label{LP}
 \text{$\|\overline\omega\|_{L_2(\s)}\leq \|\overline\omega\varphi\|_{L_2(\R)}\leq 4\|\overline\omega\|_{L_2(\s)}$\qquad for all $\overline\omega\in L_2(\s).$}
\end{align}
For $|x|<\pi$ we have
 \begin{align*}
  C_{n,m}(a_1,\ldots, a_{m})[b_1,\ldots,b_n,\overline\omega](x)&=B_{n,m}(a_1,\ldots, a_{m})[b_1,\ldots,b_n,\varphi\overline\omega](x)\\[1ex]
 &\hspace{0.424cm} -\int_{\pi<|s|<5\pi}\frac{(\varphi\overline\omega)(x-s)}{s}\cfrac{\prod_{i=1}^{n}\big(\delta_{[x,s]} b_i /s\big)}{\prod_{i=1}^{m}\big[1+\big(\delta_{[x,s]}  a_i /s\big)^2\big]}\, ds,
 \end{align*}
and it follows from \cite[Lemma 3.1]{M17x} and \eqref{LP}  that  
\begin{align*}
 \|B_{n,m}(a_1,\ldots, a_{m})[b_1,\ldots,b_n,\varphi\overline\omega]\|_{L_2((-\pi,\pi))}&\leq C\|\varphi\overline\omega\|_{L_2(\R)}\prod_{i=1}^{n} \|b_i'\|_{\infty}\leq C\|\overline\omega\|_{L_2(\s)}\prod_{i=1}^{n} \|b_i'\|_{\infty}.
\end{align*}
Moreover, it holds that
\begin{align*}
 \Big\|\int_{\pi<|s|<5\pi}\frac{(\varphi\overline\omega)(\,\cdot\,-s)}{s}\cfrac{\prod_{i=1}^{n}\big(\delta_{[\,\cdot\,,s]} b_i /s\big)}{\prod_{i=1}^{m}\big[1+\big(\delta_{[\,\cdot\,,s]}  a_i /s\big)^2\big]}\, ds\Big\|_{\infty}
 &\leq C\|\overline\omega\|_{L_2(\s)}\prod_{i=1}^{n} \|b_i'\|_{\infty}.
\end{align*}
Herewith we established  the  estimate stated  at $(i)$. If $a_1,\ldots, a_{m},\, b_1, \ldots, b_n  $ are $ 2\pi$-periodic, then so is also the function 
$C_{n,m}(a_1,\ldots, a_{m})[b_1,\ldots, b_n, \overline\omega]$, and the local Lipschitz continuity property of $C_{n,m}$ follows directly from the estimate.\medskip

In order to prove $(ii)$ we start by noticing that for $h\in {\rm C}^\infty(\s)$ it holds that 
\begin{align*}
\frac{\p}{\p s}\Big(\frac{\delta_{[x,s]}h}{s}\Big)=\frac{h'(x-s)}{s}-\frac{\delta_{[x,s]}h}{s^2}=-\frac{\delta_{[x,s]}h-sh'(x-s)}{s^2}\qquad \text{for $x\in\R$, $s\neq 0$.}
\end{align*}
Using this relation we get
\begin{align*}
  C_{n,m}(a_1,\ldots, a_{m})[b_1,\ldots, b_n,\overline\omega](x)&=\PV\int_{-\pi}^\pi\frac{\delta_{[x,s]} b_1}{s^2}\cfrac{\prod_{i=2}^{n}\big(\delta_{[x,s]} b_i /s\big)}{\prod_{i=1}^{m}\big[1+\big(\delta_{[x,s]}  a_i /s\big)^2\big]} \overline\omega(x-s)\, ds\\[1ex]
  &=C_{n-1,m}(a_1,\ldots, a_{m})[b_2,\ldots, b_n,b_1'\overline\omega](x)\\[1ex]
  &\hspace{0.424cm}-\PV\int_{-\pi}^\pi\frac{\p}{\p s}\Big(\frac{\delta_{[x,s]}b_1}{s}\Big)
  \cfrac{\prod_{i=2}^{n}\big(\delta_{[x,s]} b_i /s\big)}{\prod_{i=1}^{m}\big[1+\big(\delta_{[x,s]}  a_i /s\big)^2\big]} \overline\omega(x-s)\, ds,
\end{align*}
and the estimate established at $(i)$ yields
\begin{align}\label{KE1}
\|C_{n-1,m}(a_1,\ldots, a_{m})[b_2,\ldots, b_n,b_1'\overline\omega]\|_{L_2(\s)}\leq C \|\overline\omega\|_{H^\tau(\s)}\|b_1\|_{H^1(\s)}\prod_{i=2}^n\|b_i\|_{H^r(\s)}.
\end{align}
We are left  with the singular integral term 
\begin{align}
  &\hspace{-1cm} \PV\int_{-\pi}^\pi\frac{\p}{\p s}\Big(\frac{\delta_{[x,s]}b_1}{s}\nonumber\Big)
  \cfrac{\prod_{i=2}^{n}\big(\delta_{[x,s]} b_i /s\big)}{\prod_{i=1}^{m}\big[1+\big(\delta_{[x,s]}  a_i /s\big)^2\big]} \overline\omega(x-s)\, ds\\[1ex]
  &=(1-(-1)^n)\cfrac{\prod_{i=1}^{n}\big(\delta_{[x,\pi]} b_i /\pi\big)}{\prod_{i=1}^{m}\big[1+\big(\delta_{[x,\pi]}  a_i /\pi\big)^2\big]} \overline\omega(x-\pi)\nonumber\\[1ex]
  &\hspace{0.424cm}+(b_1C_{n-1,m}(a_1,\ldots, a_{m})[b_2,\ldots, b_n, \overline\omega'](x)-C_{n-1,m}(a_1,\ldots, a_{m})[b_2,\ldots, b_n, b_1\overline\omega'])(x)\nonumber\\[1ex]
  &\hspace{0.424cm}+\sum_{j=2}^n\int_{-\pi}^\pi  K_{1,j}(x,s)  \overline\omega(x-s)\, ds-2\sum_{j=1}^m\int_{-\pi}^\pi K_{2,j}(x,s) \   \overline\omega(x-s)\, ds,\label{BGG}
\end{align}
where
\begin{align*}
 K_{1,j}(x,s)&:= \cfrac{\prod_{i=1, i\neq j }^{n}\big(\delta_{[x,s]} b_i /s\big)}{\prod_{i=1}^{m}\big[1+\big(\delta_{[x,s]}  a_i /s\big)^2\big]}\frac{\delta_{[x,s]}b_j-sb_j'(x-s)}{s^2},\\[1ex]
 K_{2,j}(x,s)&:= \cfrac{\prod_{i=1 }^{n}\big(\delta_{[x,s]} b_i /s\big)}{\prod_{i=1}^{m}\big[1+\big(\delta_{[x,s]}  a_i /s\big)^2\big]}
 \cfrac{\delta_{[x,s]}a_j/s}{1+\big(\delta_{[x,s]}  a_j /s\big)^2} \frac{\delta_{[x,s]}a_j-sa_j'(x-s)}{s^2} 
\end{align*}
  for $x\in\R$ and $s\neq0$. 
  The relation \eqref{BGG} is obtained by using integration by parts.
  We next estimate the terms on the right hand side of \eqref{BGG} separately.
  Firstly, it is easy to see that 
   \begin{align}\label{KE2}
\Big\|\cfrac{\prod_{i=1}^{n}\big(\delta_{[\,\cdot\,,\pi]} b_i /\pi\big)}{\prod_{i=1}^{m}\big[1+\big(\delta_{[\,\cdot\,,\pi]}  a_i /\pi\big)^2\big]}
 \overline\omega(\,\cdot\,-\pi)\Big\|_{L_2(\s)}\leq C \|\overline\omega\|_{L_2(\s)}\|b_1\|_{\infty}\prod_{i=2}^n\|b_i\|_{H^r(\s)}.
\end{align}
Secondly, concerning the last two terms in \eqref{BGG}, we may adapt the arguments from the nonperiodic case \cite[Lemma 3.2]{M17x}, to arrive at  
\begin{equation}\label{KE3}
\begin{aligned}
&\Big(\int_{-\pi}^\pi\Big|\int_{-\pi}^\pi  K_{1,j}(x,s)  \overline\omega(x-s)\, ds\Big|^2dx\Big)^{1/2}\leq C\|\overline\omega\|_\infty\|b_1\|_{H^\tau(\s)}\prod_{i=2}^n\|b_i\|_{H^r(\s)},\quad 2\leq j\leq n,\\[1ex]
&\Big(\int_{-\pi}^\pi\Big|\int_{-\pi}^\pi  K_{2,j}(x,s)  \overline\omega(x-s)\, ds\Big|^2dx\Big)^{1/2}\leq C\|\overline\omega\|_\infty\|b_1\|_{H^\tau(\s)}\prod_{i=2}^n\|b_i\|_{H^r(\s)},\quad 1\leq j\leq m.
\end{aligned}
\end{equation} 
Indeed, since $H^\tau(\s)\hookrightarrow {\rm C}^{\tau-1/2}(\s)$, we obtain after appealing to Minkowski's inequality that\footnote{Recall that $\tau_s $ stands for the right translation.
 Moreover, $\wh h(k)$, $k\in\Z$, is the $k$-th Fourier coefficient of   $h\in L_1(\s)$.} 
\begin{align*}
 &\hspace{-1cm}\Big(\int_{-\pi}^\pi\Big|\int_{-\pi}^\pi  K_{1,j}(x,s)  \overline\omega(x-s)\, ds\Big|^2dx\Big)^{1/2}\leq  \int_{-\pi}^\pi \Big(\int_{-\pi}^\pi  |K_{1,j}(x,s)  \overline\omega(x-s)|^2\, dx\Big)^{1/2}ds\\[1ex]
& \leq C\|\overline\omega\|_\infty\|b_1\|_{H^\tau(\s)}\Big(\prod_{i=2,i\neq j}^n\|b_i'\|_{\infty}\Big) \int_{-\pi}^\pi s^{\tau-7/2} \Big(\int_{-\pi}^\pi  |b_j-\tau_s b_j-s\tau_s b_j'|^2(x)\, dx\Big)^{1/2}ds,
\end{align*}
where, taking into account that $|e^{ix}-1-ix|\leq 2|x|^r$ for all $x\in\R,$ we have
\begin{align*}
 \int_{-\pi}^\pi  |b_j-\tau_s b_j-s\tau_s b_j'|^2(x)\, dx&=\sum_{k\in\Z}|\wh b_j(k)|^2|e^{iks}-1-iks|^2\leq C|s|^{2r}\sum_{k\in\Z}|\wh b_j(k)|^2(1+k^2)^r\\[1ex]
 &= C|s|^{2r}\|b_j\|_{H^r(\s)}^2.
\end{align*}
Since $r+\tau-7/2>-1$, the estimate $\eqref{KE3}_1$  follows immediately (similarly for $\eqref{KE3}_2.$)

Thirdly, for the remaining  term 
$$T:=b_1C_{n-1,m}(a_1,\ldots, a_{m})[b_2,\ldots, b_n, \overline\omega']-C_{n-1,m}(a_1,\ldots, a_{m})[b_2,\ldots, b_n, b_1\overline\omega']$$ in \eqref{BGG} we obtain, in virtue of  $(i)$, that 
\begin{align}
\|T\|_{L_2(\s)}\leq C\|\overline\omega\|_{H^1(\s)}\|b_1\|_{\infty}\prod_{i=2}^n\|b_i\|_{H^r(\s)},\label{KE4}
\end{align}
and \eqref{REF2} follows from \eqref{KE2}, \eqref{KE3}, and \eqref{KE4}.

In order to derive \eqref{REF1}, we use the identity $\p(\delta_{[x,s]}\overline\omega)/\p s=\overline\omega'(x-s)$ and integration by parts to recast $T$ as
\begin{align}
T(x)&=(1-(-1)^n)\cfrac{(\delta_{[x,\pi]}\overline\omega  )\prod_{i=1}^{n}\big(\delta_{[x,\pi]} b_i /\pi\big)}{\prod_{i=1}^{m}\big[1+\big(\delta_{[x,\pi]}  a_i /\pi\big)^2\big]} 
+\sum_{j=1}^n\int_{-\pi}^\pi  K_{1,j}(x,s)  \delta_{[x,s]}\overline\omega\, ds\nonumber\\[1ex]
&\hspace{0.424cm}-2\sum_{j=1}^m\int_{-\pi}^\pi K_{2,j}(x,s) \delta_{[x,s]}\overline\omega\, ds,\label{FTT}
\end{align}
with
  \begin{align}\label{KE5}
\Big\|\cfrac{(\delta_{[\,\cdot\,,\pi]}\overline\omega  )\prod_{i=1}^{n}\big(\delta_{[\,\cdot\,,\pi]} b_i /\pi\big)}{\prod_{i=1}^{m}\big[1+\big(\delta_{[\,\cdot\,,\pi]}  a_i /\pi\big)^2\big]}
\Big\|_{L_2(\s)}\leq C \|\overline\omega\|_{L_2(\s)}\|b_1\|_{\infty}\prod_{i=2}^n\|b_i\|_{H^r(\s)}.
\end{align}
Concerning the integral terms in the last sum in \eqref{FTT}, the embedding $H^{1}(\s)\hookrightarrow {\rm C}^{r-3/2}(\s)$ together with Minkowski's inequality yields
\begin{align}
 &\hspace{-1cm}\Big(\int_{-\pi}^\pi\Big|\int_{-\pi}^\pi  K_{2,j}(x,s) \delta_{[x,s]}\overline\omega\, ds\Big|^2dx\Big)^{1/2}\leq  \int_{-\pi}^\pi \Big(\int_{-\pi}^\pi  |K_{2,j}(x,s)  \delta_{[x,s]}\overline\omega|^2\, dx\Big)^{1/2}ds\nonumber\\[1ex]
& \leq C \|b_1\|_{H^{1}(\s)}\Big(\prod_{i=2,i\neq j}^n\|b_i'\|_{\infty} \Big)\int_{-\pi}^\pi s^{r-7/2} \Big(\int_{-\pi}^\pi  |\overline\omega-\tau_s \overline\omega|^2(x)\, dx\Big)^{1/2}ds\nonumber\\[1ex]
 &=C \|b_1\|_{H^{1}(\s)}\Big(\prod_{i=2,i\neq j}^n\|b_i'\|_{\infty}\Big) \int_{-\pi}^\pi s^{r-7/2} \Big(\sum_{k\in\Z}|\wh{\overline\omega}(k)|^2|e^{iks}-1|^2\Big)^{1/2}ds\nonumber
 \end{align}
\begin{align}
  &\leq C \|b_1\|_{H^{1}(\s)}\Big(\prod_{i=2,i\neq j}^n\|b_i'\|_{\infty} \Big)\int_{-\pi}^\pi s^{r+\tau -7/2} \Big(\sum_{k\in\Z}|\wh{\overline\omega}(k)|^2(1+k^2)^\tau\Big)^{1/2}ds\nonumber\\[1ex]
 &=C\|\overline\omega\|_{H^\tau(\s)}\|b_1\|_{H^1(\s)}\prod_{i=2}^n\|b_i\|_{H^r(\s)},\qquad 1\leq j\leq m,\label{KE6}
\end{align}
where we have used the relation $|e^{ix}-1|\leq C|x|^{\tau}$, $x\in\R,$ when deriving the fourth line. 

Similarly, we find for $2\leq j\leq n$ that 
\begin{align}
\Big(\int_{-\pi}^\pi\Big|\int_{-\pi}^\pi  K_{1,j}(x,s) \delta_{[x,s]}\overline\omega\, ds\Big|^2dx\Big)^{1/2}\leq C\|\overline\omega\|_{H^\tau(\s)}\|b_1\|_{H^1(\s)}\prod_{i=2}^n\|b_i\|_{H^r(\s)}.\label{KE7}
\end{align}
In the special case when $j=1$, we  use the procedure which led to \eqref{KE6} together with   $(i)$ to conclude that 
 \begin{align}
\Big(\int_{-\pi}^\pi\Big|\int_{-\pi}^\pi  K_{1,1}(x,s) \delta_{[x,s]}\overline\omega\, ds\Big|^2dx\Big)^{1/2}&\leq C\|\overline\omega\|_{H^\tau(\s)}\|b_1\|_{H^1(\s)}\prod_{i=2}^n\|b_i\|_{H^r(\s)}\nonumber\\[1ex]
&\hspace{0.424cm}+\|\overline\omega C_{n-1,m}(a_1,\ldots,a_m)[b_2,\ldots, b_n,b_1']\|_2\nonumber\\[1ex]
&\hspace{0.424cm}+\| C_{n-1,m}(a_1,\ldots,a_m)[b_2,\ldots, b_n,\overline\omega b_1']\|_2\nonumber\\[1ex]
&\leq C\|\overline\omega\|_{H^\tau(\s)}\|b_1\|_{H^1(\s)}\prod_{i=2}^n\|b_i\|_{H^r(\s)}.\label{KE8}
\end{align}
The property \eqref{REF1} follows now from \eqref{KE1}, \eqref{KE2}, \eqref{KE3}, and \eqref{KE5}-\eqref{KE8}.
The extension property left at $(ii1)$ follows from \eqref{REF1}. 
The claim $(ii2)$ is a straight forward consequence of \eqref{REF1}.\bigskip

With respect to $(iii)$ we decompose
\[
C_{n,m}(a_1,\ldots,a_m)[b_1,\ldots, b_n,\overline\omega]=\overline\omega A-B,
\]
with
\[
A(x):=\PV\int_{-\pi}^\pi\frac{1}{s}\frac{\prod_{i=1}^{n}\big(\delta_{[x,s]} b_i /s\big)}{\prod_{i=1}^{m}\big[1+\big(\delta_{[x,s]}  a_i /s\big)^2\big]}\, ds\quad\text{and}\quad
B(x):=\int_{-\pi}^\pi\frac{(\delta_{[x,s]}\overline\omega/s)\prod_{i=1}^{n}\big(\delta_{[x,s]} b_i /s\big)}{\prod_{i=1}^{m}\big[1+\big(\delta_{[x,s]}  a_i /s\big)^2\big]}\, ds.
\]
Since $\tau>1/2$ and $H^{\tau}(\s)\hookrightarrow {\rm C}^{\tau-1/2}(\s)$, it holds
\begin{align}
 \|B\|_\infty \leq C\|\overline\omega\|_{H^\tau(\s)}\prod_{i=1}^n\|b_i\|_{H^r(\s)}\int_{-\pi}^\pi |s|^{\tau-3/2}\, ds\leq C\|\overline\omega\|_{H^\tau(\s)}\prod_{i=1}^n\|b_i\|_{H^r(\s)},\label{BUB1}
\end{align}
and we are left with the function $A$. Taking advantage of the embedding $H^r(\s)\hookrightarrow {\rm C}^{r-1/2}(\s),$ the arguments  in the proof of \cite[Lemma 3.1]{M16x} show that indeed
\begin{align}
 \|A\|_\infty \leq C \prod_{i=1}^n\|b_i\|_{H^r(\s)}.\label{BUB2}
\end{align}
The estimates \eqref{BUB1}-\eqref{BUB2} lead us to the estimate \eqref{REF3}.
The last two claims follow directly from \eqref{REF3} and the proof is complete.
\end{proof}

\bibliographystyle{siam}
\bibliography{B55}
\end{document}